\documentclass[11pt,oneside,reqno]{amsart}  

\usepackage[utf8]{inputenc}            
\usepackage[T1]{fontenc}
\usepackage{microtype}

\usepackage[english]{babel}

\usepackage{url}
\usepackage{xcolor}                   
\usepackage{graphics,graphicx,xypic}
\usepackage{dsfont,txfonts}
\usepackage{mathrsfs}
\usepackage{textcomp}
\usepackage{enumerate} 
\usepackage{hyperref}
\usepackage[toc,page]{appendix}         

\usepackage[a4paper,scale={0.72,0.74},marginratio={1:1},footskip=7mm,headsep=10mm]{geometry}

\numberwithin{equation}{section}    

\newtheorem{theorem}{Theorem}[section]

\newtheorem{proposition}[theorem]{Proposition}
\newtheorem{corollary}[theorem]{Corollary}
\newtheorem{lemma}[theorem]{Lemma}

\theoremstyle{definition}
\newtheorem{definition}[theorem]{Definition}
\newtheorem{remark}[theorem]{Remark}

\newcommand{\absv}{\hspace*{0.1cm}\bigg|\hspace*{0.1cm}}

\renewcommand{\P}{\mathrm{P}}
\newcommand{\E}{\mathrm{E}}

\newcommand{\hh}{\hat{h}}
\newcommand{\hb}{\hat{\beta}}
\newcommand{\hcc}{\mathbf{h}_c^\alpha}
\newcommand{\fhc}{\mathbf{F}^\alpha}

\newcommand{\hd}{\hat{\delta}}

\newcommand{\tolaw}{\overset{(\mathrm{d})}{\longrightarrow}}

\newcommand{\btau}{\boldsymbol{\tau}^{\alpha}}

\newcommand{\Z}{\mathbf{Z}_{\hb,\hh}^{W,\mathtt{c}}}

\newcommand{\Zrho}{\mathbf{Z}_{\hb,\hh'}^{W,\mathtt{c}}}
\newcommand{\Zo}{\mathbf{Z}_{\hb,0}^{W,\mathtt{c}}}

\newcommand{\Zf}{\mathbf{Z}_{\hb,\hh}^{W}}
\newcommand{\Zcf}{\mathbf{Z}_{c^{\alpha-\frac{1}{2}}\hb,c^\alpha\hh}^{W}}
\newcommand{\Zfo}{\mathbf{Z}_{\hb,0}^{W}}

\newcommand{\Y}{\mathbf{Y}_{\hb}^{W,\mathtt{c}}}
\newcommand{\Yf}{\mathbf{Y}_{\hb}^{W}}

\newcommand{\tZ}{\mathtt{Z}}
\newcommand{\btZ}{\boldsymbol{\mathtt{Z}}}

\newcommand{\bPsi}{\mathbf{\Psi}}

\newcommand{\Zn}{\mathrm{Z}_{\beta_N,h_N}^{\omega,\mathtt{c}}}

\newcommand{\Zno}{\mathrm{Z}_{\beta_N,0}^{\omega,\mathtt{c}}}

\newcommand{\Znf}{\mathrm{Z}_{\beta_N,h_N}^{\omega}}

\newcommand{\Znt}{\mathrm{Z}_{\beta_N,h_N'}^{\omega,\mathtt{c}}}

\newcommand{\dd}{\mathrm{d}}

\newcommand{\sso}{\mathtt{s}}
\newcommand{\tto}{\mathtt{t}}
\newcommand{\mmo}{\mathtt{m}_t}
\newcommand{\JJo}{\mathtt{J}}

\renewcommand{\ss}{\boldsymbol{\mathtt{s}}}
\renewcommand{\tt}{\boldsymbol{\mathtt{t}}}
\newcommand{\mm}{\boldsymbol{\mathtt{m}}_t}
\newcommand{\JJ}{\boldsymbol{\mathtt{J}}}
\newcommand{\cF}{\mathcal{F}}

\newcommand{\rss}{{\mathtt{s}}^{(N)}}
\newcommand{\rtt}{{\mathtt{t}}^{(N)}}
\newcommand{\rmm}{{\mathtt{m}}_t^{(N)}}
\newcommand{\rJJ}{{\mathtt{J}}^{(N)}}

\newcommand{\dss}{\boldsymbol{\mathtt{s}}^{(N)}}
\newcommand{\dtt}{\boldsymbol{\mathtt{t}}^{(N)}}
\newcommand{\dmm}{\boldsymbol{\mathtt{m}}_t^{(N)}}

\newcommand{\gne}{g_{N,\epsilon}}

\newcommand{\sumtwo}[2]{\sum_{\substack{#1 \\ #2}}} 

\newcommand{\RN}{\mathrm{R}}

\title{Universality for the pinning model in the weak coupling regime}

\author{Francesco Caravenna}

\address{Dipartimento di Matematica e Applicazioni, 
Universit\`a degli Studi di Milano-Bicocca,
via Cozzi 55,
20125 Milano, Italy}

\email{francesco.caravenna@unimib.it}

\author{Fabio Toninelli}

\address{Universit\'e de Lyon, CNRS, Institut Camille Jordan,
  Universit\'e Claude Bernard Lyon 1, 43 bd du 11 novembre 1918, 69622 Villeurbanne cedex, France
}

\email{toninelli@math.univ-lyon1.fr}

\author{Niccolò Torri}

\address{Universit\'e de Lyon, Institut Camille Jordan,
  Universit\'e Claude Bernard Lyon 1, 43 bd du 11 novembre 1918, 69622 Villeurbanne cedex, France
 }

\email{torri@math.univ-lyon1.fr}



\thanks{This work was supported by the \emph{Programme Avenir Lyon 
Saint-Etienne de l'Université de Lyon} (ANR-11-IDEX-0007), within 
the program \emph{Investissements d'Avenir} operated by the French
National Research Agency (ANR). F. T. was partially supported by the Marie Curie IEF Action “DMCP- Dimers, Markov chains and Critical Phenomena”, grant agreement n. 621894.
}

\newenvironment{myenumerate}{%
\renewcommand{\theenumi}{\arabic{enumi}}%
\renewcommand{\labelenumi}{{\rm(\theenumi)}}%
\begin{list}{\labelenumi}
	{%
	\setlength{\itemsep}{0.4em}%
	\setlength{\topsep}{0.5em}%
	\setlength\leftmargin{2.45em}%
	\setlength\labelwidth{2.05em}%
	\setlength{\labelsep}{0.4em}%
	\usecounter{enumi}%
	}%
	}%
{\end{list}
}

{\end{list}
}

{\end{myenumerate}}

\newenvironment{myitemize}{%
\begin{list}{$\bullet$}%
 	{%
	\setlength{\itemsep}{0.4em}%
	\setlength{\topsep}{0.5em}%
	\setlength\leftmargin{2.45em}%
	\setlength\labelwidth{2.05em}%
	\setlength{\labelsep}{0.4em}%
	}%
	}%
{\end{list}}

\renewenvironment{itemize}{
\begin{myitemize}}%
{\end{myitemize}}

\newcommand{\N}{\mathbb{N}}
\newcommand{\R}{\mathbb{R}}
\newcommand{\bbP}{\mathbb{P}}
\newcommand{\bbE}{\mathbb{E}}
\newcommand{\ind}{\mathds{1}}

\begin{document}

\begin{abstract}
We consider disordered pinning models,
when the return time distribution of the underlying renewal process
has a polynomial tail with exponent $\alpha \in (\frac{1}{2},1)$.
This corresponds to a regime where disorder is known to be
  \emph{relevant}, i.e. to change the critical exponent of the
  localization transition and to induce a non-trivial shift of the
  critical point.
We show that the free energy and critical curve have an
explicit universal asymptotic behavior in the weak coupling regime,
depending only on the tail of the return time distribution
and not on finer details of the models. 
This is obtained comparing the partition functions 
with corresponding continuum quantities, through coarse-graining
techniques.
\end{abstract}

\maketitle

\section{Introduction and motivation}

Understanding the effect of disorder is a key topic
in statistical mechanics, dating back at least to the seminal work of Harris~\cite{H74}.
For models that are disorder relevant, i.e.\ for which an arbitrary amount of disorder
modifies the critical properties, it was recently shown in \cite{CRZ13} that it is interesting to
look at a suitable \emph{continuum and weak disorder regime},
tuning the disorder strength to zero as the size of the system diverges, which leads to
a \emph{continuum model} in which disorder is still present. This framework includes
many interesting models, including the 2d random field Ising model with site disorder,
the disordered pinning model and the directed polymer in random environment
(which was previously considered by Alberts, Quastel and Khanin
\cite{AKQ14a,AKQ14b}).

Heuristically, a continuum model should capture the properties of a large family
of discrete models, leading to sharp predictions about the scaling behavior of key quantities,
such free energy and critical curve, in the weak disorder regime. The goal of this paper
is to make this statement rigorous in the context of disordered 
pinning models \cite{GB07,GB10,FdH07},
sharpening the available estimates in the literature and proving a form of universality.
Although we stick to pinning models, the main ideas have a general value and should 
be applicable to other models as well.

%

In this section we give a concise description
of our results, focusing on the critical curve.
Our complete results are presented in the next section.
Throughout the paper we use the conventions $\N = \{1,2,3,\ldots\}$ and $\N_0 =
\N \cup \{0\}$, and we write $a_n \sim b_n$ to mean $\lim_{n\to\infty} a_n/b_n = 1$.

\medskip

To build a disordered pinning model,
we take a Markov chain $(S=(S_n)_{n\in\mathbb{N}_0}, \P)$
starting at a distinguished state, called $0$,
and we modify its distribution by rewarding/penalizing each visit to $0$.
The rewards/penalties are determined by a sequence of  i.i.d.\ real random variables 
$(\omega=(\omega_{n})_{n\in\mathbb{N}},\mathbb{P})$, independent of $S$,
called \emph{disorder variables} (or \emph{charges}).
We make the following assumptions.
\begin{itemize}
\item The return time to $0$ of the Markov chain
$\tau_1 := \min\{n\in\N: \ S_n = 0\}$ satisfies
\begin{equation}\label{assRP0}
\P\left( \tau_1 < \infty \right) = 1 , \qquad
K(n):=\P\left(\tau_1 = n \right)\sim \frac{L(n)}{n^{1+\alpha}}, \quad n\to\infty,
\end{equation}
where $\alpha \in (0,\infty)$
and $L(n)$ is a slowly varying function \cite{BGT89}.
For simplicity we assume that $K(n)>0$ for all $n\in\mathbb{N}$,
but periodicity can be easily dealt with (e.g.\ $K(n) > 0$ iff $n \in 2\N$).

\item The disorder variables have locally finite exponential moments:
\begin{equation}\label{assD1}
\exists \beta_0>0 :\
\Lambda(\beta):=\log\mathbb{E}(e^{\beta\omega_1})<\infty,\ \forall \beta \in (-\beta_0,\beta_0),
\qquad \mathbb{E}(\omega_1)=0,\qquad \mathbb{V}(\omega_1)=1 \,,
\end{equation}
where the choice of zero mean and unit variance is just a convenient normalization.

\end{itemize}
Given a $\bbP$-typical realization of the sequence $\omega = (\omega_n)_{n\in\N}$,
the \emph{pinning model} is defined as the following
random probability law $\mathrm{P}_{\beta, h, N}^{\omega}$ on Markov chain paths $S$:
\begin{equation}\label{eq1intro}
\frac{\mathrm{d}\P_{\beta,h, N}^{\omega}}{\mathrm{d}\mathrm{P}}(S)
:=\frac{e^{\sum_{n=1}^N (\beta \omega_n - \Lambda(\beta) + h) \ind_{\{S_n = 0\}}}}
{\mathrm{Z}_{\beta,h}^{\omega}(N)} , \qquad
\mathrm{Z}_{\beta,h}^{\omega}(N) := 
\E \Big[ e^{\sum_{n=1}^N (\beta \omega_n - \Lambda(\beta) + h) \ind_{\{S_n = 0\}}} \Big] \,,
\end{equation}
where $N \in \N$ represents the ``system size'' while
$\beta \ge 0$ and $h\in\R$ tune the disorder strength and bias.
(The factor $\Lambda(\beta)$ in \eqref{eq1intro} is just a translation of $h$,
introduced so that
$\mathbb{E}[e^{\beta\omega_n-\Lambda(\beta)}]=1$.)

\smallskip

Fixing $\beta \ge 0$ and varying $h$, the pinning model undergoes a
localization/delocalization \emph{phase transition} at a critical
value $h_c(\beta) \in \R$: the typical paths $S$ under $\P_{\beta,h,N}^\omega$ 
are localized at $0$ for $h > h_c(\beta)$, while
they are delocalized away from $0$ for $h < h_c(\beta)$
(see \eqref{eq:locdeloc} below for a precise result).

It is known that $h_c(\cdot)$ is a continuous function,
with $h_c(0) = 0$
(note that for $\beta = 0$ the disorder $\omega$ disappears in \eqref{eq1intro}
and one is left with a homogeneous model, which is exactly solvable).
The behavior of $h_c(\beta)$ as $\beta \to 0$
has been investigated in depth \cite{GLT10, AZ09, DGLT09,A08,ChedH},
confirming the so-called \emph{Harris criterion} \cite{H74}: recalling that 
$\alpha$ is the tail exponent in \eqref{assRP0}, it was shown that:
\begin{itemize}
\item for $\alpha < \frac{1}{2}$ one  has $h_c(\beta) \equiv 0$ for $\beta > 0$ 
small enough (irrelevant disorder regime);

\item for $\alpha > \frac{1}{2}$, on the other hand, one has
  $h_c(\beta) > 0$ for all $\beta > 0$.  Moreover, it was proven  \cite{GTsmooth} that
  disorder changes the order of the phase transition: 
  free energy vanishes for $h\downarrow h_c(\beta)$ at least as fast as
  $(h-h_c(\beta))^2$, while for $\beta=0$ the critical exponent is $\max(1/\alpha,1)<2$.
This case is therefore called 
\emph{relevant disorder regime};

\item for $\alpha = \frac{1}{2}$, known as the ``marginal'' case,
the answer depends on the slowly varying function $L(\cdot)$ in \eqref{assRP0}: 
more precisely one has disorder relevance 
if and only if $\sum_n \frac{1}{n\, (L(n))^2}=\infty$, as recently
proved in \cite{BL15}
(see also \cite{A08,GLT10,GLT11} for previous partial results).

\end{itemize}

In the special case $\alpha>1$, when the mean return time $\E[\tau_1]$ is finite,
one has (cf. \cite{BCPSZ14}) 
\begin{equation} \label{eq:alpha>1}
\lim\limits_{\beta\to 0}\,\frac{h_c(\beta)}{\beta^{2}}
= \frac{1}{2\E[\tau_1]} \frac{\alpha}{1+\alpha} \,.
\end{equation}
In this paper we focus on the case $\alpha \in (\frac{1}{2},1)$, where
the mean return time is infinite: $\E[\tau_1]= \infty$.
In this case,
the precise asymptotic behavior of $h_c(\beta)$ as $\beta \to 0$ was known only up
to non-matching constants, cf.\ \cite{AZ09,DGLT09}:
there is a slowly varying function $\tilde{L}_\alpha$ (determined explicitly by $L$ and $\alpha$)
and constants $0 < c < C < \infty$
such that for $\beta>0$ small enough
\begin{equation}\label{eq:known}
c \, \tilde{L}_\alpha\big(\tfrac{1}{\beta}\big) \,
\beta^{\frac{2\alpha}{2\alpha-1}}\leq h_c(\beta)\leq 
C \, \tilde{L}_\alpha \big( \tfrac{1}{\beta} \big) \, \beta^{\frac{2\alpha}{2\alpha-1}} .
\end{equation}
Our key result (Theorem~\ref{Thm2} below)
shows that this relation can be made sharp:
there exists $m_{\alpha} \in (0,\infty)$ such that,
under mild assumptions on the return time and disorder distributions,
\begin{equation}\label{Conj1}
\lim\limits_{\beta\to 0}\,\frac{h_c(\beta)}{\tilde{L}_\alpha(\frac{1}{\beta})
\, \beta^{\frac{2\alpha}{2\alpha-1}}}=m_{\alpha} .
\end{equation}

Let us stress the \emph{universality} value of \eqref{Conj1}:
the asymptotic behavior of $h_c(\beta)$ as $\beta \to 0$ depends
only on the \emph{tail} of the return time distribution $K(n) = \P(\tau_1 = n)$,
through the exponent $\alpha$ and the slowly varying function $L$ appearing in \eqref{assRP0}
(which determine $\tilde L_\alpha$): all finer details of $K(n)$
beyond these key features disappear in the weak disorder regime.
The same holds for the disorder variables:
any admissible distribution for $\omega_1$ has the same effect on the asymptotic behavior
of $h_c(\beta)$.

Unlike \eqref{eq:alpha>1},
we do not know the explicit value of the limiting constant
$m_\alpha$ in \eqref{Conj1}, but we can characterize it as the critical parameter
of the \emph{continuum disordered pinning model} (CDPM)
recently introduced in \cite{CRZ14,CRZ13}.
The core of our approach is a precise quantitative comparison
between discrete pinning models and the CDPM,
or more precisely between the corresponding partition functions,
based on a subtle \emph{coarse-graining} procedure which
extends the one developed in \cite{BdH97,CG10} for the copolymer model.
This extension turns out to be quite subtle, because
unlike the copolymer case 
\emph{the CDPM admits no ``continuum Hamiltonian''}:
although it is built over the $\alpha$-stable regenerative set
(which is the continuum limit of renewal processes satisfying \eqref{assRP0},
see \S\ref{closedsetsec}), its law is \emph{not} absolutely continuous with respect to the law
of the regenerative set, cf.~\cite{CRZ14}. As a consequence, we need to introduce a suitable
coarse-grained Hamiltonian, based on partition functions, which behaves well in the continuum limit.
This extension of the coarse-graining procedure
is of independent interest and should be applicable to other models 
with no ``continuum Hamiltonian'', including the directed
polymer in random environment \cite{AKQ14b}.

Overall, our results reinforce the role of the CDPM as a universal model,
capturing the key properties of discrete pinning models in the weak coupling regime.

\section{Main results}


\subsection{Pinning model revisited}

The disordered pinning model $\P_{\beta,h,N}^\omega$ was defined 
in \eqref{eq1intro} as a perturbation of a Markov chain $S$.
Since the interaction only takes place when $S_n = 0$, it is customary to forget about the full
Markov chain path, focusing only on its zero level set
\begin{equation*}
	\tau = \{n\in\N_0: \ S_n = 0\} ,
\end{equation*}
that we look at as a random subset of $\N_0$.
Denoting by $0 = \tau_0 < \tau_1 < \tau_2 < \ldots$ the points of $\tau$,
we have a \emph{renewal process} $(\tau_k)_{k\in\N_0}$,
i.e.\ the random variables $(\tau_j - \tau_{j-1})_{j\in\N}$
are i.i.d. with values in $\N$. Note that we have the equality $\{S_n = 0\} = \{n\in\tau\}$,
where we use the shorthand
\begin{equation*}
	\{n \in \tau\} := \bigcup_{k\in\N_0} \{ \tau_k = n\}.
\end{equation*}
Consequently, viewing the pinning model $\P_{\beta,h, N}^{\omega}$
as a law for $\tau$, we can rewrite \eqref{eq1intro} as follows:
\begin{equation}\label{model}
\frac{\mathrm{d}\P_{\beta,h, N}^{\omega}}{\mathrm{d}\mathrm{P}}( \tau)
:=\frac{e^{\sum_{n=1}^N (\beta \omega_n - \Lambda(\beta) + h) \ind_{\{n \in \tau\}}}}
{\mathrm{Z}_{\beta,h}^{\omega}(N)} , \qquad
\mathrm{Z}_{\beta,h}^{\omega}(N) := 
\E \Big[ e^{\sum_{n=1}^N (\beta \omega_n - \Lambda(\beta) + h) \ind_{\{n \in \tau\}}} \Big] \,.
\end{equation}
To summarize, henceforth we fix a renewal process $(\tau = (\tau_k)_{k\in\N_0}, \P)$
satisfying \eqref{assRP0} and an i.i.d.\ sequence of
disorder variables $(\omega = (\omega_n)_{n\in\N}, \bbP)$ satisfying \eqref{assD1}.
We then define the disordered pinning model as the random probability
law $\P_{\beta,h, N}^{\omega}$ for $\tau$ defined in \eqref{model}. 


\smallskip

In order to prove our results, we need some additional assumptions.
We recall that for any renewal process satisfying \eqref{assRP0}
with $\alpha\in (0,1)$, the following local
renewal theorem holds \cite{GL63,D97}:
\begin{equation}\label{assRP1}
u(n) := \P(n\in\tau)
\sim  \frac{C_\alpha}{L(n) \, n^{1-\alpha}}, \quad n\to\infty, \qquad
\text{with} \ \ C_\alpha := \frac{\alpha\sin(\alpha\pi)}{\pi} \,.
\end{equation}  
In particular, if $\ell=o(n)$, then
$u(n+\ell)/u(n)\to 1$ as $n\to\infty$. 
We are going to assume that this convergence
takes place at a not too slow rate, i.e.\ at least a power law of $\frac{\ell}{n}$,
as in \cite[eq. (1.7)]{CRZ14}:
\begin{equation}\label{assRP2}
\exists\, C, n_0\in (0,\infty);\, \epsilon,\delta\in (0,1]:\quad \left|\frac{u(n+\ell)}{u(n)}-1\right|
\leq C\left(\frac{\ell}{n}\right)^\delta,\quad \forall n\geq n_0,\, 0\leq \ell\leq \epsilon n.
\end{equation}

\begin{remark}\label{rem:mild}\rm
This is a mild assumption, as discussed in \cite[Appendix B]{CRZ14}.
For instance, one can build a wide family of nearest-neighbor Markov 
chains on $\mathbb{N}_0$ with $\pm 1$ increments (Bessel-like random walks)
satisfying \eqref{assRP0}, cf.\ \cite{A11}, and in this case \eqref{assRP2} 
holds for any $\delta<\alpha$.
\end{remark}

Concerning the disorder distribution, we strengthen the
finite exponential moment assumption \eqref{assD1}, requiring the following
concentration inequality:
\begin{equation}\label{assD2}
\begin{split}
	\exists \gamma>1, C_1, C_2 \in (0,\infty): \text{ for all $n\in\N$ and for all } f: \R^n \to \R
	\text{ convex and 	$1$-Lipschitz} \\
	\bbP \Big( \big|
	f(\omega_1, \ldots, \omega_n) - M_f \big| 
	\ge t\Big) \le 
	C_1 \exp \bigg(-\frac{t^\gamma}{C_2}\bigg) \,, \qquad \quad
\end{split}
\end{equation}
where $1$-Lipschitz means $|f(x)-f(y)| \le |x-y|$ for all $x,y\in\R^n$,
with $|\cdot|$ the usual Euclidean norm, and $M_f$ denotes a median of
$f(\omega_1, \ldots, \omega_n)$. (One can equivalently take $M_f$ to be
the mean $\bbE[f(\omega_1, \ldots, \omega_n)]$
just by changing the constants $C_1, C_2$, cf.\ \cite[Proposition 1.8]{L05}.)

It is known that
\eqref{assD2} holds under fairly general assumptions, namely:
\begin{itemize}
\item ($\gamma=2$) \
if $\omega_1$ is bounded, i.e.\ $\bbP(|\omega_1| \le a) = 1$ for some $a \in (0,\infty)$,
cf.\ \cite[Corollary 4.10]{L05};

\item ($\gamma=2$) \ if the law of $\omega_1$ satisfies a log-Sobolev
inequality, in particular if $\omega_1$
is Gaussian, cf.\ \cite[Theorems~5.3 and Corollary~5.7]{L05}; more generally, if 
the law of $\omega_1$ is absolutely continuous with density $\exp(-U - V)$, where $U$ is uniformly
strictly convex (i.e. $U(x) - c x^2$ is convex, for some $c>0$) and $V$ is bounded,
cf.\ \cite[Theorems 5.2 and Proposition 5.5]{L05};

\item ($\gamma \in (1,2)$) \ if the law of $\omega_1$ is absolutely continuous with density 
given by
$c_\gamma \, e^{-|x|^\gamma}$
(see Propositions~4.18 and~4.19 in \cite{L05} and the following considerations).
\end{itemize}

%

\subsection{Free energy and critical curve}

The normalization constant $\mathrm{Z}_{\beta,h}^{\omega}(N)$ in \eqref{model} is called 
\emph{partition function} and plays a key role.
Its rate of exponential growth as $N \to \infty$
is called \emph{free energy}:
\begin{equation}\label{DFE}
\mathrm{F}(\beta,h)
:=\lim_{N\to\infty}\frac{1}{N} \log\mathrm{Z}_{\beta,h}^{\omega}(N)
=\lim_{N\to\infty}\frac{1}{N}\mathbb{E} \Big[ \log\mathrm{Z}_{\beta,h}^{\omega}(N)\Big],
\qquad \bbP\text{-a.s. \ and in \ } L^1,
\end{equation}
where the limit exists and is finite by super-additive arguments \cite{GB07,FdH07}. 
Let us stress that $\mathrm{F}(\beta,h)$ 
depends on the laws of the renewal process $\P(\tau_1 = n)$
and of the disorder variables $\bbP(\omega_1 \in \dd x)$, but it does not depend on
the $\bbP$-typical realization of the sequence $(\omega_n)_{n\in\N}$.
Also note that $h \mapsto \mathrm{F}(\beta,h)$ 
inherits from $h \mapsto \log\mathrm{Z}_{\beta,h}^{\omega}(N)$ the properties of being
convex and non-decreasing.

Restricting the expectation defining $\mathrm{Z}_{\beta,h}^{\omega}(N)$ 
to the event $\{\tau_1>N\}$ and recalling
the polynomial tail assumption \eqref{assRP0}, one obtains the basic but crucial
inequality
\begin{equation}\label{pinn3}
\mathrm{F}(\beta,h)\geq 0\quad \forall \beta \ge 0, h\in\mathbb{R}.
\end{equation}
One then defines the \emph{critical curve} by
\begin{equation}\label{eq:hc}
	h_c(\beta) := \sup\{h \in \R : \ \mathrm{F}(\beta,h)=0\} .
\end{equation}
It can be shown that $0 < h_c(\beta) < \infty$ for $\beta > 0$,
and by monotonicity and continuity in $h$ one has
\begin{equation}\label{eq:><}
	\mathrm{F}(\beta,h) = 0 \ \ \text{if} \ \ h \le h_c(\beta), \qquad
	\mathrm{F}(\beta,h) > 0 \ \ \text{if} \ \ h > h_c(\beta) .
\end{equation}
In particular, the
function $h \mapsto \mathrm{F}(\beta,h)$ is non-analytic at the point $h_c(\beta)$,
which is called a \emph{phase transition} point.
A probabilistic interpretation can be given looking at the quantity
\begin{equation} \label{eq:ellN}
	\ell_N := \sum_{n=1}^N \ind_{\{n\in\tau\}}
	= \big| \tau \cap (0, N] \big|	 ,
\end{equation}
which represents the number of points of 
$\tau \cap (0,N]$. By convexity, $h \mapsto \mathrm{F}(\beta,h)$
is differentiable at all but a countable number of points, and for pinning models
it can be shown that it is actually $C^\infty$ for $h \ne h_c(\beta)$
 \cite{alea}.
Interchanging differentiation and limit in \eqref{DFE}, by convexity,
relation \eqref{model} yields
\begin{equation} \label{eq:locdeloc}
	\text{for $\bbP$-a.e. $\omega$}, \qquad 
	\lim_{N\to\infty}  \E_{\beta,h, N}^{\omega} \bigg[
	\frac{\ell_N}{N} \bigg] = \frac{\partial \mathrm{F}(\beta,h)}{\partial h} \
	\begin{cases}
	= 0 & \text{if } h < h_c(\beta) \\
	> 0 & \text{if } h > h_c(\beta) 
	\end{cases} \,.
\end{equation}
This shows that the typical paths of the pinning model
are indeed localized at $0$ for $h > h_c(\beta)$
and delocalized away from $0$ for $h < h_c(\beta)$.\footnote{Note that,
in Markov chain terms, $\ell_N$ is the number 
of visits of $S$ to the state $0$, up to time $N$.}
We refer to \cite{GB07,GB10,FdH07} for details and for finer results.

\subsection{Main results}

Our goal is to study the asymptotic behavior of the
free energy $\mathrm{F}(\beta,h)$ and critical curve
$h_c(\beta)$ in the weak coupling regime $\beta, h \to 0$.

Let us recall the recent results in \cite{CRZ13,CRZ14},
which are the starting point of our analysis.
Consider any disordered pinning model where the renewal process satisfies
\eqref{assRP0}, with $\alpha \in (\frac{1}{2},1)$,
and the disorder satisfies \eqref{assD1}.
If we let $N \to \infty$ and simultaneously
$\beta \to 0$, $h \to 0$ as follows:
\begin{align}\label{eqParam}
	\beta = \beta_N := \hb \frac{L(N)}{N^{\alpha - \frac{1}{2}}},
	\qquad 
	h = h_N := \hh \frac{L(N)}{N^{\alpha}} , \qquad
	\text{for fixed } \hb > 0, \ \hh\in \R \,,
\end{align} 
the family of partition functions $\mathrm{Z}_{\beta_N,h_N}^{\omega}(Nt )$, with $t \in [0,\infty)$,
has a universal limit,
in the sense of finite-dimensional distributions \cite[Theorem 3.1]{CRZ13}:
\begin{equation} \label{eq:convd}
\Big( \Znf(Nt) \Big)_{t \in [0,\infty)} \tolaw
\Big( \Zf(t) \Big)_{t \in [0,\infty)}, \quad N\to\infty .
\end{equation}
The \emph{continuum partition function} $\Zf(t)$
depends only on the exponent
$\alpha$ 
and on a Brownian motion $(W = (W_t)_{t\ge 0}, \bbP)$,
playing the role of continuum disorder. We point out that $\Zf(t)$ has an
explicit Wiener chaos representation, as a series of deterministic and stochastic
integrals (see \eqref{eq:Wiener} below), and admits a version which is continuous
in $t$, that we fix henceforth (see \S\ref{sec:further} for more details).


\begin{remark}\rm\label{rem:hb}
For an intuitive explanation of why $\beta_N, h_N$ should scale as in \eqref{eqParam},
we refer to the discussion following Theorem~1.3 in \cite{CRZ14}.
Alternatively, one can invert the relations in \eqref{eqParam},
for simplicity in the case $\hb = 1$,
expressing $N$ and $h$ as a function of $\beta$ as follows:
\begin{equation} \label{eq:hb}
	\frac{1}{N} \sim \tilde L_\alpha(\tfrac{1}{\beta})^2 \,
	\beta^{\frac{2}{2\alpha-1}}, 
	\qquad
	h \sim \hh \, \tilde L_\alpha(\tfrac{1}{\beta}) \, \beta^{\frac{2\alpha}{2\alpha-1}} \,,
\end{equation}
where $\tilde L_\alpha$ is the same slowly varying function appearing in \eqref{eq:known},
determined explicitly by $L$ and $\alpha$.
Thus $h = h_N$ is of the same order as the critical curve
$h_c(\beta_N)$, which is quite a natural choice.

More precisely, one has $\tilde L_\alpha(x) = M^\#(x)^{-\frac{1}{2\alpha-1}}$,
where $M^\#$ is the \emph{de Bruijn conjugate} of 
the slowly varying function 
$M(x) := 1 / L(x^{\frac{2}{2\alpha-1}})$, cf.\ \cite[Theorem 1.5.13]{BGT89},
defined by the asymptotic property
$M^\#(x M(x)) \sim 1/M(x)$.
We refer to (3.17) in \cite{CRZ13} and the following lines for more details.

\end{remark}

It is natural to define a \emph{continuum free energy} ${\fhc}(\hb,\hh)$
in terms of $\Zf(t)$, in analogy with \eqref{DFE}. 
Our first result ensures the existence
of such a quantity along $t\in\N$, if we average over the disorder.
One can also show the existence of such limit,
without restrictions on $t$,  in the $\bbP(\dd W)$-a.s.\ and $L^1$ senses:
we refer to \cite{cf:T15} for a proof.

\begin{theorem}[Continuum free energy]\label{Thm1}
For all $\alpha \in (\frac{1}{2},1)$, $\hb > 0$, $\hh \in \R$
the following limit exists and is finite:
\begin{align}\label{CFcE}
{\fhc}(\hb,\hh) :=  
\lim_{t\to\infty, \, t\in\N} \, \frac{1}{t} \, \bbE \Big[ \log \Zf(t) \Big]  \,.
\end{align}
The function ${\fhc}(\hb,\hh)$ is non-negative:
${\fhc}(\hb,\hh) \ge 0$ for all $\hb > 0$, $\hh \in \R$. Furthermore,
it is a convex function of $\hh$, for fixed $\hb$, and satisfies
the following scaling relation:
\begin{equation}\label{eq:scalingfe}
	{\fhc}(c^{\alpha-\frac{1}{2}} \hb,c^{\alpha}\hh) = 
	c \, {\fhc}(\hb,\hh) \,, \qquad \forall
	\hb > 0, \ \hh \in \R, \ c \in (0,\infty) \,.
\end{equation}
\end{theorem}

\smallskip

In analogy with \eqref{eq:hc}, we define the \emph{continuum critical curve}
$\hcc(\hb)$ by
\begin{equation}\label{criticalcurvecont}
\hcc(\hb)=\sup\{\hh \in \R : \ \fhc(\hb,\hh)=0\},
\end{equation}
which turns out to be positive and finite (see Remark~\ref{rem:finite} below).
Note that, by \eqref{eq:scalingfe},
\begin{equation} \label{eq:scaling}
	{\fhc}(\hb,\hh) = 
	{\fhc}\bigg(1, \frac{\hh}{\hb^{\frac{2\alpha}{2\alpha-1}}} \bigg) \,
	\hb^{\frac{2}{2\alpha-1}} \,, \qquad \text{hence} \qquad
	\hcc(\hb) = \hcc(1) \, \hb^{\frac{2\alpha}{2\alpha-1}} \,.
\end{equation}

Heuristically, the continuum free energy ${\fhc}(\hb,\hh)$
and critical curve $\hcc(\hb)$ 
capture the asymptotic behavior
of their discrete counterparts $\mathrm{F}(\beta,h)$ and $h_c(\beta)$
in the weak coupling regime $h, \beta \to 0$.
In fact, the convergence in distribution
\eqref{eq:convd} suggests that
\begin{equation} \label{eq:ui}
	\mathbb{E}\left[\log \Zf(t)\right] =
	\lim_{N\to\infty}\mathbb{E}\left[\log \Znf(Nt)\right].
\end{equation}
Plugging \eqref{eq:ui} into \eqref{CFcE} and
\emph{interchanging the limits $t \to \infty$ and $N \to \infty$} would yield
\begin{equation}\label{eqFreeEnergyIntro3}
{\fhc}(\hb,\hh) =
\lim_{t\to\infty}\frac{1}{t}\lim_{N\to\infty}\mathbb{E}\left[\log \Znf(Nt)\right] =
\lim_{N\to\infty} N \lim_{t\to\infty}\frac{1}{Nt} \mathbb{E}\left[\log \Znf(Nt)\right],
\end{equation}
which by \eqref{DFE} and \eqref{eqParam} leads to the key relation
(with $\epsilon = \frac{1}{N}$):
\begin{equation}\label{eqFreeEnergyIntro}
{\fhc}(\hb,\hh)=\lim_{N\to\infty}N \, \mathrm{F}(\beta_N,h_N)
= \lim_{\epsilon \downarrow 0} \frac{\mathrm{F}\big(\hb \, \epsilon^{\alpha-\frac{1}{2}}
L(\frac{1}{\epsilon}) , \, \hh \,
\epsilon^\alpha L(\frac{1}{\epsilon}) \big)}{\epsilon}.
\end{equation}
We point out that relation \eqref{eq:ui} is typically justified, as the
family $(\log \Znf(Nt))_{N\in\N}$ can be shown to be uniformly integrable,
but the interchanging of limits in \eqref{eqFreeEnergyIntro3} is in general
a delicate issue.
This was shown to hold for the copolymer model
with  tail exponent $\alpha < 1$, cf. \cite{BdH97,CG10},
but it is known to \emph{fail} for both pinning and copolymer models
with $\alpha > 1$ (see point 3 in \cite[\S 1.3]{CRZ13}).

The following theorem, which is our main result,
shows that for disordered pinning models with $\alpha \in (\frac{1}{2},1)$
relation \eqref{eqFreeEnergyIntro} does hold.
We actually prove a stronger relation,
which also yields the precise asymptotic behavior of the critical curve.

\begin{theorem}[Interchanging the limits]\label{Thm2}
Let $\mathrm{F}(\beta,h)$ be the free energy of the disordered pinning
model \eqref{model}-\eqref{DFE}, where the renewal process $\tau$ satisfies
\eqref{assRP0}-\eqref{assRP2} for some $\alpha \in (\frac{1}{2}, 1)$
and the disorder $\omega$ satisfies
\eqref{assD1}-\eqref{assD2}.
For all $\hb > 0$, $\hh\in \R$ and $\eta > 0$ there exists 
$\epsilon_0 > 0$ such that
\begin{equation} \label{eq:key}
{\fhc}\left(\hb,\hh - \eta\right)\leq 
\frac{\mathrm{F}\big(\hb \, \epsilon^{\alpha-\frac{1}{2}}
L(\frac{1}{\epsilon}) , \, \hh \,
\epsilon^\alpha L(\frac{1}{\epsilon}) \big)}{\epsilon}
\leq {\fhc}\left(\hb,\hh + \eta \right), \qquad
\forall \epsilon \in (0,\epsilon_0) \,.
\end{equation}
As a consequence, relation \eqref{eqFreeEnergyIntro} holds,
and furthermore
\begin{equation}\label{eq:keyhc}
\lim_{\beta\to 0}
\,\frac{h_c(\beta)}{\tilde{L}_\alpha(\frac{1}{\beta}) \, \beta^{\frac{2\alpha}{2\alpha-1}}}=\hcc(1) ,
\end{equation}
where $\tilde L_\alpha$ is the slowly function appearing in \eqref{eq:hb}
and the following lines.
\end{theorem}

Note that relation \eqref{eqFreeEnergyIntro} follows immediately by
\eqref{eq:key}, sending first $\epsilon \to 0$ and then $\eta \to 0$,
because $\hh \mapsto {\fhc}(\hb,\hh)$ is continuous
(by convexity, cf.\ Theorem~\ref{Thm1}). Relation
\eqref{eq:keyhc} also follows by \eqref{eq:key},
cf.\ \S\ref{sec:hcfe}, but it would not follow from \eqref{eqFreeEnergyIntro},
because convergence of functions does not necessarily
imply convergence of the respective
zero level sets. This is why we prove \eqref{eq:key}.

\begin{remark}\label{rem:finite}
Relation \eqref{eq:keyhc}, coupled with the known bounds \eqref{eq:known}
from the literature, shows in particular that $0 < \hcc(1) < \infty$ (hence
$0 < \hcc(\hb) < \infty$ for every $\hb > 0$, by \eqref{eq:scaling}).
Of course, in principle this can
be proved by direct estimates on the continuum
partition function.
\end{remark}

\subsection{On the critical behavior}
Fix $\hb > 0$.
The scaling relations \eqref{eq:scaling} imply that for all $\epsilon > 0$
\begin{equation*}
	{\fhc}(\hb,\hcc(\hb) + \epsilon) = 
	\hb^{\frac{2}{2\alpha-1}} \,
	{\fhc}\left(1, \hcc(1)+ \frac{\epsilon}{\hb^{\frac{2\alpha}{2\alpha-1}}} \right) \, \,.
\end{equation*}
Thus, as $\epsilon \downarrow 0$ (i.e.\ 
as $\hh \downarrow \hcc(\hb)$) the free energy vanishes in the same way;
in particular, \emph{the critical exponent $\gamma$ is the same for every $\hb$}
(provided it exists):
\begin{equation}\label{eq:critexp}
	{\fhc}(1,\hh) \underset{\hh \downarrow \hcc(1)}{=} (\hh - \hcc(1))^{\gamma + o(1)}
	\qquad \Longrightarrow \qquad
	{\fhc}(\hb,\hh) \underset{\hh \downarrow \hcc(1)}{=} 
	\hb^{\frac{-2(\alpha\gamma-1)}{2\alpha-1}} (\hh - \hcc(\hb))^{\gamma + o(1)} \,.
\end{equation}

Another interesting observation is that the smoothing inequality of \cite{GTsmooth}
can be extended to the continuum. For instance, in the case of Gaussian disorder
$\omega_i \sim N(0,1)$, it is known that the discrete free energy
$\mathrm{F}(\beta,h)$ satisfies the following relation, for all $\beta > 0$ and $h\in\R$:
\begin{equation*}
\begin{split}
	0 \le \mathrm{F}(\beta,h) \le \frac{1+\alpha}{2 \beta^2} \, (h-h_c(\beta))^2 \,.
\end{split}
\end{equation*}
Consider a renewal process satisfying \eqref{assRP0} with $L \equiv 1$
(so that also $\tilde L_\alpha \equiv 1$, cf.\ Remark~\ref{rem:hb}).
Choosing $\beta = \hb \, \epsilon^{\alpha-\frac{1}{2}}$ 
and $h = \hh \, \epsilon^\alpha$ 
and letting $\epsilon \downarrow 0$, we can
apply our key results \eqref{eqFreeEnergyIntro} and \eqref{eq:keyhc}
(recall also \eqref{eq:scaling}),
obtaining a smoothing inequality for the continuum free energy:
\begin{equation*}
	{\fhc}(\hb,\hh) \le \frac{1+\alpha}{2 \hb^2} \, \left(\hh-
	\hcc(\hb) \right)^2 \,.
\end{equation*}
In particular, the exponent $\gamma$ in \eqref{eq:critexp} has to satisfy $\gamma \ge 2$
(and consequently, the prefactor in the second relation in \eqref{eq:critexp} is
$\hb^{-\eta}$ with $\eta > 0$).

\subsection{Further results}
\label{sec:further}

Our results on the free energy and critical curve 
are based on a comparison of discrete and continuum partition function,
whose properties we investigate in depth. Some of the results
of independent interest are presented here.

Alongside the ``free'' partition function $\mathrm{Z}_{\beta,h}^{\omega}(N)$
in \eqref{model}, it is useful to consider a family $\mathrm{Z}_{\beta,h}^{\omega,\mathtt{c}}(a,b)$
of ``conditioned'' partition functions, for $a,b\in\mathbb{N}_0$ with $a \le b$:
\begin{equation}\label{discrcondfor}
\mathrm{Z}_{\beta,h}^{\omega,\mathtt{c}}(a,b)=\mathrm{E}\left(\left.e^{\sum_{k=a+1}^{b-1}(\beta\omega_k-\Lambda(\beta)+h)\ind_{k\in\tau}} \right| a\in\tau, b\in\tau\right) \,.
\end{equation}
If we let $N\to\infty$ with $\beta_N, h_N$ as in \eqref{eqParam},
the partition
functions $\mathrm{Z}_{\beta_N,h_N}^{\omega,\mathtt{c}}(Ns,Nt)$, for $(s,t)$ in
\begin{equation*}
	[0,\infty)_{\leq}^2:=\{(s,t)\in [0,\infty)^2 \mid s\leq t\} \,,
\end{equation*}
converge in the sense of finite-dimensional distributions \cite[Theorem 3.1]{CRZ13},
in analogy with \eqref{eq:convd}:
\begin{equation} \label{eq:convd2}
\Big( \mathrm{Z}_{\beta_N,h_N}^{\omega,\mathtt{c}}(Ns,Nt)
\Big)_{(s,t) \in [0,\infty)^2_\le} \tolaw 
\Big( \Z(s,t) \Big)_{(s,t) \in [0,\infty)^2_\le}, \quad N\to\infty ,
\end{equation}
where $\Z(s,t)$ admits an explicit Wiener chaos expansion, cf.\ \eqref{eq:Wiener2} below.

\smallskip

It was shown in \cite[Theorem 2.1 and Remark 2.3]{CRZ14} that,
under the further assumption \eqref{assRP2},
the convergences \eqref{eq:convd} and \eqref{eq:convd2}
can be upgraded: by linearly interpolating the
discrete partition functions for $Ns, Nt\not\in \N_0$, one has convergence
in distribution in the space of continuous functions of $t \in [0,\infty)$
and of $(s,t) \in [0,\infty)^2_\le$, respectively, equipped
with the topology of uniform convergence on compact sets.
We strengthen this result, by
showing that the convergence is locally uniform also in the variable $\hh \in \R$.
We formulate this fact through the existence
of a suitable coupling. 

\begin{theorem}[Uniformity in $\hh$]\label{PropByProd}
Assume
\eqref{assRP0}-\eqref{assRP2}, for some $\alpha \in (\frac{1}{2},1)$, and
\eqref{assD1}. For all $\hb > 0$,
%
there is a coupling of discrete and continuum partition functions such that
the convergence \eqref{eq:convd}, resp.\ \eqref{eq:convd2}, holds 
$\bbP(\dd \omega, \dd W)$-a.s.\ 
uniformly in any compact set of values of $(t,\hh)$, resp.\ of $(s,t,\hh)$.
\end{theorem}

We prove Theorem~\ref{PropByProd} 
by showing that partition functions with
$\hh \ne 0$ can be expressed in terms of those with $\hh = 0$ through
an explicit series expansion (see Theorem~\ref{Theorem676result} below).
This representation shows that the continuum partition functions
are increasing in $\hh$. They are also log-convex in $\hh$, because
$h \mapsto \log \mathrm{Z}_{\beta, h}^\omega$ and
$h \mapsto \log \mathrm{Z}_{\beta, h}^{\omega, \mathtt{c}}$ are convex functions
(by H\"older's inequality, cf.\ \eqref{model} and \eqref{discrcondfor})
and convexity is preserved by pointwise limits.
Summarizing:

\begin{proposition} \label{th:continuous}
For all $\alpha \in (\frac{1}{2},1)$ and $\hb > 0$,
the process $\Zf(t)$, resp.\ $\Z(s,t)$, admits a version which is
continuous in $(t,\hh)$, resp.\ in $(s,t,\hh)$. For fixed $t > 0$, resp.\ $t > s$,
the function $\hh \mapsto \log \Zf(t)$, resp.\ $\hh \mapsto \log \Z(s,t)$,
is strictly convex and strictly increasing.
\end{proposition}


\smallskip

We conclude with some important estimates, bounding (positive and negative) moments
of the partition functions and providing a  deviation inequality.

\begin{proposition}\label{p1results}
Assume
\eqref{assRP0}-\eqref{assRP2}, for some $\alpha \in (\frac{1}{2},1)$, and
\eqref{assD1}. Fix $\hb > 0$, $\hh \in \R$.
For all $T > 0$ and $p \in [0,\infty)$,
there exists a constant $C_{p,T} < \infty$ such that
\begin{equation}\label{AB3} 
 \bbE \left[ \sup_{0 \le s \le t \le T}
 \Zn(Ns,Nt) ^p\right]\leq C_{p,T} \,, \qquad \forall N\in\N \,.
\end{equation}
Assuming also \eqref{assD2}, relation \eqref{AB3} holds also for every $p \in (-\infty, 0]$,
and furthermore one has
\begin{equation}\label{AB3bis}
	\sup_{0\le s \le t \le T} \bbP\Big( \log \Zn(Ns,Nt) \le - x \Big)
	\le A_T \exp \left( - \frac{x^\gamma}{B_T} \right) \,, \qquad \forall x \ge 0,
	\ \forall N\in\N \,,
\end{equation}
for suitable finite constants $A_T$, $B_T$.
Finally, relations \eqref{AB3}, \eqref{AB3bis} hold also for the free partition function
$\Znf(Nt)$ (replacing $\sup_{0\le s \le t \le T}$ with $\sup_{0\le t \le T}$).
\end{proposition}

For relation \eqref{AB3bis} we use the concentration assumptions \eqref{assD2}
on the disorder. However, since $\log \Zn$ is not a uniformly (over $N\in\N$) 
Lipschitz function of $\omega$, some work is needed.

\smallskip

Finally, since the convergences in distribution \eqref{eq:convd}, \eqref{eq:convd2} hold
in the space of continuous functions, we can easily deduce analogues of 
\eqref{AB3}, \eqref{AB3bis} for the continuum partition functions.

\begin{corollary}\label{p1resultscont}
Fix $\alpha \in (\frac{1}{2},1)$, $\hb > 0$, $\hh \in \R$.
For all $T > 0$ and $p \in \R$
there exist finite constants $A_T$, $B_T$, $C_{p,T}$ (depending also
on $\alpha, \hb, \hh$) such that
\begin{gather}\label{AB3cont} 
 \bbE \left[ \sup_{0 \le s \le t \le T}
 \Z(Ns,Nt) ^p\right]\leq C_{p,T} \,,  \\
	\label{AB3biscont}
	\sup_{0\le s \le t \le T} \bbP\Big( \log \Z(Ns,Nt) \le - x \Big)
	\le A_T \exp \left( - \frac{x^\gamma}{B_T} \right) \,, \qquad \forall x \ge 0 \,.
\end{gather}
The same relations hold for the free partition function $\Zf(t)$
(replacing $\sup_{0\le s \le t \le T}$ with $\sup_{0\le t \le T}$).
\end{corollary}

%
%

\subsection{Organization of the paper}

The paper is structured as follows.
\begin{itemize}
\item We first prove Proposition~\ref{p1results}
and Corollary~\ref{p1resultscont} in Section~\ref{sectionNB}.

\item Then we prove Theorem~\ref{PropByProd} in Section~\ref{ContPinnModel}.

\item In Section~\ref{secproofthm2} we prove our main result,
Theorem~\ref{Thm2}. Our approach yields as a by-product
the existence of the continuum free energy, i.e.\ the core
of Theorem~\ref{Thm1}.

\item The proof of Theorem~\ref{Thm1} is
easily completed in Section~\ref{sec:Thm1}.

\item Finally some more technical points have been deferred to the Appendices~\ref{sec:rege}
and~\ref{sec:misc}.
\end{itemize}


\section{Proof of Proposition \ref{p1results}
and Corollary \ref{p1resultscont}}\label{sectionNB}

In this section we prove Proposition \ref{p1results}.
Taking inspiration from \cite{cf:Moreno}, we
 first prove \eqref{AB3bis}, using concentration results, and later
we prove \eqref{AB3}. 
We start with some preliminary results.

\subsection{Renewal results}

Let $(\sigma = (\sigma_n)_{n\in\N_0},\P)$ be a renewal process such that
$\P(\sigma_1 = 1) > 0$ and
\begin{equation}\label{eqsigmarenewal}
w(n):=\P(n\in\sigma)\overset{n\to\infty}{\sim} \frac{1}{M(n) \, n^{1-\nu}},
\quad \text{with} \quad \nu\in (0,1) \quad 
\text{and} \quad M(\cdot) \quad \text{slowly varying}\,.
\end{equation}
This includes any renewal process $\tau$ satisfying \eqref{assRP0}
with $\alpha \in (0,1)$, in which case \eqref{eqsigmarenewal} holds
with $\nu = \alpha$ and $M(n) = L(n)/C_\alpha$,
by \eqref{assRP1}. When $\alpha \in (\frac{1}{2},1)$,
another important example is given by the 
\emph{intersection renewal} $\sigma = \tau \cap \tau'$, where
$\tau'$ is an independent copy of $\tau$: since
$w(n) = \P(n \in \tau \cap \tau') = \P(n \in \tau)^2$ in this case, by \eqref{assRP1}
relation \eqref{eqsigmarenewal}
holds with $\nu =2\alpha -1$ and $M(n) = L(n)^2 / C_\alpha^2$.

For $N\in\N_0$ and $\delta\in\R$, let 
$\Psi_\delta(N), \Psi_\delta^{\mathtt{c}}(N)$ denote the (deterministic) functions
\begin{equation}\label{eq:Psi}
  \Psi_{\delta}(N)=\E\left[e^{\delta \sum_{n=1}^{N}\ind_{n\in\sigma}}
  \right] \,, \qquad
  \Psi_{\delta}^{\mathtt{c}}(N)=\E\left[e^{\delta \sum_{n=1}^{N-1}\ind_{n\in\sigma}}
  \absv N\in \sigma\right] \,,
\end{equation}
which are just the partition functions of a homogeneous (i.e.\ non disordered) pinning model.
In the next result, which is essentially a deterministic version of 
\cite[Theorem 2.1]{CRZ14} (see also \cite{Soh09}), we determine their limits
when $N \to \infty$ and $\delta = \delta_N \to 0$ as follows (for fixed $\hd \in \R$):
\begin{equation}\label{eq:deltaN}
	\delta_N \sim \hd \frac{M(N)}{N^{\nu}} \,.
\end{equation}

\begin{theorem}\label{ThmUConv}
Let the renewal $\sigma$ satisfy \eqref{eqsigmarenewal}.
Then the functions $(\Psi_{\delta_N}(Nt))_{t \in [0,\infty)}$, 
$(\Psi_{\delta_N}^{\mathtt{c}}(Nt))_{t \in [0,\infty)}$,
with $\delta_N$ as in \eqref{eq:deltaN} and
linearly interpolated for $Nt \not\in \N_0$,
converges as $N\to\infty$ respectively to
 \begin{align}\label{RiemConv0}
\bPsi_{\hd}^{\nu}(t) & = 1+ \sum_{k = 1}^\infty \hd^k
\idotsint\limits_{0 < t_1<\cdots<t_k < t} 
\frac{1}{t_1^{1-\nu} (t_2 - t_1)^{1-\nu}
\cdots (t_k - t_{k-1})^{1-\nu}}\prod_{i=1}^k \dd {t_i} \,,\\
\label{RiemConv}
\bPsi_{\hd}^{\nu,\mathtt{c}}(t) & = 1+ \sum_{k = 1}^\infty \hd^k
\idotsint\limits_{0 < t_1<\cdots<t_k < t} 
\frac{t^{1-\nu} }{t_1^{1-\nu} (t_2 - t_1)^{1-\nu}
\cdots (t_k - t_{k-1})^{1-\nu}(t-t_k)^{1-\nu}}\prod_{i=1}^k \dd {t_i} \,,
\end{align}
where the convergence is uniform on compact subsets of $[0,\infty)$.
The limiting functions $\bPsi_{\hd}^{\nu}(t), \bPsi_{\hd}^{\nu,\mathtt{c}}(t)$ are
strictly positive, finite and continuous in $t$.
\end{theorem}

\noindent
Before proving of Theorem~\ref{ThmUConv}, we summarize some useful consequences in the next Lemma.

\begin{lemma}\label{th:boundsren}
Let $\tau$ be a renewal process satisfying \eqref{assRP0}
with $\alpha \in (\frac{1}{2},1)$ and let $\omega$ satisfy \eqref{assD1}.
For every $\hb > 0$, $\hh \in \R$, defining $\beta_N, h_N$  as in \eqref{eqParam}, one has:
\begin{equation}\label{eq:boundsren}
	\lim_{N\to\infty} \bbE\left[ \Zn(0,Nt) \right] = \bPsi_{C_\alpha \hh}^{\alpha,\mathtt{c}}(t) \,,
	\qquad \lim_{N\to\infty} 
	\bbE\left[ \left( \Zno(0,Nt) \right)^2 \right] = 
	\bPsi_{C_\alpha^2 \hb^2}^{2\alpha-1,\mathtt{c}}(t) \,,
\end{equation}
uniformly on compact subsets of $t \in [0,\infty)$. Consequently
\begin{equation} \label{eq:rholambda}
	\rho := \inf_{N\in\N} \inf_{t\in [0,1]} \mathbb{E}\left[\Zn(0,Nt)\right] > 0 \,,
	\qquad
	\lambda := \sup_{N\in\N} \sup_{t\in [0,1]} 
	\mathbb{E}\left[\left(\Zno(0,Nt)\right)^2\right] < \infty \,.
\end{equation}
Analogous results hold for the free partition function.
\end{lemma}

\begin{proof}
We focus on the constrained partition function (the free one is analogous),
starting with the first relation in \eqref{eq:boundsren}.
By \eqref{discrcondfor}, for $Nt\in\N_0$ we can write
\begin{equation*}
	\bbE\left[ \Zn(0,Nt) \right] =
	\E\left[e^{h_N \sum_{k=1}^{Nt}\ind_{k\in\tau}}
	\Big| \, Nt \in \tau \right] = \Psi_{h_N}^{\mathtt{c}}(Nt) \,,
\end{equation*}
where we used \eqref{eq:Psi} with $\sigma = \tau$. As we observed
after \eqref{eqsigmarenewal}, we have $M(n) = L(n)/C_\alpha$ in this case,
so comparing \eqref{eq:deltaN} with \eqref{eqParam} we see that
$h_N \sim \delta_N$ with $\hd = C_\alpha \hh$.
Theorem~\ref{ThmUConv} then yields \eqref{eq:boundsren}.

Next we prove the second relation in \eqref{eq:boundsren}.
Denoting by $\tau'$ an independent copy of $\tau$, note that
$\bbE[e^{(\beta \omega_k - \Lambda(\beta))(\ind_{k\in\tau}+\ind_{k\in\tau'})}]
= e^{(\Lambda(2\beta)-2\Lambda(\beta))\ind_{k\in\tau\cap\tau'}}$.
Then, again by \eqref{discrcondfor}, 
for $h_N = 0$ we can write
\begin{equation} \label{eq:moment2}
\begin{split}
	\bbE\left[ \left( \Zno(0,Nt) \right)^2 \right] & =
	\bbE \left[ \E \left[ \left. e^{\sum_{k=1}^{Nt-1}(\beta_N\omega_k-\Lambda(\beta_N))
	(\ind_{k\in\tau} + \ind_{k\in\tau'})}
	\right| Nt \in \tau \cap \tau' \right] \right] \\
	& = 	\E\left[e^{(\Lambda(2\beta_N)-2\Lambda(\beta_N)) 
	\sum_{k=1}^{Nt}\ind_{k\in\tau \cap \tau'}}
	\Big| \, Nt \in \tau \cap \tau' \right] = 
	\Psi_{\Lambda(2\beta_N)-2\Lambda(\beta_N)}^{\mathtt{c}}(Nt) \,,
\end{split}
\end{equation}
where in the last equality we
have applied \eqref{eq:Psi} with $\sigma = \tau \cap \tau'$,
for which $\nu = 2\alpha-1$ and $M(n) = L(n)^2/C_\alpha^2$.
Since $\Lambda(\beta) = \frac{1}{2}\beta^2 + o(\beta^2)$ as $\beta \to 0$,
by \eqref{assD1}, it follows that
$\Lambda(2\beta_N)-2\Lambda(\beta_N)\sim \beta_N^2 
\sim \delta_N$ with $\hd = C_\alpha^2 \hb^2$,
by \eqref{eqParam} and \eqref{eq:deltaN}.
In particular, Theorem~\ref{ThmUConv} yields the second relation in \eqref{eq:boundsren}.

Finally we prove \eqref{eq:rholambda}. Since the convergence \eqref{eq:boundsren} is uniform in $t$,
\begin{equation*}
	\lim_{N\to\infty} \inf_{t\in [0,1]} \bbE\left[ \Zn(0,Nt) \right] = 
	\inf_{t\in [0,1]} \bPsi_{C_\alpha \hh}^{\alpha,\mathtt{c}}(t) > 0 \,,
\end{equation*}
because $t \mapsto \bPsi_{C_\alpha^2 \hb^2}^{2\alpha-1,\mathtt{c}}(t)$ is continuous
and strictly positive. On the other hand, for fixed $N\in\N$,
\begin{equation*}
	\inf_{t\in[0,1]}\bbE\left[ \Zn(0,Nt) \right] =
	\min_{n \in \{0,1,\ldots,N\}} \bbE\left[ \Zn(0,n) \right] > 0 \,,
\end{equation*}
so the first relation in \eqref{eq:rholambda} follows. The second one is proved
with analogous arguments.
\end{proof}

\begin{proof}[Proof of Theorem~\ref{ThmUConv}]
The continuity in $t$ of $\bPsi_{\hd}^{\nu}(t), \bPsi_{\hd}^{\nu,\mathtt{c}}(t)$
can be checked directly by \eqref{RiemConv0}-\eqref{RiemConv}. They are also non-negative
and non-decreasing in $\hd$,
being pointwise limits of the non-negative and non-decreasing functions \eqref{eq:Psi}
(these properties are
not obviously seen from \eqref{RiemConv0}-\eqref{RiemConv}).
Since $\bPsi_{\hd}^{\nu}(t), \bPsi_{\hd}^{\nu,\mathtt{c}}(t)$ are clearly
analytic functions of $\hd$,
they must be strictly increasing in $\hd$, hence they must be strictly positive, as stated.

Next we prove the convergence results.
We focus on the constrained case $\Psi_{\delta_N}^{\mathtt{c}}(Nt)$,
since the free one is analogous (and simpler).
We fix $T \in (0,\infty)$ and show uniform convergence for $t\in [0,T]$.
This is equivalent, as one checks by contradiction,
to show that for any given sequence $(t_N)_{N\in\N}$ in $[0,T]$
one has $\lim_{N\to\infty} |\Psi_{\delta_N}^{\mathtt{c}}(Nt_N) 
- \bPsi_{\hd}^{\nu,\mathtt{c}}(t_N)| = 0$. By a subsequence
argument, we may assume that $(t_N)_{N\in\N}$ has a limit,
say $\lim_{N\to\infty} t_N = t \in [0,T]$, so we are left
with proving 
\begin{equation}\label{eq:goalto}
	\lim_{N\to\infty} \Psi_{\delta_N}^{\mathtt{c}}(N t_N) = \bPsi_{\hd}^{\nu,\mathtt{c}}(t) \,.
\end{equation}
We may safely assume that $N t_N \in \N_0$, since $\Psi_{\delta_N}(Nt)$
is linearly interpolated for $Nt \not\in \N_0$.
For notational simplicity we also assume that $\delta_N$ is exactly equal
to the right hand side of \eqref{eq:deltaN}.

Recalling \eqref{eqsigmarenewal}, for $0 < n_1 < \ldots < n_k < Nt_N$ we have
\begin{equation} \label{eq:facto}
	\E\left[ \ind_{n_1 \in \sigma} \ind_{n_2 \in \sigma}
	\cdots \ind_{n_k \in \sigma} \absv Nt_N\in \sigma\right] =
	\frac{w(n_1) w(n_2-n_1) \cdots w(Nt_N-n_k)}{w(Nt_N)} \,.
\end{equation}
Since $e^{\delta\ind_{n\in\tau}} = 1 + (e^\delta-1)\ind_{n\in\tau}$,
a binomial expansion in \eqref{eq:Psi} then yields
\begin{equation}\label{eq:RiSu}
\begin{split}
& \Psi_{\delta_N}^{\mathtt{c}} (Nt_N) =
1 + \sum_{k=1}^{Nt_N-1}(e^{\delta_N}-1)^k
\sum_{0 < n_1<\cdots<n_k < Nt_N}
\frac{w(n_1) w(n_2-n_1) \cdots w(Nt_N-n_k)}{w(Nt_N)} \\
& = 1 + 
\sum_{k=1}^{Nt_N-1}\left( \frac{e^{\delta_N}-1}{\delta_N} \right)^k \, \hd^k \,
\left\{ \frac{1}{N^k}
\sum_{0 < n_1<\cdots<n_k < Nt_N} \!\!\!
\frac{W_N(0,\tfrac{n_1}{N}) W_N(\tfrac{n_1}{N},\tfrac{n_2}{N}) \cdots 
W_N(\tfrac{n_k}{N},t_N)}{W_N(0,t_N)} \right\} ,
\end{split}
\end{equation}
where we have introduced for convenience the rescaled kernel
\begin{equation*}
	W_N(r,s) := M(N) N^{1-\nu}w(\lceil Ns \rceil - \lceil Nr \rceil) \,, \qquad 
	0 \le r \le s < \infty \,,
\end{equation*}
and $\lceil x \rceil := \min\{n\in\N: n \ge x\}$ denotes the upper integer part of $x$.
We first show the convergence of the term in brackets in
\eqref{eq:RiSu}, for fixed $k\in\N$; later we control the tail of the sum.

For any $\epsilon > 0$, uniformly for $r-s \ge \epsilon$
one has $\lim_{N\to\infty} W_N(r,s) = 1/(s-r)^{1-\nu}$,
by \eqref{eqsigmarenewal}. Then, for fixed $k\in\N$, the term in brackets in 
\eqref{eq:RiSu} converges to the corresponding integral in \eqref{RiemConv}
by a Riemann sum approximation,
provided the contribution to the sum given by
$n_i - n_{i-1} \le \epsilon N$ vanishes as $\epsilon \to 0$, uniformly in $N\in\N$.
We show this by a suitable upper bound on $W_N(r,s)$.
For any $\eta > 0$, by Potter's bounds \cite[Theorem 1.5.6]{BGT89},
we have $M(y)/M(x) \le C \max\{(\frac{y}{x})^\eta,(\frac{x}{y})^\eta\}$, hence
\begin{equation} \label{eq:bounddo}
	\frac{C^{-1}}{(r-s)^{1-\nu-\eta}} \le
	W_N(r,s) \le \frac{C}{(r-s)^{1-\nu+\eta}}, \qquad
	\forall N\in\N, \ \forall 0 \le r \le s \le T \,,
\end{equation}
for some constant $C = C_{\eta,T} < \infty$. Choosing $\eta \in (0,\nu)$,
the right hand side in \eqref{eq:bounddo} is integrable
and the contribution to the bracket in \eqref{eq:RiSu}
given by the terms with $n_i - n_{i-1} \le \epsilon N$ for some $i$
is dominated by the following integral 
\begin{equation}\label{eq:RiSu2}
\idotsint\limits_{0 < t_1<\cdots<t_k < t_N} 
\frac{C^{k+2}\, t_N^{1-\nu - \eta}}
{t_1^{1-\nu+\eta} (t_2 - t_1)^{1-\nu+\eta}
\cdots (t_N-t_k)^{1-\nu+\eta}}
\ind_{\{t_i - t_{i-1} \le \epsilon, \,  \text{for some}\, i =1,\cdots, k\}}
\prod_{i=1}^k \dd {t_i} \,.
\end{equation}
Plainly, for fixed $k\in\N$,
this integral vanishes as $\epsilon \to 0$ as required 
(we recall that $t_N \to t < \infty$).

It remains to show that the contribution to \eqref{eq:RiSu} given
by $k \ge M$ can be made small, \emph{uniformly in $N\in\N$}, by taking $M \in \N$ large enough.
By \eqref{eq:bounddo}, the
term inside the brackets in \eqref{eq:RiSu} can be bounded from above by the following integral
(where we make the change of variables $s_i = t_i/t_N$):
\begin{equation} \label{eq:ineq}
\begin{split}
	& \idotsint\limits_{0 < t_1<\cdots<t_k < t_N} 
	\frac{C^{k+2}\, t_N^{1-\nu- \eta} }{t_1^{1-\nu+\eta} (t_2 - t_1)^{1-\nu+\eta}
	\cdots (t_N-t_k)^{1-\nu+\eta}}
	\prod_{i=1}^k \dd {t_i} \\
	& =  \idotsint\limits_{0 < s_1<\cdots<s_k < 1} 
	\frac{ C^{k+2}\, t_N^{k(\nu-\eta)-2\eta}}
	{s_1^{1-\nu+\eta} (s_2 - s_1)^{1-\nu+\eta}
	\cdots (1-s_k)^{1-\nu+\eta}}
	\prod_{i=1}^k \dd {s_i} 
	\le \hat{C}_T^k \, c_1 e^{-c_2 k \log k} \,,
\end{split}
\end{equation}
for some constant $\hat{C}_T$ depending only on $T$ (recall that $t_N \to t \in [0,T]$),
where the inequality is proved in \cite[Lemma B.3]{CRZ13},
for some constants $c_1, c_2 \in (0,\infty)$, depending only on
$\nu, \eta$. This shows that \eqref{eq:goalto} holds
 and that the limits are finite, completing the proof.
\end{proof}

\subsection{Proof of relation \eqref{AB3bis}}

Assumption \eqref{assD2} is equivalent to a suitable
concentration inequality for the Euclidean distance 
$d(x, A) := \inf_{y\in A} |y-x|$ from a point $x \in \R^n$
to a convex set $A \subseteq \R^n$. 
More precisely, the following Lemma is quite standard (see
\cite[Proposition 1.3 and Corollary 1.4]{L05}, except for convexity
issues), but for completeness we give a proof in Appendix~\ref{sec:improve}.

\begin{lemma} \label{th:improve}
Assuming \eqref{assD2}, there exist $C_1', C_2' \in (0,\infty)$ such that
for every $n\in\N$ and for any convex set $A \subseteq \R^n$ one has
(setting $\omega = (\omega_1, \ldots, \omega_n)$ for short)
\begin{equation} \label{eq:distconc}
	\bbP(\omega \in A) \,
	\bbP(d(\omega, A) > t) \le C_1'
	\exp \left(-\frac{t^\gamma}{C_2'}\right) \,, \qquad \forall t \ge 0 \,.
\end{equation}
Viceversa, assuming \eqref{eq:distconc}, relation \eqref{assD2} holds
for suitable $C_1, C_2 \in (0,\infty)$.
\end{lemma}

The next result, proved in Appendix~\ref{sec:conve},
is essentially \cite[Proposition 1.6]{L05} and shows that
\eqref{eq:distconc} yields concentration bounds for
convex functions that are not necessarily (globally) Lipschitz.

\begin{proposition}\label{th:conve}
Assume that \eqref{eq:distconc} holds for every $n\in\N$ and for any convex set $A \subseteq \R^n$.
Then, for every $n\in\N$ and for every differentiable
convex function $f: \R^n \to \R$ one has
\begin{equation}\label{eq:keycon}
	\bbP(f(\omega) \le a-t) \, \bbP(f(\omega) \ge a, \, |\nabla f(\omega)| \le c)
	\le C_1' \exp \left(-\frac{(t/c)^\gamma}{C_2'}\right) \,, \quad \forall
	a \in \R, \ \forall t, c \in (0,\infty) \,,
\end{equation}
where $|\nabla f(\omega)| := \sqrt{\sum_{i=1}^n \big( \partial_i f(\omega) \big)^2 }$
denotes the Euclidean norm of the gradient of $f$.
\end{proposition}

The usefulness of \eqref{eq:keycon} can be understood as follows:
given a family of functions $(f_i)_{i\in I}$,
if we can control the probabilities $p_i := \bbP(f_i(\omega) \ge a, \, |\nabla f_i(\omega)| \le c)$,
showing that $\inf_{i\in I} p_i = \theta > 0$ for some fixed $a, c$,
then \eqref{eq:keycon} provides a \emph{uniform control on the left tail}
$\bbP(f_i(\omega) \le a-t)$.
This is the key to the proof of relation \eqref{AB3bis}, as we now explain.

\smallskip

We recall that $\Zn(a,b)$ was defined in \eqref{discrcondfor}.
Our goal is to prove relation \eqref{AB3bis}. Some preliminary remarks:
\begin{itemize}
\item we consider the case $T=1$, for notational simplicity;

\item we can set $s=0$ in \eqref{AB3bis}, because $\Zn(a,b)$ has the same law as
$\Zn(0,b-a)$.

\end{itemize}
We can thus reformulate our goal \eqref{AB3bis} as follows: for some constants $A, B < \infty$
\begin{equation}\label{AB3bis0}
	\sup_{0 \le t \le 1} \bbP\Big( \log \Zn(0,Nt) \le - x \Big)
	\le A \exp \left( - \frac{x^\gamma}{B} \right) \,, \qquad \forall x \ge 0,
	\ \forall N\in\N \,.
\end{equation}
We can further assume that $h_N \le 0$,
because for $h_N > 0$ we have $\Zn(0,Nt) \ge \Zno(0,Nt)$ and
replacing $h_N$ by $0$
yields a stronger statement.
Applying Proposition~\ref{th:conve} to the functions
\begin{equation*}
	f_{N,t}(\omega)  := \log \Zn(0, Nt) \,,
\end{equation*}
relation \eqref{AB3bis0} is implied by the following result.

\begin{lemma}
Fix $\hb > 0$ and $\hh \le 0$.
There are constants $a \in \R$, $c \in (0,\infty)$
such that
\begin{equation*}
	\inf_{N\in\N} \inf_{t\in [0,1]} \bbP(f_{N,t}(\omega) \ge a, \ 
	|\nabla f_{N,t}(\omega)| \le c)
	=: \theta > 0 \,.
\end{equation*}
\end{lemma}

\begin{proof}
Recall Lemma~\ref{th:boundsren}, in particular the definition \eqref{eq:rholambda}
of $\rho$ and $\lambda$. By the Paley-Zygmund inequality, for all $N\in\N$ and $t\in[0,1]$
we can write
\begin{equation} \label{eq:PZ}
 \mathbb{P}\left(\Zn(0,Nt)\geq \frac{\rho}{2} \right)
 \geq \mathbb{P}\left(\Zn(0,Nt)\geq \frac{\mathbb{E}\left[\Zn(0,Nt)\right]}{2}\right)
 \geq \frac{\left(\mathbb{E}\left[\Zn(0,Nt)\right]\right)^2}
 {4\mathbb{E}\left[\left(\Zn(0,Nt)\right)^2\right]}  \,.
\end{equation}
Replacing $h_N \le 0$ by $0$ in the denominator,
we get the following lower bound, with $a := \log \frac{\rho}{2}$:
\begin{equation}\label{eq:lbu}
	\bbP\left( f_{N,t}(\omega) \ge a\right) =
	\mathbb{P}\left(\Zn(0,Nt)\geq \frac{\rho}{2} \right) \ge
	\frac{\rho^2}{4\lambda} \,, \qquad \forall N \in\N, \ t \in [0,1]\,.
\end{equation}

Next we focus on $\nabla f_{N,t} (\omega)$.
Recalling \eqref{discrcondfor}, we have
\begin{equation*}
	\frac{\partial f_{N,t}}{\partial \omega_i} (\omega)
	= \beta_N \frac{\E[ \ind_{i \in \tau} e^{\sum_{k=1}^{Nt-1}(\beta\omega_k-\Lambda(\beta)+h)
	\ind_{k\in\tau}} | Nt \in \tau]}{\Zn(0,Nt)} \, \ind_{i \le Nt-1} \,,
\end{equation*}
hence, denoting by $\tau'$ an independent copy of $\tau$,
\begin{equation*}
	|\nabla f_{N,t} (\omega)|^2 = \sum_{i=1}^N 
	\left(\frac{\partial f_{N,t}}{\partial \omega_i} (\omega)\right)^2
	= \beta_N^2 \frac{\E[ ( \sum_{i=1}^{Nt-1} \ind_{i \in \tau \cap \tau'} )\,
	e^{\sum_{k=1}^{Nt-1}(\beta_N\omega_k-\Lambda(\beta_N)+h_N)
	(\ind_{k\in\tau} + \ind_{k\in\tau'})} |Nt \in \tau\cap\tau']}{\Zn(0,Nt)^2} \,.
\end{equation*}
Since $h_N \le 0$, we replace $h_N$ by $0$ in the numerator getting an upper bound. 
Recalling that $a = \log \frac{\rho}{2}$,
\begin{equation*}
\begin{split}
	\bbP\left(f_{N,t}(\omega) \ge a, \  |\nabla f_{N,t} (\omega)| > c \right)
	& \le \frac{\bbE[|\nabla f_{N,t} (\omega)|^2 \ind_{\{f_{N,t}(\omega) \ge a\}}]}
	{c^2} = \frac{\bbE[|\nabla f_{N,t} (\omega)|^2 \ind_{\{\Zn(0,Nt) \ge \frac{\rho}{2}\}}]}
	{c^2} \\
	& \le \frac{4}{\rho^2c^2} 
	\E\left[ \left. \left(\beta_N^2 \sum_{i=1}^{Nt-1} \ind_{i \in \tau \cap \tau'} \right)\,
	e^{(\Lambda(2\beta_N)-2\Lambda(\beta_N)) 
	\sum_{k=1}^{Nt-1}\ind_{k\in\tau \cap \tau'}} \right| Nt \in \tau \cap \tau' \right] \,.
\end{split}
\end{equation*}
We recall that $\Lambda(2\beta_N)-2\Lambda(\beta_N)\sim \beta_N^2$, by \eqref{assD1},
hence $\Lambda(2\beta_N)-2\Lambda(\beta_N) \le C \beta_N^2$ for some $C \in (0,\infty)$.
Since $x \le e^x$ for all $x \ge 0$, we obtain
\begin{equation*}
	\bbP\left(f_{N,t}(\omega) \ge a, \  |\nabla f_{N,t} (\omega)| > c \right) \le
	\frac{4}{\rho^2c^2} 
	\E\left[ \left.
	e^{(C+1) \beta_N^2
	\sum_{k=1}^{Nt-1}\ind_{k\in\tau \cap \tau'}} \right| Nt \in \tau \cap \tau' \right]
	= \frac{4}{\rho^2c^2} \Psi_{(C+1) \beta_N^2}^{\mathtt{c}}(Nt) \,,
\end{equation*}
where we used the definition \eqref{eq:Psi}, with $\sigma = \tau \cap \tau'$,
which we recall that
satisfies \eqref{eqsigmarenewal} with $\nu =2\alpha -1$ and $M(n) = L(n)^2 / C_\alpha^2$.
In particular, as we discussed in the proof of Lemma~\ref{th:boundsren},
$\beta_N^2 \sim \delta_N$ in \eqref{eq:deltaN} with $\hd = C_\alpha^2 \hb^2$,
hence $\Psi_{(C+1) \beta_N^2}^{\mathtt{c}}(Nt)$ is uniformly bounded, by Theorem~\ref{ThmUConv}:
\begin{equation} \label{eq:xi}
	\xi := \sup_{N\in\N} \sup_{t\in [0,1]} \Psi_{(C+1) \beta_N^2}^{\mathtt{c}}(Nt) < \infty \,.
\end{equation}

In conclusion, with $\rho, \lambda, \xi$
defined in \eqref{eq:rholambda}-\eqref{eq:xi},
setting $a := \log \frac{\rho}{2}$ one has, for every $c > 0$,
\begin{equation*}
\begin{split}
	\bbP(f_{N,t}(\omega) \ge a, \ 
	|\nabla f_{N,t}(\omega)| \le c) & = \bbP(f_{N,t}(\omega) \ge a)
	- \bbP(f_{N,t}(\omega) \ge a, \ 
	|\nabla f_{N,t}(\omega)| > c) \\
	& \ge \frac{\rho^2}{4\lambda} - \frac{4 \xi}{\rho^2 c^2} =: \theta \,,
	\qquad \forall N\in\N, \ t \in [0,1] \,.
\end{split}
\end{equation*}
Choosing $c > 0$ large enough one has $\theta > 0$, and the proof is completed.
\end{proof}

\subsection{Proof of \eqref{AB3}, case $p \ge 0$.}
We recall Garsia's inequality \cite{G72} with $\Psi(x)=|x|^p$ and $\phi(u)=u^q$: 
for all $p\geq 1$, $\mu>0$ with $p\mu>4$ we have
for every $0\le s_i\le t_i\le 1$, $i=1,2$,
 \begin{align}\label{eqControlG}
  \left|\Zn(Ns_1,Nt_1)-\Zn(Ns_2,Nt_2)\right|&\leq \frac{8\mu}{\mu-4/p}
 \, B_N \,
  |(s_1,t_1)-(s_2,t_2) |^{\mu-4/p}
 \end{align}
where $|\cdot|$ denotes the Euclidean norm
and $B_N$ is an explicit (random) constant depending of $p$:
 \begin{equation}\label{eqBgarzia}
  B_N^p= 2^{\mu/2} \displaystyle\int_{[0,1]_{\leq}^2 \times [0,1]_{\leq}^2}  
  \frac{\left|\Zn(Ns_1,Nt_1)-\Zn(Ns_2,Nt_2)\right|^p}{| (s_1,t_1)-(s_2,t_2) |^{p\mu}} 
  \dd s_1 \dd t_1 \dd s_2 \dd t_2.
 \end{equation} 
Since $\Zn(0,0) = 1$ and $|a+b|^p \le 2^p (|a|^p + |b|^p)$, it follows that
\begin{equation*}
	 \bbE \left[ \sup_{0 \le s \le t \le T}
	\Zn(Ns,Nt) ^p\right]\leq 2^p\left( 1 + 
	\left(\frac{8\mu }{\mu-4/p}\right)^p\,(\sqrt{2} T )^{p\mu-4}
	\bbE\left[ B_N^p \right]
	\right) \,.
\end{equation*}
We are thus reduced to estimating $\bbE[ B_N^p ]$.

It was shown in \cite[Section 2.2]{CRZ14} that for any $p\geq 1$ 
there exist $C_p>0$ and $\eta_p>2$ for which 
\begin{equation}\label{Eq1Garsia}
 \sup_{N\in\N} \mathbb{E}\left(\left|\Zn(Ns_1,Nt_1)-\Zn(Ns_2,Nt_2)\right|^p\right)\leq 
 C_p |(t_1,s_1)-(t_s,s_2) |^{\eta_p}.
\end{equation}
The value of $\eta_p$ is actually explicit,
cf.\ \cite[eq. (2.25), (2.34), last equation in \S2.2]{CRZ14},
and such that
\begin{equation*}
	\lim_{p\to\infty} \frac{\eta_p}{p} = \bar \mu > 0 \,, \qquad\text{where}
	\qquad \bar\mu = \frac{1}{2} \min\left\{\alpha'-\tfrac{1}{2}, \delta\right\} \,,
\end{equation*}
where $\delta > 0$ is the exponent in \eqref{assRP2}
and $\alpha'$ is any fixed number in $(\frac{1}{2},\alpha)$.
If we choose any $\mu \in (0,\bar\mu)$,
plugging \eqref{Eq1Garsia} into \eqref{eqBgarzia} we see that
the integral is finite for large $p$, completing the proof.\qed

\subsection{Proof of \eqref{AB3}, case $p \le 0$.}

We prove that an analogue of \eqref{Eq1Garsia} holds. Once proved this, the proof runs 
as for the case $p\ge 0$, using Garsia's inequality \eqref{eqControlG}
for $1/\Zn(Ns,Nt)$.

We first claim that for every $p > 0$ there exists $D_p<\infty$ such that
\begin{equation}\label{eqNB1}
\mathbb{E}\left( \Zn(Ns,Nt)^{-p}\right)\leq D_p \,, \qquad
\forall N\in\N, \ 0 \le s \le t \le 1 \,.
\end{equation}  
This follows by \eqref{AB3bis}:
\begin{equation*}
\begin{split}
	\mathbb{E}\left( \Zn(0,Nt) ^{-p}\right)
	& = \int_{0}^\infty \bbP\left( \Zn(0,Nt) ^{-p} > y \right) \dd y 
	= \int_{0}^\infty \bbP\left( \log \Zn(0,Nt) < -p \log y \right) \dd y \\
	& \le  1+A \int_{1}^\infty \exp \left( - \frac{p^\gamma (\log y)^\gamma}{B} \right)
	\, \dd y =1+ A \int_{0}^\infty \exp \left( - \frac{p^\gamma x^\gamma}{B} \right)
	\, e^x \, \dd x < \infty \,
\end{split}
\end{equation*}
 where in the last step we used $\gamma>1$.
Then, by \eqref{eqNB1}, applying the
Cauchy-Schwarz inequality twice gives
\begin{align*}
 &\mathbb{E}\left[\left|\frac{1}{\Zn(Ns_1,Nt_1)}-\frac{1}{\Zn(Ns_2,Nt_2)}\right|^p\right]
 = \mathbb{E}\left[\left|\frac{\Zn(Ns_1,Nt_1) - \Zn(Ns_2,Nt_2)}
 {\Zn(Ns_1,Nt_1) \, \Zn(Ns_2,Nt_2)}
 \right|^p\right]  \\
  & \le \sqrt{D_{4p}} \,
 \mathbb{E}\left(\left|{\Zn(Ns_1,Nt_1)}-\Zn(Ns_2,Nt_2)\right|^{2p}\right)^{\frac{1}{2}}\overset{\eqref{Eq1Garsia}}{\leq} 
 \sqrt{D_{4p} \, C_p} \, |(t_1,s_1)-(t_s,s_2)|^{\eta_{2p}/2} \,,
\end{align*}
completing the proof.\qed

\section{Proof of Theorem \ref{PropByProd}}\label{ContPinnModel}

Throughout this section we fix $\hb > 0$. 
We recall that the discrete partition functions
$\mathrm{Z}_{\beta,h}^\omega(Nt)$,
$\mathrm{Z}_{\beta,h}^{\omega,\mathtt{c}}(Ns, Nt)$
are linearly interpolated for $Ns, Nt \not\in \N_0$.
We split the proof in three steps.

\medskip

\noindent
\textit{Step 1. The coupling.}
For notational clarity, we denote with the letters $\mathrm{Y}, \mathbf{Y}$
the discrete and continuum partition functions $\mathrm{Z}, \mathbf{Z}$ 
in which we set $h, \hh = 0$:
\begin{equation} \label{eq:YY}
\begin{split}
	\mathrm{Y}_{\beta}^\omega(N) := \mathrm{Z}_{\beta,0}^\omega(N) \,, \qquad
	& \Yf(t) := \Zfo(t) \,, \\
	\mathrm{Y}_{\beta}^{\omega,\mathtt{c}}(a,b) 
	:= \mathrm{Z}_{\beta,0}^{\omega,\mathtt{c}}(a,b) \,, \qquad
	& \Y(s,t) := \Zo(s,t) \,.
\end{split}
\end{equation}
We know by \cite[Theorem 2.1 and Remark 2.3]{CRZ14} that 
for fixed $\hh$ (in particular, for $\hh = 0$) 
the convergence in distribution
\eqref{eq:convd}, resp.\ \eqref{eq:convd2}, holds in the space of continuous
functions of $t \in [0,\infty)$, resp.\ $(s,t) \in [0,\infty)^2_\le$, 
with uniform convergence on compact sets.
 By Skorohod's representation theorem (see Remark~\ref{rem:Sk} below),
we can fix a continuous version of the processes $\mathbf{Y}$ and a
coupling of $\mathrm{Y}, \mathbf{Y}$ such that $\bbP(\dd\omega, \dd W)$-a.s.
\begin{equation}\label{eq:Sk}
	\forall T > 0: \quad
	\sup_{0 \le t \le T} \Big| \mathrm{Y}_{\beta_N}^\omega(Nt)
	- \Yf(t) \Big| \xrightarrow[N\to\infty]{}0 \,, \quad \
	\sup_{0 \le s \le t \le T} \Big| \mathrm{Y}_{\beta_N}^{\omega,\mathtt{c}}(Ns,Nt)
	- \Y(s,t) \Big| \xrightarrow[N\to\infty]{}0 \,.
\end{equation}
We stress that the coupling depends only on the fixed value of $\hb > 0$.

The rest of this section consists in showing that under this coupling
of $\mathrm{Y}, \mathbf{Y}$,
the partition functions converge locally uniformly also in
the variable $\hh$. More precisely, we show that there is a version of the processes
$\Zf(t)$ and $\Z(s,t)$ such that $\bbP(\dd\omega, \dd W)$-a.s.
\begin{equation}\label{eq:Sk2}
\begin{split}
	\forall T, M \in (0,\infty): \quad
	& \sup_{0 \le t \le T, \ |\hh|\le M} \Big| 
	\mathrm{Z}_{\beta_N, h_N}^\omega(Nt)
	- \Zf(t) \Big| \xrightarrow[N\to\infty]{}0 \,, \\
	& \sup_{0 \le s \le t \le T, \ |\hh|\le M} 
	\Big| \mathrm{Z}_{\beta_N, h_N}^{\omega,\mathtt{c}}(Ns,Nt)
	- \Z(s,t) \Big| \xrightarrow[N\to\infty]{}0 \,.
\end{split}
\end{equation}

\begin{remark}\label{rem:Sk}
A slightly strengthened version of the usual Skorokhod representation theorem \cite[Corollaries
5.11--5.12]{K97}
ensures that one can indeed couple not only the processes $\mathrm{Y}, \mathbf{Y}$,
but even the environments $\omega, W$ of which they are functions,
so that \eqref{eq:Sk} holds.
More precisely, one can define on the same probability space
a Brownian motion $W$ and a family $(\omega^{(N)})_{N\in\N}$,
where $\omega^{(N)} = (\omega^{(N)}_i)_{i\in\N}$ is for each $N$
an i.i.d.\ sequence with the original disorder distribution, 
such that \emph{plugging $\omega = \omega^{(N)}$ into
$\mathrm{Y}_{\beta_N}^\omega(\cdot)$}, relation \eqref{eq:Sk} holds a.s..
(Of course, the sequences $\omega^{(N)}$ and $\omega^{(N')}$ will not be independent
for $N \ne N'$.)
We write $\bbP(\dd\omega, \dd W)$ for the joint
probability with respect to $(\omega^{(N)})_{N\in\N}$ and $W$. For notational simplicity,
we will omit the superscript $N$ from $\omega^{(N)}$ in $\mathrm{Y}_{\beta_N}^\omega(\cdot)$,
$\mathrm{Z}_{\beta_N, h_N}^\omega(\cdot)$, etc..
\end{remark}

\medskip
\noindent
\textit{Step 2. Regular versions.}
The strategy to deduce \eqref{eq:Sk2} from \eqref{eq:Sk} is to express the partition
functions $\mathrm{Z}, \mathbf{Z}$ for $\hh \ne 0$ in terms of the $\hh = 0$ case,
i.e.\ of $\mathrm{Y}, \mathbf{Y}$. We start doing this in the continuum.

We recall the Wiener chaos expansions of the continuum partition functions,
obtained in \cite[Theorem 3.1]{CRZ13},
where as in \eqref{assRP1} we define the constant $C_\alpha := \frac{\alpha\sin(\alpha\pi)}{\pi}$:
\begin{equation} \label{eq:Wiener}
	\Zf(t) = 1 + \sum_{n=1}^\infty \
	\idotsint\limits_{0 < t_1 < t_2 < \ldots < t_n < t}   \,
	\frac{C_\alpha^n}
	{ t_1^{1-\alpha} (t_2-t_1)^{1-\alpha}\cdots (t_n-t_{n-1})^{1-\alpha}}
	\prod_{i=1}^n \big( \hb \,\dd W_{t_i} + \hh\, \dd t_i \big) \,.
\end{equation}
\begin{equation} \label{eq:Wiener2}
\begin{split}
	& \Z(s,t) = 1 + \\
	& \quad \ \sum_{n=1}^\infty \
	\idotsint\limits_{s < t_1 < t_2 < \ldots < t_n < t}   \,
	\frac{C_\alpha^n \, (t-s)^{1-\alpha}}
	{ (t_1-s)^{1-\alpha} (t_2-t_{1})^{1-\alpha} \cdots (t_n-t_{n-1})^{1-\alpha}
	(t-t_n)^{1-\alpha}}
	\prod_{i=1}^n \big( \hb \,\dd W_{t_i} + \hh\, \dd t_i \big) \,.
\end{split}
\end{equation}
These equalities should be understood in the a.s.\ sense, since stochastic
integrals are not defined pathwise. 
In the next result, of independent interest, we exhibit versions of the
continuum partition functions which are jointly continuous in $(t,\hh)$ and $(s,t,\hh)$.
As a matter of fact, we do not need this result in the sequel, so we only sketch its proof.

\begin{theorem}\label{Theorem676result}
Fix $\hb > 0$ and let $(\Yf(t))_{t\in [0,\infty)}$, $(\Y(s,t))_{(s,t) \in [0,\infty)^2_\le}$ 
be versions of \eqref{eq:YY} that are continuous
in $t$, resp.\ in $(s,t)$.
Then, for all $\hh \in \R$ and all $s \in [0,\infty)$, resp. $(s,t) \in [0,\infty)^2_\le$,
\begin{equation}\label{eq:version}
	\Zf(t) \overset{(a.s.)}{=} \Yf(t)
	+ \sum_{k=1}^\infty 
	C_\alpha^k \, \hh^k 
	\left( \,\, \idotsint\limits_{0 < t_1 < t_2 < \ldots < t_k < t}   \!\!
	\frac{\Y(0,t_1)}{t_1^{1-\alpha}} \, 
	\frac{\Y(t_1,t_2)}{(t_2-t_1)^{1-\alpha}} \cdots 
	\frac{\Y(t_{k-1},t_k)}{(t_k-t_{k-1})^{1-\alpha}} \,
	\prod_{i=1}^k  \dd t_i \right) \,,
\end{equation}
\begin{equation}\label{eq:version2}
\begin{split}
	& \Z(s,t) \overset{(a.s.)}{=} \Y(s,t) \, + \, (t-s)^{1-\alpha} \times \\
	& \quad \ \times
	\sum_{k=1}^\infty 
	C_\alpha^k \, \hh^k 
	\left( \,\, \idotsint\limits_{s < t_1 < t_2 < \ldots < t_k < t}   \!\!
	\frac{\Y(s,t_1)}{(t_1-s)^{1-\alpha}} \, 
	\frac{\Y(t_1,t_2)}{(t_2-t_1)^{1-\alpha}} \cdots 
	\frac{\Y(t_{k-1},t_k)}{(t_k-t_{k-1})^{1-\alpha}} \,
	\frac{\Y(t_{k},t)}{(t-t_{k})^{1-\alpha}} \,
	\prod_{i=1}^k  \dd t_i \right) \,.
\end{split}
\end{equation}
The right hand sides 
of \eqref{eq:version}, \eqref{eq:version2} are versions
of the continuum partition functions \eqref{eq:Wiener}, \eqref{eq:Wiener2} that are jointly continuous 
in $(t,\hh)$, resp.\ in $(s,t,\hh)$.
\end{theorem}

\begin{remark}\rm
The equalities \eqref{eq:version} and \eqref{eq:version2} hold on a set of probability $1$ which depends
on $\hh$. On the other hand, the right hand sides of these relations are continuous functions of $\hh$, 
for $W$ in a fixed set of probability $1$.
\end{remark}

\begin{proof}[Proof (sketch).]
We focus on \eqref{eq:version2}, since \eqref{eq:version} is analogous.
We rewrite the $n$-fold integral in \eqref{eq:Wiener2} expanding
the product of differentials in a binomial fashion,
obtaining $2^n$ terms. Each term contains $k$ ``deterministic variables'' $\dd t_i$
and $n-k$ ``stochastic variables'' $\dd W_{t_j}$, whose locations are intertwined.
If we relabel the deterministic variables as $u_1 < \ldots < u_k$,
performing the sum over $n$ in \eqref{eq:Wiener2} yields
\begin{equation*} 
	\Z(s,t) = 1 + (t-s)^{1-\alpha}\sum_{k=1}^\infty C_\alpha^k 
	\idotsint\limits_{s < u_1 < u_2 < \ldots < u_k < t}   \,
	A(s,u_1) A(u_1, u_2) \cdots A(u_{k-1}, u_k) A(u_k, t)
	\prod_{i=1}^k \hh \, \dd u_i \,,
\end{equation*}
where $A(u_m, u_{m+1})$ gathers the contribution of the integrals over the stochastic variables
$\dd W_{t_j}$ with indexes $t_j \in (u_m, u_{m+1})$, i.e. (relabeling such variables
as $t_1, \ldots, t_n$)
\begin{equation*}
\begin{split}
	A(a, b) = & \frac{1}{(b-a)^{1-\alpha}} + \\
	& + \sum_{n=1}^\infty \
	\idotsint\limits_{a < t_1 < t_2 < \ldots < t_n < b}   \,
	\frac{C_\alpha^n}
	{ (t_1-a)^{1-\alpha} (t_2-t_1)^{1-\alpha} \cdots (t_n-t_{n-1})^{1-\alpha}
	(b-t_n)^{1-\alpha}}
	\prod_{j=1}^n \hb \,\dd W_{t_j} \,.
\end{split}
\end{equation*}
A look at \eqref{eq:Wiener2} shows that $A(a, b) = \frac{1}{(b-a)^{1-\alpha}}
\Zo(s,t) = \frac{1}{(b-a)^{1-\alpha}} \Y(s,t)$, proving \eqref{eq:version2}.


Since the process $\Y(s,t)$ is continuous by assumption,
it is locally bounded and consequently the series in \eqref{eq:version2} converges by
the upper bound in \cite[Lemma C.1]{CRZ14} (that we already used in \eqref{eq:ineq}).
The continuity 
of the right hand side of \eqref{eq:version2}
in $(s,t,\hh)$ is then easily checked.
\end{proof}

\medskip
\noindent
\textit{Step 3. Proof of \eqref{eq:Sk2}.}
We now prove \eqref{eq:Sk2}, focusing on the second relation, since the first
one is analogous. We are going to prove it with $\Z(s,t)$ \emph{defined}
as the right hand side of \eqref{eq:version2}.

Since $e^{h\ind_{n\in\tau}} = 1 + (e^h-1)\ind_{n\in\tau}$,
a binomial expansion yields
\begin{equation} \label{eq:plug}
	e^{h\sum_{n=q+1}^{r-1}\ind_{n\in\tau}} =
	\prod_{n=q+1}^{r-1} e^{h\ind_{n\in\tau}} =
	1 + \sum_{k=1}^{r-q-1}\displaystyle\sum_{\small{q+1\leq n_1<\cdots<n_k\leq r-1}} (e^h-1)^k\, 	
	\ind_{n_1\in\tau}\cdots \ind_{n_k\in\tau} \,.
\end{equation}
We now want to plug \eqref{eq:plug} into \eqref{discrcondfor}.
Setting $n_0 := r$, we can write (in analogy with \eqref{eq:facto})
\begin{equation*}
\begin{split}
	&\mathrm{E}\left(\left.e^{\sum_{k=q+1}^{r-1}(\beta\omega_k-\Lambda(\beta))\ind_{k\in\tau}} 
	\ind_{n_1\in\tau}\cdots \ind_{n_k\in\tau} 
	\, \right| q\in\tau, r\in\tau\right) \\
	=&
	\left(\prod_{i=1}^k e^{\beta \omega_{n_i}-\Lambda(\beta)}
	\mathrm{Y}_\beta^{\omega,\mathtt{c}}(n_{i-1},n_i)\right)
	\frac{\mathrm{Y}_\beta^{\omega,\mathtt{c}}(n_{k},r)}{
	\mathrm{Y}_\beta^{\omega,\mathtt{c}}(q,r)}\,
	\left(\prod_{i=1}^k u(n_i-n_{i-1})\right)
	\frac{u(r-n_k)}{u(r-q)}\,,
\end{split}
\end{equation*}
where we recall that $\mathrm{Y}_{\beta}^{\omega,\mathtt{c}} :=\mathrm{Z}_{\beta,0}^{\omega,\mathtt{c}}$, cf.\eqref{eq:YY}. For brevity we set
\begin{equation}\label{eq:Qdisc}
	Q_\beta^\omega(a,b) := e^{\beta \omega_a - \Lambda(\beta)}
	\, \mathrm{Y}_{\beta}^{\omega,\mathtt{c}}(a,b) \,.
\end{equation}
Then, plugging \eqref{eq:plug} into \eqref{discrcondfor}, we obtain
a discrete version of \eqref{eq:version2}:
\begin{equation} \label{eq:multiple}
\begin{split}
	& \mathrm{Z}_{\beta,h}^{\omega,\mathtt{c}}(q,r)=
	\mathrm{Y}_{\beta}^{\omega,\mathtt{c}}(q,r) \\
	& \quad + \sum_{k=1}^{r-q-1}(e^h-1)^k \!\!\!\!
	\sum_{\small{q+1\leq n_1<\cdots<n_k\leq r-1}} 
	\left( \prod_{i=1}^{k}Q(n_{i-1},n_i) \right) \frac{Q_\beta^\omega(n_k, r)}
	 {Q_{\beta}^{\omega,\mathtt{c}}(q,r)} \left(\prod_{i=1}^k u(n_i-n_{i-1})\right)
	\frac{u(r-n_k)}{u(r-q)}\,.
\end{split}
\end{equation}

\smallskip

We are now ready to prove \eqref{eq:Sk2}. For this purpose we are
going to use an analogous argument as in 
Theorem~\ref{ThmUConv}: it will be necessary and sufficient to prove that, $\bbP(\dd\omega,\dd W)$-a.s.,
for any convergent 
sequence $(s_N,t_N,\hh_N)_{N\in\mathbb{N}} \to (s_{\infty},t_{\infty},\hh_{\infty})$ 
in $[0,T]_{\leq}^2\times [0,M]$ one has
\begin{equation}\label{eq:toprr}
	\lim_{N\to\infty}\left|\mathrm{Z}_{\beta_N,\tilde{h}_N}^{\omega,\mathtt{c}}(Ns_N,Nt_N)
	-\mathbf{Z}_{\hb,\hh_N}^{W,\mathtt{c}}(s_N,t_N)\right|=0
\end{equation}
where $h_N=\hh_N\, L(N)N^{-\alpha}$. 
Recall that we have fixed a coupling under which $\mathrm{Y}_{\beta_N}^{\omega,\mathtt{c}}(Ns,Nt)$ 
converges uniformly to $\Y(s,t)$, $\mathbb{P}$-a.s. (cf.\ \eqref{eq:Sk}).
Borel-Cantelli estimates ensure that $\max_{a \le N} |\omega_a| = O(\log N)$
$\mathbb{P}$-a.s., by \eqref{assD1}, hence $Q_{\beta_N}^\omega(Ns,Nt)$ also
converges uniformly to $\Y(s,t)$, $\mathbb{P}$-a.s.. We call this event of probability one 
$\Omega_{\mathrm{Y}}$ and in the rest of the proof we work on that event,
proving \eqref{eq:toprr}.

It is not restrictive to assume $Ns_N\, Nt_N\in \mathbb{N}_0$.
Then
we rewrite \eqref{eq:multiple} with $q=Ns_N,\, r=Nt_N$ as a Riemann sum:
setting $t_0=s_N,\, t_{k+1}=t_N$,
\begin{equation} \label{eq:multiple1}
\begin{split}
	& \mathrm{Z}_{\beta_N,h_N}^{\omega,\mathtt{c}}(Ns_N,Nt_N)=
	\mathrm{Y}_{\beta_N}^{\omega,\mathtt{c}}(Ns_N,Nt_N) \\
	& \ + \sum_{k=1}^{N(t_N-s_N)-1} \left(\frac{e^{h_N}-1}{h_N}\right)^{k} \,
	\left\{ \frac{1}{N^k} \sumtwo{t_1, \ldots, t_k \in \frac{1}{N}\N_0}{s_N<t_1<\cdots<t_k<t_N} 
	\frac{\prod_{i=1}^{k+1}
	\left\{Q_{\beta_N}^\omega(Nt_{i-1},Nt_i) \, (N \, h_N)
	\, u(Nt_i-Nt_{i-1}) \right\}}
	{Q_{\beta_N}^{\omega,\mathtt{c}}(Ns,Nt) \, (N \, h_N) \, u(Nt_N-Ns_N)} \right\} .
\end{split}
\end{equation}
Observe that $N \, h_N = \hh_N\, L(N) N^{1-\alpha} \sim \hh_\infty \, L(N) N^{1-\alpha}$.
Recalling \eqref{assRP1},
on the event $\Omega_{\textrm{Y}}$ we have
\begin{equation}\label{eq:Qdisc1}
	\lim_{N\to\infty} Q_{\beta_N}^\omega(Nx, Ny) \, 
	(Nh_N)\, u(\lceil Ny\rceil - \lceil Nx\rceil)
	= \hh_{\infty}C_{\alpha}\frac{\Y(x,y)}{(y-x)^{1-\alpha}}\,\quad \forall\, 0\leq x<y<\infty,
\end{equation}
and for any $\epsilon>0$ the convergence is uniform on $y-x\geq \epsilon$. 
Then, for fixed $k\in\mathbb{N}$, the term in brackets in
\eqref{eq:multiple1} converges to the corresponding integral in \eqref{eq:version2}, by
Riemann sum approximation, because the contribution to the sum given by $t_i-t_{i-1}<\epsilon$ 
vanishes as $\epsilon\to 0$.
This claim follows by using Potter's bounds as in  \eqref{eq:bounddo}, with 
$W_N(r,s)=L(N)N^{1-\alpha}u(\lceil Nr \rceil-\lceil Ns \rceil)$, and the uniform 
convergence of $Q_{\beta_N}^\omega(Ns, Nt)$ which provides for any $\eta>0$ a random 
constant $C_{\eta,T}\in (0,\infty)$ such that for all $N\in\N$ and for all $0\leq x < y\leq T$
\begin{equation}
\frac{C_{\eta, T}^{-1}}{(y-x)^{1-\alpha-\eta}} \le
Q_{\beta_N}^\omega(Nx, Ny) \, (N\, h_N)\, u(\lceil Ny\rceil - \lceil Nx\rceil) 
\leq \, \frac{C_{\eta, T}}{(y-x)^{1-\alpha+\eta}} .
\end{equation}
Therefore the contribution of the terms $t_i-t_{i-1}<\epsilon$ in the brackets of \eqref{eq:multiple1}   is estimated by
\begin{equation*}
	\idotsint\limits_{\small{s_N<t_1<\cdots<t_k<t_N}}
	\frac{C_{\eta, T}^{k+2} (s_N-t_N)^{1-\alpha-\eta}}{(t_1-s_N)^{1-\alpha-\eta}(t_2-t_1)^{1-\alpha+\eta}\cdots 
	(t_N-t_k)^{1-\alpha+\eta}}\ind_{\{t_i - t_{i-1} \le \epsilon,
        \, \text{for some} \, i =1,\cdots, k\}}
\prod_{i=1}^k \dd {t_i}.
\end{equation*}
For any fixed $k\in\mathbb{N}$ once chosen $\eta\in (0,\alpha)$ this integral vanishes as 
$\epsilon \to 0$ (recall that $(s_N,t_N) \to (s_\infty, t_\infty) \in [0,T]^2_\le$).
To get the convergence of the whole sum \eqref{eq:multiple1} we show that the 
contribution of the terms $k\geq M$ in \eqref{eq:multiple1} can be made arbitrarily small uniformly 
in $N$, by taking $M$ large enough. This follows by the same bound as in \eqref{eq:ineq},
as the term in brackets in \eqref{eq:multiple1} is bounded by
\begin{equation*}
\begin{split}
	&\idotsint\limits_{\small{s_N<t_1<\cdots<t_k<t_N}}
	\frac{C_{\eta, T}^{k+2} (s_N-t_N)^{1-\alpha-\eta}}
	{(t_1-s_N)^{1-\alpha+\eta}(t_2-t_1)^{1-\alpha+\eta}\cdots (t_N-t_k)^{1-\alpha+\eta}} \dd t_1
	\cdots \dd t_k\\
	=&\idotsint\limits_{\small{0<u_1<\cdots<u_k<1}}\frac{C_{\eta,T}^{k+2} 
	(t_N-s_N)^{k(\alpha-\eta)-2\eta}}{u_1^{1-\alpha+\eta}(u_2-u_1)^{1-\alpha+\eta}\cdots 
	(1-u_k)^{1-\alpha+\eta}} \dd u_1\cdots \dd u_k\, \leq ( \hat{C}_{\eta,T})^kc_1 e^{-c_2 k\log k} ,
	\end{split}
\end{equation*}
for some constant $\hat{C}_{\eta,T}\in (0,\infty)$,
cf.\ \cite[Lemma B.3]{CRZ14}. This completes the proof.\qed

\section{Proof of Theorem~\ref{Thm2}}\label{secproofthm2}

In this section we prove Theorem \ref{Thm2}.
Most of our efforts are devoted to proving the key relation \eqref{eq:key},
through a fine comparison of the discrete and continuum
partition functions, based on a coarse-graining procedure.
First of all, we (easily) deduce \eqref{eq:keyhc} from \eqref{eq:key}.

\subsection{Proof of relation \eqref{eq:keyhc} assuming  \eqref{eq:key}}
\label{sec:hcfe}
We set $\hb = 1$ and we use \eqref{eqParam}-\eqref{eq:hb} 
(with $\epsilon = \frac{1}{N}$)
to rewrite \eqref{eq:key} as follows: 
for all $\hh\in \R$, $\eta > 0$ there exists $\beta_0 > 0$ such that
\begin{equation} \label{eq:rewrite}
{\fhc}\left(1,\hh - \eta\right)\leq 
\frac{\mathrm{F} \big( \beta, \, \hh \, \tilde{L}_\alpha (\frac{1}{\beta})
\, \beta^{\frac{2\alpha}{2\alpha-1}})}{\tilde{L}_\alpha(\frac{1}{\beta})^2 \,
\beta^{\frac{2}{2\alpha-1}}} \leq
{\fhc}\left(1,\hh + \eta\right), \qquad \forall \beta \in (0,\beta_0)\,.
\end{equation}
If we take $\hh := \hcc(1) - 2\eta$, then $\fhc(1,\hh + \eta)=0$ by
the definition \eqref{criticalcurvecont} of $\hcc$. Then \eqref{eq:rewrite}
yields $\mathrm{F}\left(\beta, \hh \, \tilde{L}_\alpha(\frac{1}{\beta})
\, \beta^{\frac{2\alpha}{2\alpha-1}}\right)=0$ for $\beta < \beta_0$,
that is $h_c(\beta) \ge \hh \, \tilde{L}_\alpha(\frac{1}{\beta})
\, \beta^{\frac{2\alpha}{2\alpha-1}}$ by the definition \eqref{eq:hc}
of $h_c$, hence
\begin{equation*}
	\liminf_{\beta \to 0} \frac{h_c(\beta)}{\tilde{L}_\alpha(\frac{1}{\beta})
	\, \beta^{\frac{2\alpha}{2\alpha-1}}} \ge \hh = \hcc(1) - 2\eta \,.
\end{equation*}
Letting $\eta \to 0$ proves ``half'' of \eqref{eq:keyhc}.
The other half follows along the same line, choosing
$\hh := \hcc(1) + 2\eta$ and using the first inequality in \eqref{eq:rewrite}.\qed


\subsection{Renewal process and regenerative set}\label{closedsetsec}
Henceforth we devote ourselves to the proof of relation \eqref{eq:key}.
For $N\in\mathbb{N}$ we consider the \emph{rescaled renewal process}
$$
\frac{\tau}{N}=\left\{\frac{\tau_i}{N}\right\}_{i\in\mathbb{N}}
$$ 
viewed as a random subset of $[0,\infty)$. As $N\to\infty$, 
under the original law $\P$, the
random set $\tau/N$ converges in distribution
to a universal random closed set $\btau$, the so-called $\alpha$-stable regenerative set.
We now summarize the few properties of $\btau$ that will be needed in the sequel,
referring to \cite[Appendix A]{CRZ14} for more details.

Given a closed subset $C\subseteq \mathbb{R}$ and a point $t\in\mathbb{R}$, we define
\begin{equation}\label{gdeq}
\mathtt{g}_t(C):=\sup\left\{x\mid x\in C\cap [-\infty,t)\right\}, \quad \mathtt{d}_t(C):=
\inf\left\{x\mid x\in C\cap [t,\infty)\right\}.
\end{equation}
A key fact is that as $N\to\infty$ 
the process $((\mathtt{g}_t(\tau/N),\mathtt{d}_t(\tau/N))_{t \in [0,\infty)}$
converges in the sense of finite-dimensional distribution to
$((\mathtt{g}_t(\btau),\mathtt{d}_t(\btau))_{t \in [0,\infty)}$
(see \cite[Appendix A]{CRZ14}).

Denoting by $\P_x$ the law of the regenerative set started at $x$, that is 
$\P_x(\btau\in\cdot):=\P(\btau+x\in \cdot)$, 
the joint distribution $(\mathtt{g}_t(\btau),\mathtt{d}_t(\btau))$ is
\begin{align} \label{gdlaw}
\frac{\P_{x}\left(\mathtt{g}_t(\btau) \in \dd u,\mathtt{d}_t(\btau) \in \dd v\right)}
{\dd u \, \dd v} = C_{\alpha}\frac{\ind_{u\in (x,t)}
\ind_{v\in (t,\infty)}}{(u-x)^{1-\alpha}(v-u)^{1+\alpha}},
\end{align}
where $C_{\alpha}=\frac{\alpha\sin(\pi\alpha)}{\pi}$. We can deduce
\begin{gather}\label{glaw}
\frac{\P_{x}\left(\mathtt{g}_t(\btau) \in \dd u\right)}{\dd u}
= \frac{C_{\alpha}}{\alpha}\frac{\ind_{u\in(x,t)}}{(u-x)^{1-\alpha}(t-u)^{\alpha}} \,,
\\
\label{dlaw}
\frac{\P_{x}\left( \left.\mathtt{d}_t(\btau) \in \dd v
\,\right|\, \mathtt{g}_t(\btau) = u \right)}{\dd v}
= \frac{\alpha \, (t-u)^\alpha}{(v-u)^{1+\alpha}} \, \ind_{v\in(t,\infty )} \,.
\end{gather} 

Let us finally state the \emph{regenerative property} of $\btau$.
Denote by $\mathcal{G}_u$ the filtration generated by $\btau\cap [0,u]$
and let $\sigma$ be a $\{\mathcal{G}_u\}_{u\geq 0}$-stopping time such that 
$\P(\sigma\in\btau)=1$ (an example is $\sigma = \mathtt{d}_t(\btau)$).
Then the law of $\btau \cap [\sigma,\infty)$
conditionally on $\mathcal{G}_\sigma$ equals $\P_x |_{x=\sigma}$,
i.e.\, the translated random set $(\btau-\sigma)\cap [0,\infty)$ 
is independent of $\mathcal{G}_{\sigma}$ and it is distributed as
the original $\btau$ under $\P = \P_0$. 


%
%

\subsection{Coarse-grained decomposition}\label{courgradec}

We are going to express the discrete and continuum partition functions 
in an analogous way, in terms of the random sets $\tau/N$ and $\btau$, respectively.

We partition $[0,\infty)$ in intervals 
of length one, called blocks. For a given random set $X$ ---
it will be either the rescaled renewal process
$\tau/N$ or the regenerative set $\btau$ --- we look at the \emph{visited blocks}, i.e.\ those blocks
having non-empty intersection with $X$.
More precisely, we write $[0,\infty)=\bigcup_{k=1}^{\infty}B_k$ , where $B_k=[k-1,k)$,
and we say that a block $B_k$ is visited if $X \cap B_k \ne \emptyset$. If we define
\begin{equation}\label{eq:JJ}
\JJo_{1}(X) := \min\{j>0 : B_j\cap X\neq \emptyset\} \,, \qquad
\JJo_{k}(X) := \min\{j>\JJo_{k-1} : B_j\cap X\neq \emptyset\} \,,
\end{equation}
the visited blocks are $\big(B_{\JJo_k(X)}\big)_{k\in\mathbb{N}}$.
The last visited block before $t$ is $B_{\mmo(X)}$, where we set
\begin{equation}\label{mmeqdef}
\mmo(X) := \sup\{k > 0:\JJo_k(X) \leq t\} \,.
\end{equation}
We call $\sso_k(X)$ and $\tto_k(X)$ the first and last visited points
in the block $B_{\JJo_k(X)}$, i.e. (recalling \eqref{gdeq})
\begin{equation}\label{eq:sstt}
\sso_{k}(X) := \inf\{x\in X\cap B_{\JJo_k}\} = \mathtt{d}_{\JJo_k-1}(X) \,, \qquad
\tto_k(X) := \sup\{x\in X\cap B_{\JJo_k}\} = \mathtt{g}_{\JJo_k}(X) \,.
\end{equation}
(Note that $\JJo_k(X) = \lfloor \sso_k(X)\rfloor = \lfloor \tto_k(X)\rfloor$ can be recovered
from $\sso_{k}(X)$ or $\tto_k(X)$; analogously, $\mmo(X)$ can be recovered from
$(\JJo_k(X))_{k\in\N}$; however, it will be practical to use $\JJo_k(X)$ and $\mmo(X)$.)

\begin{definition}\label{def:Jstm}
The random variables $\left(\JJo_k(X), \sso_k(X), \tto_k(X)\right)_{k\in\mathbb{N}}$
and $(\mmo(X))_{t\in\N}$
will be called the \emph{coarse-grained decomposition} of the random set $X \subseteq [0,\infty)$.
In case $X = \btau$ we will simply
write $\left(\JJ_k, \ss_k, \tt_k\right)_{k\in\mathbb{N}}$ and $(\mm)_{t\in\N}$,
while in case $X = \tau/N$ we will write $(\rJJ_k,\rss_k, \rtt_k)_{k\in\mathbb{N}}$ 
and $(\rmm)_{t\in\N}$.
\end{definition}

\begin{figure}[h]
\centering
\includegraphics[scale=0.8]{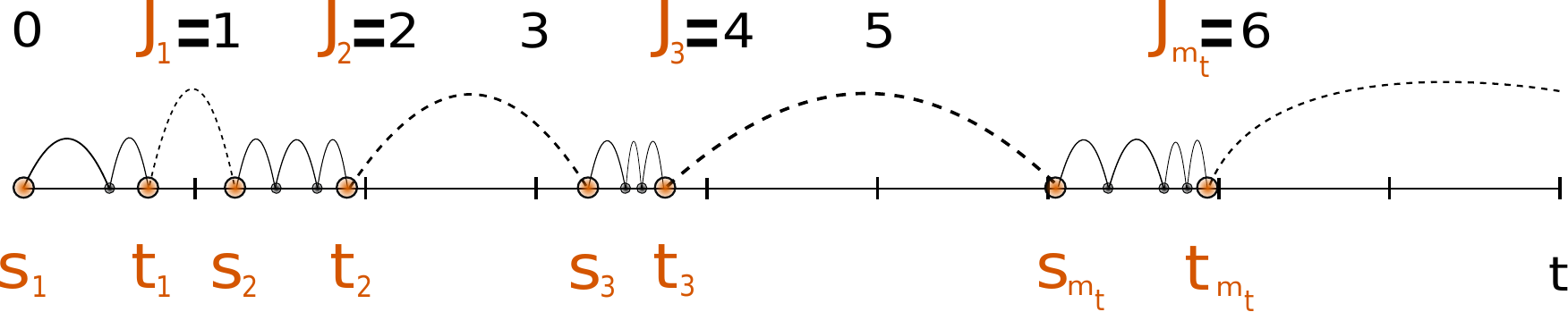}
\caption{In the figure we have pictured a random set $X$, given as the zero level set of a stochastic process, whose excursions are represented by the semi-arcs (dotted arcs represents excursions between two consecutive visited blocks). The coarse-grained decomposition of $X$ is given by the first and last points -- $\sso_k(X), \tto_k(X)$ -- inside each visited block $[\JJ_k-1(X),\JJ_k(X))$, marked by a big dot in the figure. By construction, between visited blocks there are no points of $X$; all of its points are contained in the set $\cup_{k\in\N}\Big[\sso_k(X),\tto_k(X)\Big]$. }
\label{intro:figcoarsegraining}
\end{figure}

\begin{remark}\rm
For every $t\in\N$, one has the convergence in distribution
\begin{equation}\label{eq:codi}
\big(\rmm, (\rss_k, \rtt_k)_{1 \le k \le \rmm}\big) \xrightarrow[N\to\infty]{d}
\big(\mm, (\ss_k, \tt_k)_{1 \le k \le \mm}\big) \,,
\end{equation}
thanks to the convergence in distribution of 
$(\mathtt{g}_s(\tau/N),\mathtt{d}_s(\tau/N))_{s\in\N}$
toward $(\mathtt{g}_s(\btau),\mathtt{d}_s(\btau))_{s\in\N}$.
\end{remark}

Using \eqref{gdlaw} and the regenerative property, one can write
explicitly the joint density of $\JJ_k, \ss_k, \tt_k$. This yields
the following estimates of independent
interest, proved in Appendix~\ref{sec:lemmaP2}.

\begin{lemma}\label{th:lemmaP2}
For any $\alpha \in (0,1)$
there are constants $A_\alpha, B_\alpha \in (0,\infty)$ such that for all $\gamma \ge 0$
\begin{gather}\label{lemmaP2}
\sup_{(x,y)\in [0,1]_{\leq}^2}
\P_x\left(\left. \tt_2\in [\JJ_2-\gamma,\JJ_2]\,\right\vert\, \tt_{1}=y\right)\leq A_\alpha 
\, \gamma^{1-\alpha} \,, \\
\label{eq:LUC1}
 \sup_{(x,y)\in [0,1]_{\leq}^2} \P_{x}\left(\left. 
 \tt_2-\ss_2\leq \gamma \, \right\vert\, \tt_1=y\right) \leq B_\alpha \, \gamma^{\alpha} \,,
\end{gather}
where $\P_{x}$ is the law of the $\alpha$-stable regenerative set starting from $x$.
\end{lemma}

\smallskip

We are ready to express the partition functions $\Znf(Nt)$ and $\Zf(t)$
in terms of the random sets $\tau/N$ and $\btau$, 
through their coarse-grained decompositions. 
Recall that $\beta_N, h_N$ are linked to $N$ and $\hb, \hh$ by \eqref{eqParam}.
For notational lightness,
we denote by $\E$ the expectation
with respect to either $\tau/N$ or $\btau$.

\begin{theorem}[Coarse-grained Hamiltonians]
For $t\in\mathbb{N}$ we can write the discrete and continuum partition functions as follows:
\begin{equation}\label{HVCF}
\Znf(Nt) = \E \left[ e^{\mathrm{H}_{N,t;\hb, \hh}^\omega(\tau/N)} \right] \,, \qquad
\Zf(t) = \E \left[ e^{\mathbf{H}_{t;\hb,\hh}^W(\btau)} \right] \,,
\end{equation}
where the \emph{coarse-grained Hamiltonians} $\mathrm{H}(\tau/N)$ and $\mathbf{H}(\btau)$
depend on the random sets $\tau/N$ and $\btau$ only through their coarse-grained
decompositions, and are defined by
\begin{equation}\label{Hm}
\mathrm{H}_{N,t;\hb, \hh}^\omega(\tau/N)
:= \sum_{k=1}^{\rmm}\log\Zn(N\rss_k, N\rtt_k) \,, \qquad
\mathbf{H}^W_{t; \hb,\hh}(\btau)=\sum_{k=1}^{\mm}\log\Z(\ss_k,\tt_k) \,.
\end{equation}
\end{theorem}

\begin{proof}
Starting from the definition \eqref{model}
of $\Znf(Nt)$, we disintegrate according
to the random variables $\rmm$ and $(\rss_k, \rtt_k)_{1 \le k \le \rmm}$.
Recalling \eqref{discrcondfor},
the renewal property of $\tau$ yields
\begin{align}\label{HVDF}
\Znf(Nt)=&\E\left[ \Zn(0,N\rtt_1) \, \Zn(N\rss_2, N\rtt_2)\cdots
\Zn(N\rss_{\rmm},N\rtt_{\rmm}) \right] \,,
\end{align}
which is precisely the first relation in \eqref{HVCF}, with $\mathrm{H}$
defined as in \eqref{Hm}.

The second relation in \eqref{HVCF} can be proved with analogous arguments,
by the regenerative property of $\btau$. Alternatively,
one can exploit the convergence in distribution \eqref{eq:codi},
that becomes a.s.\ convergence under a suitable coupling of
$\tau/N$ and $\btau$; since $\Zn(Ns,Nt) \to \Z(s,t)$
uniformly for $0 \le s \le t \le T$,
under a coupling of $\omega$ and $W$ (by Theorem~\ref{PropByProd}),
letting $N\to\infty$ in \eqref{HVDF} 
yields, by dominated convergence, the second relation in \eqref{HVCF}, with $\mathbf{H}$
defined as in \eqref{Hm}.
\end{proof}

The usefulness of the representations in \eqref{HVCF} is that they express the discrete
and continuum partition functions in closely analogous ways,
which behave well in the continuum limit $N\to\infty$.
To appreciate this fact,
note that although the discrete partition function is expressed through
an Hamiltonian of the form $\sum_{n=1}^N (\beta \omega_n - \Lambda(\beta) + h) \ind_{\{n\in\tau\}}$, 
cf.\ \eqref{model}, such a ``microscopic'' Hamiltonian
admits no continuum analogue, because
the continuum disordered pinning model studied in \cite{CRZ14} is \emph{singular} 
with respect to the regenerative set $\btau$, cf.\ \cite[Theorem 1.5]{CRZ14}.
The ``macroscopic'' coarse-grained Hamiltonians in \eqref{Hm}, on the other hand,
will serve our purpose.

\subsection{General Strategy}
\label{sec:strate}

We now describe a general strategy to prove the key relation \eqref{eq:key} of
Theorem~\ref{Thm2}, exploiting the representations in \eqref{HVCF}.
We follow the strategy developed for the copolymer model
in \cite{BdH97,CG10}, with some simplifications and strengthenings.

\begin{definition}
Let $f_{t}(N,\hb,\hh)$ and $g_{t}(N,\hb,\hh)$ be two real functions
of $t, N\in\N$, $\hb > 0$, $\hh\in\R$.
We write $f\prec g$ if for 
all fixed $\hb, \hh, \hh'$ with $\hh <\hh'$ there exists $N_0(\hb,\hh,\hh') < \infty$ 
such that for all $N>N_0$ 
\begin{equation}
\begin{split}
&\limsup_{t\to\infty}f_{t}(N,\hb,\hh)\leq \limsup_{t\to\infty}g_{t}(N,\hb,\hh'),\\
&\liminf_{t\to\infty}f_{t}(N,\hb,\hh)\leq \liminf_{t\to\infty}g_{t}(N,\hb,\hh').
\end{split}
\end{equation}
where the limits are taken along $t\in\N$.
If both $f\prec g$ and $g\prec f$ hold, then we write $f\simeq g$.
\end{definition}
Keeping in mind \eqref{DFE} and \eqref{CFcE}, we define $f^{(1)}$ and $f^{(3)}$ respectively 
as the continuum and discrete (rescaled) finite-volume free energies, averaged over the disorder:
\begin{align}
\label{fvdfr}
&f_{t}^{(1)}(N,\hb,\hh) 
:=\frac{1}{t}\mathbb{E}\left(\log\Zf (t)\right)\,,\\
\label{fvcfr}
&f_{t}^{(3)}(N,\hb,\hh) 
:=\frac{1}{t}\mathbb{E}\left(\log\Znf(Nt)\right)  \,.
\end{align}
(Note that $f^{(1)}$ does not depend on $N$.)
\emph{Our goal is to prove that $f^{(3)}\simeq f^{(1)}$},
because this yields the key relation \eqref{eq:key} in Theorem~\ref{Thm2},
and also the existence of the 
averaged continuum free energy as $t\to\infty$ along $t\in\N$ 
(thus proving part of Theorem~\ref{Thm1}). Let us start checking these claims.

\begin{lemma}\label{th:impex}
Assuming $f^{(3)}\simeq f^{(1)}$, the following limit exists
along $t\in\N$ and is finite:
\begin{equation} \label{eq:impex}
	\fhc(\hb,\hh) := \lim_{t\to\infty} f_{t}^{(1)}(N,\hb,\hh) =
	\lim_{t\to\infty}
	\frac{1}{t}\mathbb{E}\left(\log\Zf (t)\right) \,.
\end{equation}
\end{lemma}

\begin{proof}
The key point is that 
$f_{t}^{(3)}$ admits a limit as $t\to\infty$: by
\eqref{DFE}, for all $N\in\N$ we can write
\begin{equation}\label{eq:truelimit}
\lim_{t\to\infty} f_{t}^{(3)}(N,\hb,\hh)
=N\,\mathrm{F}(\beta_N,h_N) \,
\end{equation}
where we 
agree that limits are taken along $t\in\N$.
For every $\epsilon > 0$, the relation $f^{(3)}\simeq f^{(1)}$ yields
\begin{equation} \label{eq:131}
	\limsup_{t\to\infty} f_{t}^{(1)}(N,\hb,\hh - 2\epsilon)
	\le \lim_{t\to\infty} f_{t}^{(3)}(N,\hb,\hh - \epsilon)
	\le \liminf_{t\to\infty} f_{t}^{(1)}(N,\hb,\hh) \,,
\end{equation}
for $N\in\N$ large enough (depending on $\hb, \hh$ and $\epsilon$).
Plugging the definition \eqref{fvdfr} of $f_{t}^{(1)}$,
which does not depend on $N\in\N$, into this relation, we get
\begin{equation}
\begin{split}
\limsup_{t\to\infty }\frac{1}{t}\mathbb{E}\left(\log\mathbf{Z}_{\hb,\hh - 2\epsilon}^W (t)\right)
\le 
\liminf_{t\to\infty }\frac{1}{t}\mathbb{E}\left(\log\mathbf{Z}_{\hb,\hh}^W (t)\right) .
\end{split}
\end{equation}
The left hand side of this relation is a convex function of $\epsilon \ge 0$
(being the $\limsup$ of convex functions, by Proposition~\ref{th:continuous})
and is finite (it is bounded by $N \,\mathrm{F}(\beta_N,h_N) < \infty$,
by \eqref{eq:truelimit} and \eqref{eq:131}). It follows that it is a continuous
function of $\epsilon \ge 0$, so letting $\epsilon \downarrow 0$ completes the proof.
\end{proof}

\begin{lemma}
Assuming $f^{(3)}\simeq f^{(1)}$, relation \eqref{eq:key} in Theorem~\ref{Thm2}
holds true.
\end{lemma}

\begin{proof}
We know that $\lim_{t\to\infty} f_{t}^{(1)}(N,\hb,\hh) = \fhc(\hb,\hh)$
by Lemma~\ref{th:impex}. Recalling \eqref{eq:truelimit}, 
relation $f^{(3)}\simeq f^{(1)}$ 
can be restated as follows: for all $\hb > 0$,
$\hh \in \R$ and $\eta > 0$ there exists 
$N_0 < \infty$ such that
\begin{equation*}
	\fhc(\hb,\hh - \eta) \le 
	N\,\mathrm{F}\left( \hb \frac{L(N)}{N^{\alpha-\frac{1}{2}}},
	\hh \frac{L(N)}{N^\alpha} \right) \le
	\fhc(\hb,\hh + \eta) \,, \qquad \forall N \ge N_0 \,.
\end{equation*}
Incidentally, this relation holds also when $N \in [N_0,\infty)$ is not an integer,
because the same holds for relation \eqref{eq:truelimit}.
Setting $\epsilon := \frac{1}{N}$ and $\epsilon_0 :=
\frac{1}{N_0}$ yields precisely relation \eqref{eq:key}.
\end{proof}

The rest of this section is devoted to proving $f^{(1)}\simeq f^{(3)}$.
By \eqref{fvdfr}-\eqref{fvcfr} and \eqref{HVCF}, we can write
\begin{equation}\label{eq:f12}
f_{t}^{(1)}(N,\hb,\hh) =\frac{1}{t}\mathbb{E}\left(\log
\E \left[ e^{\mathbf{H}_{t;\hb,\hh}^W(\btau)} \right]\right) \,, \qquad
f_{t}^{(3)}(N,\hb,\hh) =\frac{1}{t}\mathbb{E}\left(\log
\E \left[ e^{\mathrm{H}_{N,t;\hb, \hh}^\omega(\tau/N)} \right]\right) \,.
\end{equation}
Since relation $\simeq$ is transitive, it suffices to prove that
\begin{equation}
f^{(1)}\simeq f^{(2)}\simeq f^{(3)} \,,
\end{equation}
for a suitable intermediate quantity $f^{(2)}$ which somehow interpolates between
$f^{(1)}$ and $f^{(3)}$. We define $f^{(2)}$ replacing the
rescaled renewal $\tau/N$ by the regenerative set $\btau$ in $f^{(3)}$:
\begin{equation}\label{HVDFs}
	f_{t}^{(2)}(N,\hb,\hh) 
	:=\frac{1}{t}\mathbb{E}\left(\log \E \left[
	e^{\mathrm{H}_{N,t; \hb, \hh}^\omega(\btau)} \right]\right) \,.
\end{equation}
Note that each function $f^{(i)}$, for $i=1,2,3$, is of the form 
\begin{equation}\label{Eqfs1}
f_{t}^{(i)}(N,\hb,\hh)=\frac{1}{t}\mathbb{E}
\left( \log\E\left[e^{\mathrm{H}_{N,t;\hb,\hh}^{(i)}}\right] \right),
\end{equation}
for a suitable Hamiltonian $\mathrm{H}_{N,t;\hb,\hh}^{(i)}$.
We recall that $\bbE$ is expectation with respect to the disorder
(either $\omega$ or $W$) while $\E$ is expectation with
respect to the random set (either $\tau/N$ or $\btau$).

\smallskip

The general strategy to
to prove $f^{(i)}\prec f^{(j)}$ can be described as follows ($i=1,j=2$ for clarity).
For fixed $\hb, \hh, \hh'$ with $\hh < \hh'$,
we couple the two Hamiltonians $\mathrm{H}_{N,t;\hb,\hh}^{(1)}$ and 
$\mathrm{H}_{N,t;\hb,\hh'}^{(2)}$ (both with respect to the random set 
and to the disorder) and we define for $\epsilon \in (0,1)$
\begin{equation}\label{eqHam1}
\Delta_{N,\epsilon}^{(1,2)}(t) :=
\mathrm{H}_{N,t;\hb,\hh}^{(1)} - (1-\epsilon)\mathrm{H}_{N,\hb,\hh'}^{(2)}
\end{equation}
(we omit the dependence of $\Delta_{N,\epsilon}^{(1,2)}(t)$ on $\hb,\hh,\hh'$ for short).
H\"older's inequality then gives 
\begin{align*}
&\E\left(e^{\mathrm{H}_{N, t; \hb,\hh}^{(1)}}\right)\leq 
\E\left(e^{\mathrm{H}_{N, t; \hb,\hh'}^{(2)}}\right)^{1-\epsilon}
\E\left(e^{\frac{1}{\epsilon}\Delta_{N,\epsilon}^{(1,2)}(t)}\right)^{\epsilon} \,.
\end{align*} 
Denoting by $\lim^*_{t\to\infty}$ either 
$\liminf_{t\to\infty}$ or $\limsup_{t\to\infty}$
(or, for that matter, the limit of any convergent subsequence), 
recalling \eqref{Eqfs1} and applying Jensen's inequality leads to
\begin{gather*}
\lim^*_{t\to\infty}f_{t}^{(1)}(N,\hb,\hh)\leq 
(1-\epsilon)\lim^*_{t\to\infty}f_{t}^{(2)}(N,\hb,\hh')
+\epsilon \limsup_{t\to\infty}\frac{1}{t}\log \E\mathbb{E}
\left(e^{\frac{1}{\epsilon}\Delta^{(1,2)}_{N,\epsilon}(t)}\right) \,.
\end{gather*}
In order to prove $f^{(1)} \prec f^{(2)}$ it then suffices to show the following:
for fixed $\hb, \hh, \hh'$ with $\hh < \hh'$,
\begin{equation}\label{EqGS}
\begin{split}
\exists \epsilon \in (0,1), \ N_0 \in (0,\infty): \qquad 
&\limsup_{t\to\infty}\frac{1}{t}\log \E\mathbb{E}
\left(e^{\frac{1}{\epsilon}\Delta^{(1,2)}_{N,\epsilon}(t)}\right)\leq 0, 
\quad \forall N \ge N_0 \,.
\end{split}
\end{equation}
(Of course, $\epsilon$ and $N_0$ will depend on the fixed values of $\hb, \hh, \hh'$.)

We will give details only for the proof of $f^{(1)}\prec f^{(2)}\prec f^{(3)}$,
because with analogous arguments one proves 
$f^{(1)}\succ f^{(2)}\succ f^{(3)}$.
Before starting, we describe the coupling of the coarse-grained Hamiltonians.

\begin{remark}\label{rem:piecewise}
For technical convenience, instead of linearly interpolating 
the discrete partition functions when $Ns, Nt \not\in \N_0$,
it will be convenient in \S\ref{SECONDstepSec} to consider their piecewise constant extension
$\Zn(\lfloor Ns \rfloor, \lfloor Nt\rfloor)$.
Plainly, relation \eqref{eq:Sk2} still holds.
\end{remark}

\subsection{The coupling}\label{sec:coupling}

The coarse-grained Hamiltonians $\mathrm{H}$ and $\mathbf{H}$, defined in \eqref{Hm},
are functions of the disorders $\omega$ and $W$ and of the random sets $\tau/N$ and $\btau$.
We now describe how to couple the disorders
(the random sets will be coupled through Radon-Nikodym derivatives, cf.\ \S\ref{SECONDstepSec}).

Recall that $[a,b)^2_\le := \{(x,y): \ a \le x \le y \le b\}$.
For $n\in\N$, we let $\tZ^{(n)}_N$ and $\btZ^{(n)}$ denote the families
of discrete and continuum partition functions with endpoints in $[n,n+1)$:
\begin{equation*}
	\tZ^{(n)}_N := \left( \Zn(Ns,Nt) \right)_{(s,t) \in [n,n+1)^2_\le} \,,
	\qquad
	\btZ^{(n)} :=
	\left( \Z(s,t) \right)_{(s,t) \in [n,n+1)^2_\le} \,.
\end{equation*}
Note that both $(\tZ^{(n)}_N)_{n\in\N}$ and $(\btZ^{(n)})_{n\in\N}$ are i.i.d.\ sequences.
A look at \eqref{Hm} reveals that that the coarse-grained Hamiltonian $\mathrm{H}$ 
depends on the disorder $\omega$ only through $(\tZ^{(n)}_N)_{n\in\N}$,
and likewise $\mathbf{H}$ depends on $W$ only through $(\btZ^{(n)})_{n\in\N}$.
Consequently, to couple $\mathrm{H}$ and $\mathbf{H}$
it suffices to couple $(\tZ^{(n)}_N)_{n\in\N}$ and $(\btZ^{(n)})_{n\in\N}$,
i.e.\ to define a law for the joint sequence $\big((\tZ^{(n)}_N, \btZ^{(n)})\big)_{n\in\N}$.
\emph{We take this to be i.i.d.}: discrete and continuum partition functions are coupled 
independently in each block $[n,n+1)$.

 It remains to define a coupling for $\tZ^{(1)}_N$ and $\btZ^{(1)}$.
Throughout the sequel we fix $\hb > 0$ and $\hh, \hh' \in \R$ with
$\hh < \hh'$. We can then use the coupling provided by Theorem~\ref{PropByProd},
which ensures that relation \eqref{eq:Sk2} holds $\bbP(\dd\omega, \dd W)$-a.s., 
with $T=1$ and $M = \max\{|\hh|,|\hh'|\}$.

\subsection{First step: $f^{(1)}\prec f^{(2)}$}
Our goal is to prove \eqref{EqGS}.
Recalling \eqref{eqHam1}, \eqref{eq:f12} and \eqref{HVDFs}, as well as \eqref{Hm},
for fixed $\hb, \hh, \hh'$ with $\hh < \hh'$ we can write
\begin{equation}\label{Delta2}
\Delta_{N,\epsilon}^{(1,2)}(t)= \mathbf{H}_{t;\hb,\hh}^W(\btau)
- (1-\epsilon)\mathrm{H}_{N,t;\hb,\hh'}^W(\btau) =
\sum_{k=1}^{\mm}\log\frac{\Z(\ss_k,\tt_k)}{\Znt(N\ss_k, N\tt_k)^{1-\epsilon}},
\end{equation}
where we set $h'_N = \hh' L(N)/N^\alpha$ for short, cf.\ \eqref{eqParam}.
Consequently 
\begin{equation}\label{DeltaEq}
\E \mathbb{E}\left(e^{\frac{1}{\epsilon}\Delta_{N,\epsilon}^{(1,2)}(t)}\right)
= \E\left[ \prod_{k=1}^{\mm}f_{N,\epsilon}(\ss_k, \tt_k) \right] \,,
\quad \ \ \text{where} \quad
f_{N,\epsilon}(s,t) := \mathbb{E}
\left[\left(\frac{\Z(s,t)}{\Znt(Ns, Nt)^{1-\epsilon}}
\right)^{\frac{1}{\epsilon}}\right] \,,
\end{equation}
because discrete and continuum partition functions
are coupled independently in each block $[n,n+1)$, cf.\ \S\ref{sec:coupling},
hence the $\bbE$-expectation factorizes. (Of course,
$f_{N,\epsilon}(s,t)$ also depends on $\hb,\hh,\hh'$.)

Let us denote by $\cF_M=\sigma\left((\ss_i,\tt_i):i\leq M\right)$
the filtration generated by the first $M$ visited blocks.
By the regenerative property, the regenerative set $\btau$ starts afresh
at the stopping time $\ss_{k-1}$, hence
\begin{equation} \label{eq:itera}
	\E\big[f_{N,\epsilon}(\ss_k, \tt_k) \,|\, \cF_{k-1}\big] 
	= \E\big[f_{N,\epsilon}(\ss_k, \tt_k) \,|\, \ss_{k-1}, \tt_{k-1}\big] \,,
\end{equation}
where we agree that $\E[\,\cdot\,|\, \ss_0, \tt_0] := \E[\,\cdot\,]$.
Defining the constant
\begin{equation}\label{eqG1sec}
\Lambda_{N,\epsilon} :=
\sup_{k, \ss_{k-1}, \tt_{k-1}} \E\big[f_{N,\epsilon}(\ss_k, \tt_k) \,|\, \ss_{k-1}, \tt_{k-1}\big] \,,
\end{equation}
we have $\E\big[f_{N,\epsilon}(\ss_k, \tt_k) \,|\, \cF_{k-1}\big]  \le \Lambda_{N,\epsilon}$,
hence
$\E\left[ \prod_{k=1}^{M}f_{N,\epsilon}(\ss_k, \tt_k) \right] \le (\Lambda_{N,\epsilon})^M$
for every $M\in\N$, hence
\begin{equation} \label{eq:geomm}
\E \mathbb{E}\left(e^{\frac{1}{\epsilon}\Delta_{N,\epsilon}^{(1,2)}(t)}\right)
= \E\left[ \prod_{k=1}^{\mm}f_{N,\epsilon}(\ss_k, \tt_k) \right] 
\le \sum_{M=1}^\infty \E\left[ \prod_{k=1}^{M}f_{N,\epsilon}(\ss_k, \tt_k) \right]
\le \sum_{M=1}^\infty (\Lambda_{N,\epsilon})^M
= \frac{\Lambda_{N,\epsilon}}{1-\Lambda_{N,\epsilon}} < \infty \,,
\end{equation}
provided $\Lambda_{N,\epsilon} < 1$.
The next Lemma shows that this is indeed the case, if
$\epsilon > 0$ is small enough and $N > N_0(\epsilon)$.
This completes the proof of \eqref{EqGS}, hence of $f^{(1)} \prec f^{(2)}$.

\begin{lemma}\label{LT1}The following relation holds for $\Lambda_{N,\epsilon}$
defined in \eqref{eqG1sec}, with $f_{N,\epsilon}$ defined in \eqref{DeltaEq}:
\begin{equation}\label{eq:lemma1sec}
	\limsup_{\epsilon\to 0} \, \limsup_{N\to\infty} \, \Lambda_{N,\epsilon}=0 \,.
\end{equation}
\end{lemma}

The proof of Lemma~\ref{LT1} is deferred to the Appendix~\ref{proofLemmaLT1}.
The key idea is that, for fixed $s < t$,
the function $f_{N,\epsilon}(s,t)$ in \eqref{DeltaEq} is small
when $\epsilon > 0$ small and $N$ large,
because the discrete partition function in the denominator is close to the
continuum one appearing in the numerator, but with $\hh' > \hh$ 
 (recall that the continuum partition function is strictly increasing in $\hh$, by 
Proposition~\ref{th:continuous}).
To prove that $\Lambda_{N,\epsilon}$ in \eqref{eqG1sec} is small,
we replace $s,t$ by the random points $\ss_k, \tt_k$,
showing that they cannot be too close to each other, conditionally on
(and uniformly over) $\ss_{k-1}, \tt_{k-1}$.

\subsection{Second Step: $f^{(2)}\prec f^{(3)}$}\label{SECONDstepSec}

Recalling \eqref{eq:f12} and \eqref{HVCF}-\eqref{Hm},
we can write $f^{(3)}$ as follows:
\begin{equation} \label{eq:exp3}
	f_{t}^{(3)}(N,\hb,\hh) =\frac{1}{t}\mathbb{E}\left(\log
	\E\left[\prod_{k=1}^{\rmm}
	\Zn(N\rss_k, N\rtt_k) \right]\right) \, .
\end{equation}
Note that $f^{(2)}$, defined in \eqref{HVDFs}, enjoys the same
representation \eqref{eq:exp3}, with $\rmm$ and $\rss_k, \rtt_k$ replaced respectively by
their continuum counterparts $\mm$ and $\ss_k, \tt_k$.
 Since we extend the discrete partition function
in a piecewise constant fashion  $\Zn(\lfloor Ns \rfloor, \lfloor Nt\rfloor)$,
cf.\ Remark~\ref{rem:piecewise}, we can replace $\ss_k, \tt_k$ by
their left neighbors $\dss_k, \dtt_k$ on the lattice $\frac{1}{N}\N_0$, i.e.\
\begin{equation} \label{eq:disccont}
	\dss_k := \frac{\lfloor N \ss_k \rfloor}{N} \,, \qquad 
	\dtt_k := \frac{\lfloor N \tt_k \rfloor}{N} \,,
\end{equation}
getting to the following representation for $f_t^{(2)}$:
\begin{equation} \label{eq:exp2}
	f_{t}^{(2)}(N,\hb,\hh) =\frac{1}{t}\mathbb{E}\left(\log
	\E\left[\prod_{k=1}^{\mm}
	\Zn(N\dss_k, N\dtt_k) \right]\right) \, .
\end{equation}

The random vectors $(\mm, (\dss_k,\dtt_k)_{1 \le k \le \dmm})$ and 
$(\rmm, (\rss_k,\rtt_k)_{1 \le k \le \rmm})$ 
are mutually absolutely continuous.
Let us denote by $\RN_t$ 
the Radon-Nikodym derivative
\begin{equation}\label{eq:RADONNIKODYM}
	\RN_t\left(M, (x_k,y_k)_{k=1}^{M}\right) =
	\frac{\P\left(\rmm =M, (\rss_k,\rtt_k)_{k=1}^{m}=(x_k,y_k)_{k=1}^{M}\right)}
	{\P\left(\mm =M, (\dss_k,\dtt_k)_{k=1}^{M}=(x_k,y_k)_{k=1}^{M}\right)} \,,
\end{equation}
for $M\in\N$ and $x_k, y_k \in \frac{1}{N}\N_0$ (note that necessarily $x_1 =0$). 
We can then rewrite \eqref{eq:exp3} as follows:
\begin{equation} \label{eq:exp3bis}
	f_{t}^{(3)}(N,\hb,\hh) =\frac{1}{t}\mathbb{E}\left(\log
	\E\left[\prod_{k=1}^{\mm}\left(\Zn(N\dss_k, N\dtt_k)\right) \cdot 
	\RN_t\left(\mm, (\dss_k,\dtt_k)_{k=1}^{\mm}\right)\right] \right) \,,
\end{equation}
which is identical to \eqref{eq:exp2}, apart from the Radon-Nikodym derivative
$\RN_t$.

Relations \eqref{eq:exp2} and \eqref{eq:exp3bis} are  useful because
$f_{t}^{(2)}$ and $f_{t}^{(3)}$ are averaged with respect to \emph{the same
random set $\btau$} (through its coarse-grained decomposition
$\mm$ and $\dss_k, \dtt_k$). This allows to apply the general strategy
of \S\ref{sec:strate}. 
Defining $\Delta_{N,\epsilon} = \Delta^{(2,3)}_{N,\epsilon}$ as in \eqref{eqHam1},
we can write by \eqref{eq:exp2}-\eqref{eq:exp3bis}
\begin{equation}\label{EqStep2}
\begin{split}
\E\mathbb{E} \left( e^{\frac{1}{\epsilon}\Delta_{N,\epsilon}(t)} \right) =
\E\left[\left\{ \prod_{k=1}^{\mm} \mathbb{E}\left[\left(\frac{\Zn(N\dss_k, N\dtt_k)}
{\Znt(N\dss_k, N\dtt_k)^{1-\epsilon}}\right)^{\frac{1}{\epsilon}}\right]\,\right\}
\, \RN_t\left(\mm, (\dss_k,\dtt_k)_{k=1}^{\mm}\right)^{\frac{1}{\epsilon}}\right],
\end{split}
\end{equation}
and our goal is to prove \eqref{EqGS} with $\Delta^{(1,2)}_{N,\epsilon}$ replaced by
$\Delta_{N,\epsilon}$: explicitly, for fixed $\hb, \hh, \hh'$ with $\hh < \hh'$,
\begin{equation}\label{EqGS23}
\begin{split}
\exists \epsilon \in (0,1), \ N_0 \in (0,\infty): \qquad 
&\limsup_{t\to\infty}\frac{1}{t}\log \E\mathbb{E}
\left(e^{\frac{1}{\epsilon}\Delta_{N,\epsilon}(t)}\right)\leq 0, 
\quad \forall N \ge N_0 \,.
\end{split}
\end{equation}

\smallskip

In order to simplify \eqref{EqStep2},
in analogy with \eqref{DeltaEq}, we define
\begin{equation} \label{eq:tildef}
	\gne(s,t) := \mathbb{E}\left[\left(\frac{\Zn(N s, Nt)}
	{\Znt(Ns, Nt)^{1-\epsilon}}\right)^{\frac{1}{\epsilon}}\right] \, .
\end{equation}
The Radon-Nikodym derivative $\RN_t$ in
\eqref{eq:RADONNIKODYM} does not factorize exactly, but
an approximate factorization holds: as we show in section
\ref{sec:RNSETAPP} (cf. Lemma \ref{th:auxi}),
for suitable functions $r_N$ and $\tilde r_{N}$
\begin{equation} \label{eq:RNfact}
	\RN_t\left(M, (x_k,y_k)_{k=1}^{M}\right) \le
	\left\{\prod_{\ell=1}^M r_N(y_{\ell-1}, x_\ell, y_\ell) \right\}
	\tilde r_{N}(y_M, t) \,,
\end{equation}
where we set $y_0 := 0$ (also note that $x_1 = 0$). Looking back at \eqref{EqStep2}, we can write
\begin{equation}\label{EqStep2b}
\begin{split}
\E\mathbb{E} \left( e^{\frac{1}{\epsilon}\Delta_{N,\epsilon}(t)} \right) \le
\E\left[\left\{ \prod_{k=1}^{\mm} \gne
\left(\dss_k, \dtt_k\right) \, 
r_N\left(\dtt_{k-1}, \dss_k, \dtt_k\right)^\frac{1}{\epsilon} \right\}
\,  \tilde r_{N}\left(\dtt_{\mm}, t
\right)^{\frac{1}{\epsilon}}\right] \,.
\end{split}
\end{equation}

Let us now explain the strategy.
We can easily get rid of the last term $\tilde r_{N}$ by Cauchy-Schwarz, so
we focus on the product appearing in brackets.
The goal would be to prove that \eqref{EqGS23} holds by bounding
\eqref{EqStep2b} through a geometric series, as in \eqref{eq:geomm}.
This could be obtained, in analogy with \eqref{eq:itera}-\eqref{eqG1sec},
by showing that for $\epsilon$ small and $N$ large the conditional expectation
\begin{equation*}
	\E \left[ \left. \gne(\dss_k, \dtt_k) \,
	r_N (\dtt_{k-1}, \dss_k, \dtt_k )^\frac{1}{\epsilon} \right| \cF_{k-1} \right]
	= \E \left[ \left. \gne(\dss_k, \dtt_k) \,
	r_N (\dtt_{k-1}, \dss_k, \dtt_k )^\frac{1}{\epsilon} \right| \ss_{k-1},
	\tt_{k-1} \right]
\end{equation*}
is smaller than $1$, \emph{uniformly in $\ss_{k-1}, \tt_{k-1}$}.
Unfortunately this fails, because the Radon-Nikodym
term $r_N$ is \emph{not} small when $\tt_{k-1}$ is close to the right
end of the block to which it belongs, i.e.\ to $\JJ_{k-1}$. 

To overcome this difficulty, we 
distinguish the two events $\{\tt_{k-1} \le \JJ_{k-1} - \gamma\}$ and
$\{\tt_{k-1} > \JJ_{k-1} - \gamma\}$, for $\gamma > 0$ that will be chosen small enough.
The needed estimates on the functions $\gne$, $r_N$ and $\tilde r_{N}$
are summarized in the next Lemma, proved in Appendix~\ref{sec:RNSETAPP}.
Let us define for $p \ge 1$ the constant
\begin{equation}
\begin{split}\label{eqG1app00}
\Lambda_{N,\epsilon,p} :=
\sup_{k, \ss_{k-1}, \tt_{k-1}} \E \left(\left. 
\gne(\ss_k,\, \tt_k)^p \,\right|\, \ss_{k-1}, \tt_{k-1} \right) \,,
\end{split}
\end{equation}
where we recall that $\gne(s,t)$ is defined in \eqref{eq:tildef}, 
and we agree that $\E[\,\cdot\,|\, \ss_0, \tt_0] := \E[\,\cdot\,]$. 

\begin{lemma}\label{lemma:lemmone1}
Let us fix $\hb \in \R$ and $\hh, \hh' \in \R$ with $\hh < \hh'$.
\begin{itemize}
\item For all $p\geq 1$ 
\begin{equation}\label{eq:point1lemmabv}
\limsup_{\epsilon\to 0} \, \limsup_{N\to\infty} \, \Lambda_{N,\epsilon,p} =0 \,.
\end{equation}

\item For all $\epsilon \in (0,1)$, $p\geq 1$
there is $C_{\epsilon,\, p} < \infty$ such that for all $N\in\N$
\begin{gather}\label{eq:alw}
	\forall k \ge 2: \quad
	\E\left[\left. r_{N}\left(\dtt_{k-1},\, \dss_k,\, \dtt_k \right)^{\frac{p}{\epsilon}}\, 
	\right\vert \, \ss_{k-1}\, \tt_{k-1}\right] \leq C_{\epsilon,p} \,, \\
	\label{eq:alw2}
	\E\left[\tilde{r}_{N}\left(\dtt_{\mm} , t
	\right)^{\frac{p}{\epsilon}} \right]\, \leq C_{\epsilon,p} \,.
\end{gather}

\item For all $\epsilon \in (0,1)$, $p\geq 1$, $\gamma \in (0,1)$ there is
$\tilde N_0= \tilde N_0(\epsilon,\, p,\, \gamma) < \infty$ such that 
for $N\ge \tilde N_0$
\begin{gather}\label{eq:spe}
	\forall k \ge 2: \quad
	\E\left[\left. r_{N}\left(\dtt_{k-1},\, \dss_k,\, \dtt_k \right)^{\frac{p}{\epsilon}}\, 
	\right\vert \, \ss_{k-1}\, \tt_{k-1}\right] \leq 2
	\quad \text{on the event} \ \{\tt_{k-1} \le \JJ_{k-1}-\gamma \} \,,\\
	\label{eq:spe2}
	\E\left[ r_{N}\left(0,\, 0,\, \dtt_1 \right)^{\frac{p}{\epsilon}}\, 
	\right]\, \leq 2 \,.
\end{gather}
%
%
%
%
\end{itemize}
\end{lemma}

\smallskip

We are ready to estimate \eqref{EqStep2b}, with the goal of proving \eqref{EqGS23}.
Let us define
\begin{equation}\label{Phi31}
\Phi_{k,N}^{(\epsilon)}:=
\gne\left(\dss_k,\, \dtt_k \right)^2\, 
r_N\left(\dtt_{k-1},\, \dss_k,\, \dtt_k \right)^{\frac{2}{\epsilon}},
\end{equation}
with the convention that $\dtt_{0} := 0$
(note that also $\dss_1 =0$). Then, by \eqref{eq:alw2} and Cauchy-Schwarz,
\begin{equation*}
	\E\mathbb{E} \left( e^{\frac{1}{\epsilon}\Delta_{N,\epsilon}(t)} \right) \le
	C_{\epsilon,2} \, \E\left[\prod\limits_{k=1}^{\mm}
	\Phi_{k,N}^{(\epsilon)} \right]
	\le C_{\epsilon,2} \, \sum\limits_{M=1}^{\infty}\E\left[\prod\limits_{k=1}^{M}
	\Phi_{k,N}^{(\epsilon)} \right] \,.
\end{equation*}
We are going to show that
\begin{equation} \label{eq:gogo}
	\exists \epsilon \in (0,1), \ N_0 \in (0,\infty): \qquad 
	\E\left[\prod\limits_{k=1}^{M} \Phi_{k,N}^{(\epsilon)} \right]
	\le \frac{1}{2^M} \quad \forall M\in\N, \ N \ge N_0 \,,
\end{equation}
which yields the upper bound
$\E\mathbb{E} ( e^{\frac{1}{\epsilon}\Delta_{N,\epsilon}(t)} ) \le
C_{\epsilon,2} $, completing the proof of \eqref{EqGS23}.

\smallskip

In the next Lemma, that will be proved in a moment,
we single out some properties of $\Phi_{k,N}^{(\epsilon)}$,
that are direct consequence of Lemma~\ref{lemma:lemmone1}.

\begin{lemma}\label{lemma11111}
One can choose $\epsilon \in (0,1)$, $c \in (1,\infty)$,
$\gamma \in (0,1)$ and $N_0 < \infty$ such that for $N\geq N_0$
\begin{align}\label{Phi1}
\E\left[ \Phi_{1,N}^{(\epsilon)} \right]\leq \, \frac{1}{4} \,; \qquad \
\forall k \ge 2: \ \ \
\E\left[\left.\Phi_{k,N}^{(\epsilon)}\, \right\vert\,  \ss_{k-1},\, \tt_{k-1}\right] 
\leq \begin{cases}
 c & \text{always} \\
\frac{1}{4} & \text{on} \ \{\tt_{k-1} \le \JJ_{k-1}-\gamma \} 
\end{cases} \,,
\end{align}
and moreover
\begin{align}\label{Phi3}
& \E\left[\Phi_{1,N}^{(\epsilon)} 
\ind_{\{\tt_{1} > 1-\gamma\}} \right] 
\leq \frac{1}{8c} \,; \qquad \forall k \ge 2: \ \ \
\E\left[\left.\Phi_{k,N}^{(\epsilon)} 
\ind_{\{\tt_{k} > \JJ_{k}-\gamma \}}\, \right\vert\, \ss_{k-1},\, \tt_{k-1}\right] 
\leq \frac{1}{8c} \,.
\end{align}
\end{lemma}

Let us now deduce \eqref{eq:gogo}. 
We fix $\epsilon$, $c$, 
$\gamma$ and $N_0$ as in Lemma~\ref{lemma11111}.
Setting for compactness
\begin{equation*}
	D_{M,N} := \prod\limits_{k=1}^{M} \Phi_{k,N}^{(\epsilon)} \,,
\end{equation*}
we show the following strengthened version of \eqref{eq:gogo}: 
\begin{align}\label{point} 
\E\left[D_{M,N}\right]\leq\, \frac{1}{2^{M}} \,,
\qquad 
\E\left[D_{M,N}
\ind_{\{\tt_{M} > \JJ_{M}-\gamma\}}\right]\leq\, \frac{1}{c 2^{M+2}}\,,
\qquad \forall M \in\N\,, \ N \ge N_0 \,.
\end{align} 
We proceed by induction on $M\in\N$. The case $M=1$ holds by
the first relations in \eqref{Phi1}, \eqref{Phi3}.
For the inductive step, we fix $M \ge 2$
and we assume that \eqref{point} holds for $M-1$, then
\begin{align*}
\nonumber
\E\left[D_{M,N} \right]
& = \E\left[D_{M-1,N} \, \E\left(\left.\Phi_{M,N}^{(\epsilon)} 
\, \right\vert\, \cF_{M-1}\right) \right] 
= \E\left[D_{M-1,N} \, \E\left(\left.\Phi_{M,N}^{(\epsilon)} 
\, \right\vert\, \ss_{M-1},\, \tt_{M-1}\right) \right]\\
&= \E\left[D_{M-1,N} \, 
\E\left(\left.\Phi_{M,N}^{(\epsilon)} \, \right\vert\, \ss_{M-1},\, \tt_{M-1}\right)
\ind_{\{\tt_{M-1} > \JJ_{M-1}-\gamma\}}\right] \\ 
&\qquad+
\E\left[D_{M-1,N} \, \E\left(\left.\Phi_{M,N}^{(\epsilon)}  \, \right\vert\, 
\ss_{M-1},\, \tt_{M-1}\right) \ind_{\{\tt_{M-1} \le \JJ_{M-1}-\gamma\}}
\right] \\
& \le c \, \E\left[D_{M-1,N} 
\ind_{\{\tt_{M-1} > \JJ_{M-1}-\gamma\}}\right] \,+\,
\frac{1}{4} \,\E\left[D_{M-1,N} \right] \leq
c \, \frac{1}{c 2^{M+1}} + \frac{1}{4} \frac{1}{2^{M-1}} \le \frac{1}{2^M} \,,
\end{align*}
where in the last line we have applied \eqref{Phi1} and the induction
step. Similarly, applying the second relation in \eqref{Phi3}
and the induction step,
\begin{align*}
& \E\left[{D_{M,N} }\ind_{\{\tt_{M} > \JJ_{M}-\gamma \}}\right]
=\E\left[\left. D_{M-1,N} \E\left({\Phi_{M,N}^{(\epsilon)} }
\ind_{\{\tt_{M} > \JJ_{M}-\gamma \}} \,\right\vert\, \ss_{M-1},\, \tt_{M-1}\right)\right]
\le \frac{1}{8c}\E\left[D_{M-1,N}\right]
\leq \frac{1}{c 2^{M+2}}.
\end{align*}
This completes the proof of \eqref{point},
hence of \eqref{eq:gogo}, hence of $f^{(2)}\prec f^{(3)}$.

\begin{proof}[Proof of Lemma~\ref{lemma11111}]
We fix $\epsilon > 0$ such that,
by relation \eqref{eq:point1lemmabv}, for some $\hat N_0 < \infty$ one has 
\begin{equation}\label{eq:LLa}
	\Lambda_{N,\epsilon,4p} 	\le \frac{1}{32} \,, \quad \forall N \ge \hat N_0,
	\quad \text{for both } p=1 \text{ and } p=2 \,.
\end{equation}
Given the parameter $\gamma \in (0,1)$, to be fixed later, 
we are going to apply
relations \eqref{eq:spe}-\eqref{eq:spe2}, that hold for $N \ge \tilde N_0(\gamma)$
and for $p \in \{1,2\}$ (we stress that $\epsilon$ has been fixed).
Defining $N_0 := \max\{\tilde N_0(\gamma), \hat N_0\}$, whose value will
be fixed once $\gamma$ is fixed, henceforth we assume that $N \ge N_0$.

Recalling
\eqref{Phi31} and \eqref{eqG1app00}, for $k\ge 2$ and $p \in \{1,2\}$ one has, by Cauchy-Schwarz,
\begin{equation} \label{eq:twoli}
\begin{split}
\E\left[\left.\left(\Phi_{k,N}^{(\epsilon)}\right)^p \, \right\vert\,  
\ss_{k-1},\, \tt_{k-1}\right]^2
 & \leq \E\left[\left.\gne\left(\dss_k,\, \dtt_k\right)^{4p}\, 
 \right\vert\,  \ss_{k-1},\, \tt_{k-1}\right]
 \cdot \\
& \qquad \quad \cdot
\E\left[\left.r_N\left(\dtt_{k-1},\, \dss_k,\, \dtt_k \right)^{\frac{4p}{\epsilon}}\, 
 \right\vert\,  \ss_{k-1},\, \tt_{k-1}\right] \\
 & \le \Lambda_{N,\epsilon,4p} \cdot
 \E\left[\left.r_N\left(\dtt_{k-1},\, \dss_k,\, \dtt_k \right)^{\frac{4p}{\epsilon}}\, 
 \right\vert\,  \ss_{k-1},\, \tt_{k-1}\right] \\
 & \le	\begin{cases}
	\frac{1}{32} \cdot C_{\epsilon, 4p} & \text{always} \\
	\frac{1}{32} \cdot 2 = \frac{1}{4^2} & \text{on }  \ \{\tt_{k-1} \le \JJ_{k-1}-\gamma \} 
	\end{cases} \,
\end{split}
\end{equation}
having used \eqref{eq:LLa}.
Setting $p=1$, the second relation in \eqref{Phi1} holds with 
$c := \sqrt{\frac{C_{\epsilon,4}}{32}}$. The first relation in \eqref{Phi1} is proved
similarly,
setting $\E[\,\cdot\,|\, \ss_0, \tt_0] := \E[\,\cdot\,]$ in \eqref{eq:twoli}
and applying \eqref{eq:spe2}.

Coming to \eqref{Phi3}, by Cauchy-Schwarz
\begin{equation}\label{eq:twoli2}
\begin{split}
\E\left[\left.\Phi_{k,N}^{(\epsilon)} 
\ind_{\{\tt_{k} > \JJ_{k}-\gamma \}}\, \right\vert\, \ss_{k-1},\, \tt_{k-1}\right] ^2 
& \leq \E\left[\left.\left(\Phi_{k,N}^{(\epsilon)}\right)^2 \, \right\vert\, 
\ss_{k-1},\, \tt_{k-1}\right]\cdot \P\left(\left. \tt_k > \JJ_k-\gamma \,
\right\vert\, \ss_{k-1},  \tt_{k-1}\right) \\
& \le \frac{C_{\epsilon,8}}{32} \, \left\{ \sup_{(x,y)\in [0,1)_{\leq}^2}
\P_x\left(\left. \tt_2 > \JJ_2-\gamma \,\right\vert\, \tt_{1}=y\right) \right\} \,,
\end{split}
\end{equation}
having applied \eqref{eq:twoli} for $p=2$, together with the regenerative
property and translation invariance of $\btau$.
By Lemma~\ref{th:lemmaP2}, we can choose $\gamma > 0$ small enough so that the second relation
in \eqref{Phi3} holds (recall that $c > 1$ has already been fixed, as a function of
$\epsilon$ only). The first relation in \eqref{Phi3} holds by similar arguments,
setting $\E[\,\cdot\,|\, \ss_0, \tt_0] := \E[\,\cdot\,]$ in \eqref{eq:twoli2}.
\end{proof}


\section{Proof of Theorem~\ref{Thm1}}
\label{sec:Thm1}
The existence and finiteness of the limit \eqref{CFcE} has been already proved
in Lemma~\ref{th:impex}. The fact that ${\fhc}(\hb,\hh)$ is non-negative
and convex in $\hh$ follows immediately by relation \eqref{eqFreeEnergyIntro}
(which is a consequence of Theorem~\ref{Thm2}, that we have already proved),
because the discrete partition function $\mathrm{F}(\beta,h)$ has these properties.
(Alternatively, one could also give direct proofs of these properties, following
the same path as for the discrete model.)
Finally, the scaling relation \eqref{eq:scalingfe} holds because
$\Zf(ct)$ has the same law as $\Zcf(t)$, by \eqref{eqParam}-\eqref{eq:convd}
(see also \cite[Theorem~2.4]{CRZ14}).

\appendix

\section{Regenerative Set}
\label{sec:rege}

\subsection{Proof of Lemma~\ref{th:lemmaP2}}
\label{sec:lemmaP2}
We may safely assume that $\gamma < \frac{1}{4}$, since for
$\gamma \ge \frac{1}{4}$ relations
\eqref{lemmaP2}-\eqref{eq:LUC1} are trivially satisfied, by choosing
$A_\alpha$, $B_\alpha$ large enough.

We start by \eqref{eq:LUC1},
partitioning on the index $\JJ_2$ of the block containing
$\ss_2$, $\tt_2$ (recall \eqref{eq:JJ}, \eqref{eq:sstt}):
\begin{align*}
\P_{x}\left(\tt_2-\ss_2\leq \gamma \mid \tt_1=y\right)=&
\sum_{n=2}^\infty \P_{x}\left(\tt_2-\ss_2\leq \gamma ,\, 
\JJ_2 =n \mid \tt_1=y\right) \,,
\end{align*}
for $(x,y) \in [0,1]^2_\le$.
Then \eqref{eq:LUC1} 
is proved if we show that there exists $c_{\alpha}\in (0,\infty)$ such that 
\begin{equation}\label{eq:usebou}
p_n(\gamma,x,y) := \P_{x}\left(\tt_2-\ss_2\leq \gamma ,\, 
\JJ_2 =n \mid \tt_1=y\right)
\leq \frac{c_{\alpha}}{n^{1+\alpha}} \, \gamma^{{\alpha}}, \quad
\forall n \ge 2, \ \forall (x,y)\in [0,1]_{\leq}^2.
\end{equation}

Let us write down the density of $(\tt_2,\ss_2, \JJ_2)$ given $\ss_1=x,\, \tt_1=y$.
Writing for simplicity $\mathtt{g}_t := \mathtt{g}_t(\btau)$ and
$\mathtt{d}_t := \mathtt{d}_t(\btau)$, we can write for $(z,w) \in [n-1,n]^2_\le$
\begin{equation*}
\begin{split}
	\P_{x}\left(\left. \ss_2 \in \dd z\, , \tt_2\in \dd w\, , \JJ_2=n \,\right\vert\, 
	\tt_1 =y\right) & = \frac{\P_{x}\left(\mathtt{g}_{1}\in  \dd y, \, 
	\mathtt{d}_1\in \dd z,\, \mathtt{g}_{n}\in \dd w \right)}
	{\P_x(\mathtt{g}_1 \in \dd y)} \\
	& = \frac{\P_{x}\left(\mathtt{g}_{1}\in  \dd y, \, 
	\mathtt{d}_1\in \dd z \right) \P_z\left( \mathtt{g}_{n}\in \dd w \right)}
	{\P_x(\mathtt{g}_1 \in \dd y)} \,,
\end{split}
\end{equation*}
where we have applied the regenerative property at the stopping time $\mathtt{d}_1$.
Then by \eqref{gdlaw}, \eqref{glaw} we get
\begin{equation}\label{glawcon}
\begin{split}
\frac{\P_{x}\left(\left. \ss_2 \in \dd z\, , \tt_2\in \dd w\, , \JJ_2=n \,\right\vert\, 
\tt_1 =y\right)}{ \dd z\, \dd w} =\, &
C_{\alpha}\, \frac{(1-y)^{\alpha}}{(z-y)^{1+\alpha}(w-z)^{1-\alpha}
(n-w)^{\alpha}} \\
& \rule{0pt}{1.1em}\text{for} \quad x \le y \le 1 \,, \ \ n-1 \le z \le w \le n \,.
\end{split}
\end{equation}
where $C_{\alpha}=\frac{\alpha\sin(\pi\alpha)}{\pi}$.
Note that this density is independent of $x$. Integrating over $w$, by
\eqref{glaw} we get
\begin{equation}\label{glawcon2}
\begin{split}
\frac{\P_{x}\left(\left. \ss_2 \in \dd z\, , \JJ_2=n \,\right\vert\, 
\tt_1 =y\right)}{ \dd z} =\, &
\alpha \, \frac{(1-y)^{\alpha}}{(z-y)^{1+\alpha}} \qquad
\text{for} \quad x \le y \le 1 \,, \ \ n-1 \le z \le n \,.
\end{split}
\end{equation}

We can finally estimate $p_n(\gamma,x,y)$ in \eqref{eq:usebou}.
We compute separately the contributions from the
events $\{\ss_2 \le n- \gamma\}$ and $\{\ss_2 > n - \gamma\}$,
starting with the former.
By \eqref{glawcon}
\begin{align}\label{eq:intze}
\begin{split}
&{C_{\alpha}}\, (1-y)^{\alpha} \int_{n-1}^{n-\gamma}
\, \frac{1}{(z-y)^{1+\alpha}}
\left( \int_{z}^{z+\gamma}
\, \frac{1}{(w-z)^{1-\alpha}(n-w)^{\alpha}} \, \dd w  \right)  \dd z\\
& \quad \leq \frac{C_\alpha}{\alpha}\, (1-y)^{\alpha}  \, \gamma^{\alpha}
\int_{n-1}^{n-\gamma}\, \frac{1}{(z-y)^{1+\alpha}}\frac{1}{(n -\gamma-z)^{\alpha}}\dd z \,,\end{split}
\end{align}
because $n-w \ge n-\gamma-z$.
In case $n\ge 3$, since $z-y \ge n-2$ (recall that  $y \in [0,1]$),
\begin{equation} \label{eq:intze2}
	\eqref{eq:intze} \le
	\frac{C_\alpha}{\alpha}\, \gamma^{\alpha} \, \frac{1}{(n-2)^{1+\alpha}}
	\int_{n-1}^{n-\gamma}\, \frac{1}{(n -\gamma-z)^{\alpha}} \, \dd z \, \le \,
	\frac{C_{\alpha}}{\alpha (1-\alpha)} \, \frac{\gamma^{\alpha}}{(n-2)^{1+\alpha}} \,,
\end{equation}
which matches with the right hand side of \eqref{eq:usebou}
(just estimate $n-2 \ge n/3$ for $n\ge 3$).
The same computation works also for $n=2$, provided 
we restrict the last integral in \eqref{eq:intze} on $\frac{3}{2} \le z \le 2-\gamma$,
which leads to \eqref{eq:intze2} with $(n-2)$ replaced by $1/2$.
On the other hand, in case $n=2$ and $1 \le z \le \frac{3}{2}$,
we bound $n-\gamma-z = 2 -\gamma-z \ge \frac{1}{4}$ in \eqref{eq:intze}
(recall that $\gamma < \frac{1}{4}$ by assumption), getting
\begin{equation*}
	\eqref{eq:intze} \le \frac{C_\alpha}{\alpha}\, (1-y)^{\alpha}  \, \gamma^{\alpha}
	\, 4^\alpha \int_1^{\infty} \frac{1}{(z-y)^{1+\alpha}} \, \dd z
	= \frac{C_{\alpha}}{\alpha^2} \, 4^\alpha \, \gamma^{\alpha} < \infty \,.
\end{equation*}
Finally, we consider the contribution
to $p_n(\gamma,x,y)$ of the event $\{\ss_2 > n- \gamma\}$,
i.e.\ by \eqref{glawcon2}
\begin{align*}
\int_{n-\gamma}^{n}\, \alpha\, \frac{(1-y)^{\alpha}}{(z-y)^{1+\alpha}}\dd z\leq\, 
\alpha\,  \frac{\gamma}{(n-\frac{3}{2})^{1+\alpha}} \,,
\qquad \forall n \ge 2 \,,
\end{align*}
because for $y\le 1$ we have $z-y \ge n-\gamma-1 \ge n-\frac{3}{2}$
(recall that $\gamma < \frac{1}{4}$).
Recalling that $\alpha < 1$, this matches with \eqref{eq:usebou},
completing the proof of \eqref{eq:LUC1}.

\smallskip
Next we turn to \eqref{lemmaP2}.
Disintegrating over the value of $\JJ_2$, for $0 \le x \le y \le 1$ we write
\begin{equation*}
	\P_x\left(\left. \tt_2\in [\JJ_2-\gamma,\JJ_2]\,\right\vert\, \tt_{1}=y\right)
	= \sum_{n=2}^\infty 
	\P_x\left(\left. \tt_2\in [n-\gamma,n], \,
	\JJ_2 = n\,\right\vert\, \tt_{1}=y\right)
	=: \sum_{n=2}^\infty q_n(\gamma,x,y) \,.
\end{equation*}
It suffices to prove that there exists $c_\alpha \in (0,\infty)$ such that
\begin{equation}\label{eq:goaa}
	q_n(\gamma,x,y) \le  \frac{c_\alpha}{n^{1+\alpha}} \, \gamma^{1-\alpha} \,,
	\qquad \forall n \ge 2 \,, \ \forall (x,y) \in [0,1]^2_\le \,.
\end{equation}
By \eqref{glawcon} we can write
\begin{equation} \label{eq:inner}
	q_n(\gamma,x,y) = {C_{\alpha}}\, (1-y)^{\alpha} 
	\int_{n-\gamma}^{n}
	\left( \int_{n-1}^{w}
	\frac{1}{(z-y)^{1+\alpha} \, (w-z)^{1-\alpha}} \,  \dd z \right)
	\frac{1}{(n-w)^{\alpha}}  \, \dd w \,.
\end{equation}
If $n \ge 3$ then $z-y \ge n-2$ (since $y \le 1$), which plugged into
in the inner integral yields
\begin{equation} \label{eq:yie}
	q_n(\gamma,x,y) \le {C_{\alpha}}\, \frac{(1-y)^{\alpha}}{(n-2)^{1+\alpha}} \,
	\frac{1}{\alpha} \, \int_{n-\gamma}^{n}\frac{1}{(n-w)^{\alpha}}  \, \dd w
	\le {C_{\alpha}}\, \frac{1}{(n-2)^{1+\alpha}} \,
	\frac{1}{\alpha} \, \frac{\gamma^{1-\alpha}}{(1-\alpha)} \,,
\end{equation}
which matches with \eqref{eq:goaa}, since $n-2 \ge n/3$ for $n\ge 3$.
An analogous estimate applies also for $n=2$, if we restrict the inner
integral in \eqref{eq:inner} to $z \ge n-1+\frac{1}{2} = \frac{3}{2}$,
in which case \eqref{eq:yie} holds with $(n-2)$ replaced by $1/2$.
On the other hand, always for $n=2$, in the range $1 \le z \le \frac{3}{2}$
we can bound $w-z \ge (2-\gamma)-\frac{3}{2} \ge \frac{1}{4}$ 
in the inner integral in \eqref{eq:inner} (recall that $\gamma < \frac{1}{4}$), getting
the upper bound
\begin{equation*}
	\frac{{C_{\alpha}}\, (1-y)^{\alpha}}{(\frac{1}{4})^{1-\alpha}}
	\left( \int_{1}^{\infty}
	\frac{1}{(z-y)^{1+\alpha}} \,  \dd z \right)
	\left( \int_{2-\gamma}^{2}
	\frac{1}{(2-w)^{\alpha}}  \, \dd w \right)
	= \frac{4^{1-\alpha} \, C_\alpha}{\alpha (1-\alpha)} \, \gamma^{1-\alpha} \,.
\end{equation*}
This completes the proof of \eqref{eq:goaa}, hence of Lemma~\ref{th:lemmaP2}.\qed

\subsection{Proof of Lemma~\ref{LT1}}\label{proofLemmaLT1}
Recall the definition \eqref{eqG1sec} of $\Lambda_{N,\epsilon}$. Note that
\begin{equation*}
	\E\big[f_{N,\epsilon}(\ss_k, \tt_k) \,|\, \ss_{k-1}, \tt_{k-1}\big]
	= \E_x \big[f_{N,\epsilon}(\ss_2, \tt_2) \,|\, \tt_1=y\big] 
	\big|_{(x,y) = (\ss_{k-1},\tt_{k-1})} \,,
\end{equation*}
where we recall that $\E_x$ denotes expectation with respect to the
regenerative set started at $x$, 
and $\tt_1$ under $\P_x$ denotes the last visited point of $\btau$
in the block $[n,n+1)$, where $n = \lfloor x \rfloor$, while $\ss_2, \tt_2$ denote
the first and last points of $\btau$ in the next visited block, cf.\ \eqref{eq:JJ}.
Then we can rewrite \eqref{eqG1sec} as
%
\begin{equation}
\begin{split}\label{eqG1app0}
\Lambda_{N,\epsilon}=
\sup_{n\in\N_0} \,
\sup_{(x,y)\in[n,n+1)_{\leq}^{2}}\E_{x}
\left[ \left. f_{N,\epsilon}(\ss_2, \tt_2)\,\right|\, \tt_{1}=y\right] \,.
\end{split}
\end{equation}

We first note that one can set $n=0$ in \eqref{eqG1app0}, by
translation invariance,
because $f_{N,\epsilon}(s+n, t+n) = f_{N,\epsilon}(s, t)$, cf. \eqref{DeltaEq}, and the joint law of 
$\big(\Z(s,t), \Znt(Ns, Nt) \big)_{(s,t) \in [m,m+1)^2_\le}$
does not depend on $m\in\N$, by the choice of the coupling, cf.\ \S\ref{sec:coupling}.
Setting $n=0$ in \eqref{eqG1app0}, we obtain
\begin{equation}
\begin{split}\label{eqG1app}
\Lambda_{N,\epsilon}=
\sup_{(x,y)\in[0,1)_{\leq}^{2}}\E_{x}
\left(\mathbb{E}\left[\left.\left(\frac{\Z(\ss_2,\tt_2)}{\Znt(N\ss_2, N\tt_2)^{1-\epsilon}}\right)^{\frac{1}{\epsilon}}\right]\,\right|\, \tt_{1}=y\right) \,.
\end{split}
\end{equation}

In the sequel we fix $\hb > 0$ and $\hh, \hh' \in \R$ with $\hh'>\hh$ (thus $h_N'>h_N$).
Our goal is to prove that 
\begin{equation}\label{eq:lemma1app5}
\limsup_{\epsilon\to 0} \, \limsup_{N\to\infty}\, \Lambda_{N,\epsilon} =0 \,.
\end{equation}
By Proposition~\ref{p1results}, there exists a constant $C < \infty$ such that
\begin{align*}
\sup_{N\in \mathbb{N}} \, \sup_{0 \le s \le t < \infty: \ |t-s| < 1} 
\mathbb{E}\left[\Znt(Ns, Nt)^2 \right] =
\sup_{N\in \mathbb{N}} \, \sup_{(s,t)\in\, [0,1]_{\leq}^2} 
\mathbb{E}\left[\Znt(Ns, Nt)^2 \right]\, \leq C \,,
\end{align*}
where the first equality holds because the law of $\Znt(Ns, Nt)$ only depends on $t-s$.
If we set
\begin{align} \label{eq:WW}
& W_N(s,t) :=\frac{\Z(s,t)}{\Znt(Ns,Nt)}, \qquad
\mathbf{W}(s,t) :=\frac{\Z(s,t)}{\Zrho(s,t)},
\end{align} 
we can get rid of the exponent $1-\epsilon$ in the denominator of \eqref{eqG1app},
by Cauchy-Schwarz:
\begin{align*}
\mathbb{E}\left[\left(\frac{\Z(s,t)}{\Znt(Ns, Nt)^{1-\epsilon}}\right)^{\frac{1}{\epsilon}}\right]&=
\mathbb{E}\left[\Znt(Ns, Nt)
\, W_N(s,t)^{\frac{1}{\epsilon}}
\right]\leq\,  C^{\frac{1}{2}}\, \mathbb{E}\left[W_N(s,t)^{\frac{2}{\epsilon}}\right]^{\frac{1}{2}} .
\end{align*}
We can then conclude by Jensen's inequality that
\begin{equation} \label{eqG2bis}
	(\Lambda_{N,\epsilon})^2
	\leq C\, \sup_{(x,y)\in[0,1)_{\leq}^{2}} \E_{x}
	\left(\left. \mathbb{E}\left[W_N(\ss_2,\tt_2)^{\frac{2}{\epsilon}}\right]\, 
	\right\vert\,  \tt_{M-1}=y\right) \,,
\end{equation}
and we can naturally split the proof of our goal \eqref{eq:lemma1app5} in two parts:
\begin{align}
\forall \epsilon > 0: \qquad
& \limsup_{N\to\infty} \, (\Lambda_{N,\epsilon})^2 	\leq C\, \sup_{(x,y)\in[0,1)_{\leq}^{2}} \E_{x}
	\left(\left. \mathbb{E}\left[\mathbf{W}(\ss_2,\tt_2)^{\frac{2}{\epsilon}}\right]\, 
	\right\vert\,  \tt_{M-1}=y\right) \,, 
	\label{ac1}\\
&\limsup_{\epsilon\to 0} \left( \sup_{(x,y)\in [0,1)_{\leq}^2} 
\E_{x}\left[\left. \mathbb{E}\left(\mathbf{W}(\ss_2,\tt_2)^{\frac{2}{\epsilon}}\right)\,  
\right\vert\,  \tt_1=y\right] \,\right) =0.\label{ac2}
\end{align}

We start proving \eqref{ac1}.
Let $\epsilon>0$ be fixed. 
It suffices to show that the right hand side of \eqref{eqG2bis} converges to the right hand
side of \eqref{ac1} as $N\to\infty$.
Writing the right hand sides of \eqref{eqG2bis} and \eqref{ac1} respectively as
$C \sup_{(x,y) \in [0,1)^2_\le} g_N(x,y)$ and $C \sup_{(x,y) \in [0,1)^2_\le} 
\mathbf{g}(x,y)$, it suffices to show that
$\sup_{(x,y) \in [0,1]^2_\le} |g_N(x,y) - \mathbf{g}(x,y)| \to 0$ as $N\to\infty$.
Note that
\begin{equation} \label{eq:trel}
\begin{split}
|g_N(x,y) - \mathbf{g}(x,y)| & = \left\vert\, 
\E_{x}\left[\left. \mathbb{E}\left(W_N(\ss_2,\tt_2)^{\frac{2}{\epsilon}}\right)\,  
\right\vert\,  \tt_1=y\right]-\E_{x}\left[\left. \mathbb{E}
\left(\mathbf{W}(\ss_2,\tt_2)^{\frac{2}{\epsilon}}\right)\,  \right\vert\,  
\tt_1=y\right]\, \right\vert\,\\
& \leq \, 
\E_{x}\left[\left. 
\mathbb{E}\left( \left\vert W_N(\ss_2,\tt_2)^{\frac{2}{\epsilon}}
\, -\, \mathbf{W}(\ss_2,\tt_2)^{\frac{2}{\epsilon}} \right\vert \right)  \,
\right\vert\,  \tt_1=y\right] \\
& \le \sup_{n\in\N_0} \, \sup_{(s,t) \in [n,n+1]^2_\le}
 \mathbb{E}\left( \left\vert W_N(s,t)^{\frac{2}{\epsilon}}
\, -\, \mathbf{W}(s,t)^{\frac{2}{\epsilon}} \right\vert \right)  \,,
\end{split}
\end{equation}
where the last inequality holds because $n \le \ss_2 \le \tt_2 \le n+1$ for some integer $n\in\N$.
The joint law of $(W_N(s,t), \mathbf{W}(s,t))_{(s,t) \in [n,n+1]^2_\le}$ does
not depend on $n\in\N$, 
by our definition of the coupling in \S\ref{sec:coupling}, hence
the $\sup_{n\in\N_0}$ in the last line of \eqref{eq:trel} can be dropped,
setting $n=0$. The proof of \eqref{ac1} is thus reduced to showing that
\begin{equation} \label{eq:ttoopp}
\forall \epsilon > 0: \qquad
\lim_{N\to\infty} \, \mathbb{E}\left[\mathtt{S}_N\right]=0 \,,
\qquad \text{with} \quad
\mathtt{S}_N := \sup_{(s,t)\in [0,1]_{\leq}^2}
\left\vert\, W_N(s,t)^{\frac{2}{\epsilon}}
-\mathbf{W}(s,t)^{\frac{2}{\epsilon}}\, \right\vert \,.
\end{equation}

Recall the definition \eqref{eq:WW} of $W_N$ and $\mathbf{W}_N$
and observe that $\lim_{N\to\infty}\mathtt{S}_N = 0$ a.s., because
by construction $\Zn(Ns,Nt)$ converges a.s.\ to $\Z(s,t)$,
uniformly in $(s,t) \in [0,1]^2_\le$, and $\Z(s,t) > 0$
uniformly in $(s,t) \in [0,1]^2_\le$,
by \cite[Theorem 2.4]{CRZ14}.
To prove that $\lim_{N\to\infty} \, \mathbb{E}\left[\, \mathtt{S}_N\, \right]=0$ 
it then suffices to show that $(\mathtt{S}_N)_{N\in\mathbb{N}}$ is bounded in $L^2$
(hence uniformly integrable). To this purpose we observe 
\begin{equation*}
\mathtt{S}_N^2\leq \, 2\sup_{(s,t)\in [0,1]_{\leq}^2} 
W_N(s,t)^{\frac{4}{\epsilon}} + 2\sup_{(s,t)\in [0,1]_{\leq}^2}
\mathbf{W}(s,t)^{\frac{4}{\epsilon}} \,,
\end{equation*}
and note that $\mathbf{W}(s,t)\leq 1$, because $\hh\mapsto \Z(s,t)$ is increasing,
cf. Proposition~\ref{th:continuous}. Finally, the first term has bounded expectation,
by Proposition~\ref{p1results} and Corollary~\ref{p1resultscont}: recalling \eqref{eq:WW},
\begin{equation*}
\sup_{N\in\N}\mathbb{E}\left[\sup_{(u,v)\in [0,1]_{\leq}^2}
W_N(s,t)^{\frac{4}{\epsilon}} \right]\leq
\mathbb{E}\left[\sup_{(u,v)\in [0,1]_{\leq}^2}
{\Z(u,v)}^{\frac{8}{\epsilon}}
\right]^{\frac{1}{2}} \sup_{N\in\N} \mathbb{E}\left[\sup_{(s,t)\in [0,1]_{\leq}^2}
{\Znt(s,t)}^{-\frac{8}{\epsilon}}\right]^{\frac{1}{2}} < \infty \,.
\end{equation*}

\medskip

Having completed the proof of \eqref{ac1}, we focus on \eqref{ac2}.
Let us fix $\gamma > 0$.
In analogy with \eqref{eq:trel}, we can bound the contribution to \eqref{ac2} of the
event $\{\tt_2 - \ss_2 \ge \gamma\}$ by
\begin{equation} \label{eq:chai}
	\sup_{n\in\N_0} \, \sup_{\substack{(s,t)\in [n,n+1]_{\leq}^2\\ |t-s|\geq \gamma}} \,
	\bbE \left[ \mathbf{W}(s,t)^{\frac{2}{\epsilon}} \right] =
	\sup_{\substack{(s,t)\in [0,1]_{\leq}^2\\ |t-s|\geq \gamma}} \,
	\bbE \left[ \mathbf{W}(s,t)^{\frac{2}{\epsilon}} \right] \le
	\bbE \left[ \sup_{\substack{(s,t)\in [0,1]_{\leq}^2\\ |t-s|\geq \gamma}}
	\mathbf{W}(s,t)^{\frac{2}{\epsilon}} \right] \,,
\end{equation}
where the equality holds because the law of $(\mathbf{W}(s,t))_{(s,t) \in [n,n+1]^2_\le}$
does not depend on $n\in\N_0$. Recall that by Proposition~\ref{th:continuous} one has, a.s.,
$\mathbf{W}(s,t) \leq  1$ for all $(s,t)\in (0,1]_{\leq}^2$,
with $\mathbf{W}(s,t) <  1$ for $s < t$. By continuity of $(s,t) \mapsto \mathbf{W}(s,t)$
it follows that also $\sup_{(s,t) \in [0,1]^2_\le: \ |t-s| \ge \gamma} \mathbf{W}(s,t) < 1$, a.s.,
hence the right hand side of \eqref{eq:chai} vanishes as $\epsilon \to 0$,
for any fixed $\gamma > 0$, by dominated convergence. This means that in order
to prove \eqref{ac2} we can focus on the event $\{\tt_2 - \ss_2 < \gamma\}$, and note that
\begin{equation*}
	\sup_{(x,y)\in [0,1)_{\leq}^2} 
	\E_{x}\left[\left. \mathbb{E}\left(\mathbf{W}(\ss_2,\tt_2)^{\frac{2}{\epsilon}}\right)\,  
	\ind_{\{\tt_2 - \ss_2 < \gamma\}}\right\vert\,  \tt_1=y\right] \le
	\sup_{(x,y)\in [0,1)_{\leq}^2} \P_x \left(\left. \tt_2-\ss_2\leq \gamma 
	\,\right\vert \tt_1=y\right) \,,
\end{equation*}
because $\mathbf{W}(s,t) \leq  1$.
Since $\gamma > 0$ was arbitrary, in order to prove \eqref{ac2} it is enough to show that
\begin{equation}\label{eq:lemma1appa2}
	\lim_{\gamma \to 0} \, \sup_{(x,y)\in [0,1)_{\leq}^2} \,
	\P_x \left(\left. \tt_2-\ss_2\leq \gamma 
	\,\right\vert \tt_1=y\right) = 0 \,.
\end{equation}
This is a consequence of
relation \eqref{eq:LUC1} in Lemma~\ref{th:lemmaP2}, which concludes the proof of 
Lemma~\ref{LT1}.\qed

\subsection{Proof of Lemma~\ref{lemma:lemmone1}}\label{sec:RNSETAPP}

We omit the proof of relation \eqref{eq:point1lemmabv}, because it is
analogous to (and simpler than)
the proof of relation \eqref{eq:lemma1sec} in Lemma~\ref{LT1}: compare
the definition of $f_{N,\epsilon}$
in \eqref{DeltaEq} with that of $g_{N,\epsilon}$ in \eqref{eq:tildef}, and the definition of
$\Lambda_{N,\epsilon}$ in \eqref{eqG1sec} with that of
$\Lambda_{N,\epsilon,p}$ in \eqref{eqG1app00} (note that the exponent $p$ in \eqref{eqG1app00}
can be brought inside the $\bbE$-expectation in \eqref{eq:tildef},
by Jensen's inequality).

\smallskip

In order complete the proof of Lemma~\ref{lemma:lemmone1}, we state
an auxiliary Lemma, proved in \S\ref{sec:auxi} below.
Recall that
$\RN_t(M, (x_k,y_k)_{k=1}^{M})$ was defined in \eqref{eq:RADONNIKODYM},
for $t,M\in\N$ and $x_k, y_k \in \frac{1}{N}\N_0$
satisfying the constraints $0 = x_1 \le y_1 < x_2 \le y_2 < \ldots
< x_M \le y_M \le t$. 
Also recall that $L: \N \to (0,\infty)$ 
denotes the slowly varying function appearing in \eqref{assRP0},
and we set $L(0) = 1$ for convenience.

\begin{lemma} \label{th:auxi}
Relation \eqref{eq:RNfact} holds for suitable functions $r_N$, $\tilde r_{N}$,
satisfying the following relations:
\begin{itemize}
\item there is $C \in (0,\infty)$ such that
for all $N\in\N$ and all admissible $y', x, y$, resp.\ $z, t$,
\begin{equation}\label{eq:rNbound}
	r_N(y', x, y) \le C \, \frac{L(N(x-y'))}{L(N(\lceil y' \rceil - y'))}
	\frac{L(N(\lceil y \rceil - y))}{L(N(y-x))} \,, \qquad
	\tilde r_N(z,t) \le C \, \frac{L(N(t-z))}{L(N(\lceil z \rceil-z))} \,;
\end{equation}

\item for all $\eta > 0$ there is $M_0 = M_0(\eta) < \infty$
such that for all $N \in \N$
and for admissible $y', x, y$ 
\begin{align}
	\label{eq:rNbound20}
	r_N(0, 0, y) & \le (1+\eta) \, \frac{L(N(\lceil y \rceil - y))}{L(N y )} \,, \qquad
	\text{if} \quad y \ge \frac{M_0}{N} \,; \\
	\label{eq:rNbound2}
	r_N(y', x, y) & \le (1+\eta) \, \frac{L(N(x-y'))}{L(N(\lceil y' \rceil - y'))}
	\frac{L(N(\lceil y \rceil - y))}{L(N(y-x))} \,, \qquad
	\text{if} \quad y-x \ge \frac{M_0}{N}, \ \ x - y' \ge \frac{M_0}{N} \,.
\end{align}
\end{itemize}
\end{lemma}

We can now prove relations \eqref{eq:alw}, \eqref{eq:alw2}.
By Potter's bounds \cite[Theorem 1.5.6]{BGT89},
for any $\delta>0$ there is a constant $c_{\delta}>0$ such that 
$L(m)/L(\ell)\leq c_{\delta}\max\big\{\frac{m+1}{\ell+1}, 
\frac{\ell+1}{m+1}\big\}^{\delta}$
for all $m,\ell\in\N_0$ (the ``$+1$''
is because we allow $\ell, m$ to attain the value $0$).
Looking at \eqref{eq:rNbound}-\eqref{eq:rNbound2}, 
recalling that the admissible values of $y', x, y$ are such that
$\lceil y' \rceil - y' \le x - y'$ and $y-x \le 1$,
$\lceil y \rceil - y \le 1$,
we can estimate
\begin{equation*}
\begin{split}
	\frac{L(N(x-y'))}{L(N(\lceil y' \rceil - y'))}
	\frac{L(N(\lceil y \rceil - y))}{L(N(y-x))} & \le c_\delta^2 \,
	\left(\frac{x-y'+ \frac{1}{N}}{\lceil y' \rceil-y' + \frac{1}{N}}\right)^\delta
	\max\left\{ \frac{y-x+ \frac{1}{N}}
	{\lceil y \rceil - y + \frac{1}{N}}, 
	\frac{\lceil y \rceil - y+ \frac{1}{N}}{y-x + \frac{1}{N}} \right\}^\delta \\
	& \le 2^\delta c_\delta^2 \,
	\left(\frac{x-y'+ \frac{1}{N}}{\lceil y' \rceil-y' + \frac{1}{N}}\right)^\delta
	\frac{1}{(\lceil y \rceil - y + \frac{1}{N})^\delta}
	\,\frac{1}{(y-x + \frac{1}{N})^\delta} \,.
\end{split}
\end{equation*}
We now plug in $y' = \dss_{k-1}$, $x = \dss_k$, $y = \dtt_k$
(so that $\lceil y' \rceil = \JJ_{k-1}$ and $\lceil y \rceil = \JJ_{k}$).
The first relation in \eqref{eq:rNbound} then yields
\begin{equation*}
\begin{split}
	r_N\left( \dtt_{k-1}, \dss_k, \dtt_k \right) & \le C \, 2^\delta c_\delta^2 \,
	\left(\frac{\dss_k-\dtt_{k-1}+ \frac{1}{N}}{\JJ_{k-1} - \dtt_{k-1} + \frac{1}{N}}\right)^\delta
	\frac{1}{(\JJ_k - \dtt_k+ \frac{1}{N})^\delta}
	\, \frac{1}{(\dtt_k - \dss_k+ \frac{1}{N})^\delta} \\
	& \le C \, 2^\delta c_\delta^2 \,
	\left(\frac{\ss_k-\tt_{k-1}}{\JJ_{k-1} - \tt_{k-1}}\right)^\delta
	\frac{1}{(\JJ_k - \tt_k)^\delta}
	\,\frac{1}{(\tt_k - \ss_k)^\delta} \,,
\end{split}
\end{equation*}
where the last inequality holds by monotonicity,
since
$\dss_k \le \ss_k$, $\dtt_i \le \tt_i$ for $i=k-1, k$ and
$\dtt_k-\dss_k + \frac{1}{N} \ge \tt_k - \ss_k$ by definition \eqref{eq:disccont}.
Setting $C'_\delta := C \,
2^\delta c_\delta^2$, by the regenerative property
\begin{equation*}
	\E\left[\left. r_{N}\left(\dtt_{k-1},\, \dss_k,\, \dtt_k \right)^{\frac{p}{\epsilon}}\, 
	\right\vert \, \ss_{k-1}\, \tt_{k-1}\right] \leq 
	\left(C'_\delta \right)^{\frac{p}{\epsilon}} \,
	\E_x\left[\left. \left(\frac{\ss_2-y}{1 - y}
	\right)^{\frac{\delta p}{\epsilon}} 	\frac{1}{(\JJ_2 - \tt_2)^{\frac{\delta p}{\epsilon}}}
	\,\frac{1}{(\tt_2 - \ss_2)^{\frac{\delta p}{\epsilon}}} \right| \tt_1 = y \right] \,,
\end{equation*}
with $(x,y) = (\ss_{k-1},\tt_{k-1})$. Since
$\E[XYZ] \le (\E[X^3] \E[Y^3] \E[Z^3])^{1/3}$ by H\"older's inequality, we split
the expected value in the right hand side in three parts, estimating each term separately. 

First, given $x, y \in [n,n+1)$ for some $n\in\N$, then
$\tt_1 = \mathtt{g}_n(\btau)$ and $\ss_2 = \mathtt{d}_n(\btau)$, hence by \eqref{dlaw}
\begin{equation*}
	\E_x\left[\left. \left(\frac{\ss_2-y}{1 - y}
	\right)^{\frac{3\delta p}{\epsilon}} \right| \tt_1=y \right] =
	\E_x\left[\left. \left(\frac{\mathtt{d}_n(\btau)-y}{1 - y}
	\right)^{\frac{3\delta p}{\epsilon}} \right| \mathtt{g}_n(\btau)=y \right]
	= \int_{n}^\infty \left(\frac{v-y}{n - y}
	\right)^{\frac{3\delta p}{\epsilon}} \, \frac{(n-y)^\alpha}{(v-y)^{1+\alpha}}
	\, \dd v \,,
\end{equation*}
and the change of variable $z := \frac{v-y}{n-y}$ yields
\begin{equation} \label{eq:C1}
	\E_x\left[\left. \left(\frac{\ss_2-y}{1 - y}
	\right)^{\frac{3\delta p}{\epsilon}} \right| \tt_1=y \right] =
	\int_{1}^\infty z^{\frac{3\delta p}{\epsilon}-1-\alpha}	\, \dd z
	= \frac{1}{\alpha - \frac{3\delta p}{\epsilon}}
	=: C_1 < \infty \,,
	\qquad \text{if} \quad \delta < \frac{\alpha \epsilon}{3p} \,.
\end{equation}
Next, since $\E[X^{-a}] = \int_0^\infty \P(X^{-a} \ge t) \, \dd t
= \int_0^\infty \P(X \le t^{-1/a}) \, \dd t$ for any random variable $X \ge 0$,
\begin{equation} \label{eq:C2}
\begin{split}
	& \E_x\left[\left. \frac{1}{(\JJ_2 - \tt_2)^{\frac{3\delta p}{\epsilon}}}
	\right| \tt_1=y \right] = \int_0^\infty \P_x\left( \left.
	\JJ_2 - \tt_2 \le \gamma^{-\frac{\epsilon}{3\delta p}} \right| \tt_1=y \right) \,
	\dd \gamma \\
	& \qquad \qquad \le A_\alpha \int_0^\infty 
	\min\{1,\gamma^{-(1-\alpha)\frac{\epsilon}{3\delta p}}\} \, \dd \gamma 
	=: C_2 < \infty \,, \qquad
	\text{if} \quad \delta < \frac{(1-\alpha)\epsilon}{3p} \,,
\end{split}
\end{equation}
having used \eqref{lemmaP2}. Analogously, using \eqref{eq:LUC1},
\begin{equation} \label{eq:C3}
\begin{split}
	& \E_x\left[\left. \frac{1}{(\tt_2 - \ss_2)^{\frac{3\delta p}{\epsilon}}}
	\right| \tt_1=y \right] \le B_\alpha \int_0^\infty 
	\min\{1,\gamma^{- \alpha\frac{\epsilon}{3\delta p}}\} \, \dd \gamma 
	=: C_3 < \infty \,, \qquad
	\text{if} \quad \delta < \frac{\alpha\epsilon}{3p} \,.
\end{split}
\end{equation}
In conclusion, given $\epsilon \in (0,1)$ and $p \ge 1$, if we fix
$\delta < \min\{\alpha, 1-\alpha\}\frac{\epsilon}{3p}$, 
by \eqref{eq:C1}-\eqref{eq:C2}-\eqref{eq:C3}
there are constants $C_1, C_2, C_3 < \infty$
(depending on $\epsilon, p$) such that for all $N\in\N$ and $k \ge 2$
\begin{equation} \label{eq:C123}
	\E\left[\left. r_{N}\left(\dtt_{k-1},\, \dss_k,\, \dtt_k \right)^{\frac{p}{\epsilon}}\, 
	\right\vert \, \ss_{k-1}\, \tt_{k-1}\right] \leq \left(C'_\delta\right)^{\frac{p}{\epsilon}}
	\left(C_1 \, C_2 \, C_3\right)^{1/3}
	=: C_{\epsilon,p} < \infty \,,
\end{equation}
which proves \eqref{eq:alw}. Relation \eqref{eq:alw2} is proved
with analogous (and simpler) estimates, using the second relation
in \eqref{eq:rNbound}.

\smallskip

Finally, we prove relations \eqref{eq:spe}-\eqref{eq:spe2},
exploiting the upper bound \eqref{eq:rNbound2}
in which we plug $y' = \dss_{k-1}$, $x = \dss_k$, $y = \dtt_k$
(recall that $\lceil y' \rceil = \JJ_{k-1}$ and $\lceil y \rceil = \JJ_{k}$).
We recall that, by the uniform convergence theorem of slowly
varying functions \cite[Theorem 1.2.1]{BGT89},
$\lim_{N\to\infty} L(Na)/L(Nb) = 1$ \emph{uniformly for $a,b$
in a compact subset of $(0,\infty)$}. It follows by \eqref{eq:rNbound2}
that for all $\eta > 0$ and for all
$\gamma, \tilde\gamma \in (0,1)$, $T \in (0,\infty)$ there
is $\hat N_0 = \hat N_0(\gamma, \tilde \gamma, \eta,T) < \infty$ such that
for all $N \ge \hat N_0$ and for $k\ge 2$
\begin{gather*}
	r_N\left( \dtt_{k-1}, \dss_k, \dtt_k \right) \le (1+\eta)^2 \\
	\text{on the event} \ \ \ \big\{ \JJ_{k-1} - \tt_{k-1} \ge \gamma \big\}
	\cap \big\{ \JJ_k- \tt_k \ge \tilde\gamma\,, \ 
	\tt_k - \ss_k \ge \tilde\gamma\,,  \ \ss_k - \tt_{k-1} \le T \big\} \,.
\end{gather*}
Consequently, on the event $\{\JJ_{k-1} - \tt_{k-1} \ge \gamma\}
= \{\tt_{k-1} \le \JJ_{k-1}-\gamma\}$ we can write
\begin{equation*}
\begin{split}
	& \E\left[\left. r_{N}\left(\dtt_{k-1},\, \dss_k,\, \dtt_k \right)^{\frac{p}{\epsilon}}\, 
	\right\vert \, \ss_{k-1}\, \tt_{k-1}\right] \\
	& \qquad \le
	(1+\eta)^{\frac{2p}{\epsilon}} +
	\E\left[\left. r_{N}\left(\dtt_{k-1},\, \dss_k,\, \dtt_k \right)^{\frac{p}{\epsilon}}\, 
	\ind_{\{\JJ_k - \tt_k \ge \tilde\gamma, \, \tt_k - \ss_k \ge \tilde\gamma, \,
	\ss_k - \tt_{k-1} \le T\}^c}
	\right\vert \, \ss_{k-1}\, \tt_{k-1}\right] \\
	& \qquad \le (1+\eta)^{\frac{2p}{\epsilon}} +
	\sqrt{ C_{\epsilon, 2p} \,
	\P_x \big( \{\JJ_2 - \tt_2 \ge \tilde\gamma, \, \tt_2 - \ss_2 \ge \tilde\gamma, \,
	\ss_2 - y \le T\}^c  \,\big|\,  \tt_{1} = y \big) } \,,
\end{split}
\end{equation*}
where in the last line we have applied Cauchy-Schwarz,
relation \eqref{eq:C123}
and the regenerative property, with $(x,y) = (\ss_{k-1},\tt_{k-1})$. Since
for $x, y \in [n,n+1)$ one has
$\tt_1 = \mathtt{g}_n(\btau)$ and $\ss_2 = \mathtt{d}_n(\btau)$, by \eqref{dlaw}
\begin{equation*}
	\P_x(\ss_2 - y > T \,|\, \tt_1 = y) = \P_x ( \mathtt{d}_n(\btau) > T+y
	\,|\, \mathtt{g}_n(\btau) = y)
	= \int_{y+T}^\infty \frac{(n-y)^\alpha}{(v-y)^{1+\alpha}}
	\, \dd v 
	\le \frac{1}{\alpha \, T^\alpha} \,,
\end{equation*}
because $n-y \le 1$. Applying relations \eqref{lemmaP2}-\eqref{eq:LUC1},
we have shown that for $N \ge \hat N_0$ and $k \ge 2$,
on the event $\{\tt_{k-1} \le \JJ_{k-1}-\gamma\}$ we have the estimate
\begin{equation} \label{eq:quasispe}
	\E\left[\left. r_{N}\left(\dtt_{k-1},\, \dss_k,\, \dtt_k \right)^{\frac{p}{\epsilon}}\, 
	\right\vert \, \ss_{k-1}\, \tt_{k-1}\right] \le
	(1+\eta)^{\frac{2p}{\epsilon}} +
	\sqrt{ C_{\epsilon, 2p} \, \left( A_\alpha \tilde\gamma^{1-\alpha}
	+ B_\alpha \tilde\gamma^{\alpha} + \alpha^{-1} \, T^{-\alpha}
	\right)} \,.
\end{equation}

We can finally fix $\eta$, $\tilde\gamma$ small enough
and $T$ large enough (depending only on $\epsilon$ and $p$)
so that the right hand side of \eqref{eq:quasispe}
is less than $2$. This proves relation \eqref{eq:spe},
for all $\epsilon \in (0,1)$, $p\ge 1$, $\gamma \in (0,1)$,
with $\tilde N_0(\epsilon, p, \gamma) 
:= \hat N_0(\gamma, \tilde \gamma, \eta, T)$.
Relation \eqref{eq:spe2} is proved similarly, using  \eqref{eq:rNbound20}.\qed

\subsection{Proof of Lemma~\ref{th:auxi}}
\label{sec:auxi}

We recall that the random variables $\rss_k, \rtt_k, \rmm$
in the numerator of \eqref{eq:RADONNIKODYM} refer to the rescaled renewal
process $\tau/N$, cf.\ Definition~\ref{def:Jstm}. By \eqref{assRP0}-\eqref{assRP1}, we can write
the numerator in \eqref{eq:RADONNIKODYM},
which we call $\mathtt{L}_{M}$,  as follows:
for $0 = x_1 \le  y_1 < x_2 \le y_2 < \ldots < x_M \le y_M < t$, with
$x_i, y_i \in \frac{1}{N}\N_0$,
\begin{equation} \label{eq:rendisc}
	\mathtt{L}_{M}=u\big(Ny_1\big) \left( \prod_{i=2}^M K\big(N(x_i - y_{i-1})\big) 
	\, u\big(N(y_i - x_i)\big) \right)
	\bar K\big(N(t-y_M)\big) \,,
\end{equation}
where we set $\bar K(\ell) := \sum_{n > \ell} K(n)$.
Analogously,  using repeatedly \eqref{gdlaw} and the regenerative property,
the denominator in \eqref{eq:RADONNIKODYM}, which we call $\mathtt{I}_{M}$, 
can be rewritten as
\begin{gather}\label{eq:int}
	\mathtt{I}_{M} := \idotsint\limits_{\substack{u_i \in [x_i, x_i + \frac{1}{N}], \, 2 \le i \le M \\
	v_i \in [y_i, y_i+\frac{1}{N}], \, 1 \le i \le M}}
	\frac{C_{\alpha}}{v_1^{1-\alpha}}
	\left( \prod_{i=2}^M 
	\frac{C_{\alpha} \, \ind_{\{u_i < v_i\}}}{(u_i-v_{i-1})^{1+\alpha} \, (v_i - u_i)^{1-\alpha}}
	 \right) \frac{1}{\alpha\, (t-v_M)^\alpha} \, 
	 \dd v_1 \, \dd u_2 \, \dd v_2 \cdots \dd u_M \, \dd v_M \,.
\end{gather}
Bounding uniformly 
\begin{equation}\label{eq:unifbound}
	u_i - v_{i-1} \le x_i - y_{i-1} + \tfrac{1}{N}, \qquad
	v_i - u_i \le y_i - x_i + \tfrac{1}{N}, \qquad
	t-v_M \le t-y_M+\tfrac{1}{N} \,,
\end{equation}
we obtain a lower bound for $\mathtt{I}_{M}$ which is factorized
as a product over blocks:
\begin{equation} \label{eq:LBint}
\begin{split}
	& \frac{1}{N^{2M-1}} \,
	\frac{C_{\alpha}}{(x_1 + \frac{1}{N})^{1-\alpha}}
	\left( \prod_{i=2}^M 
	\frac{C_{\alpha}}{(x_i-y_{i-1}+\frac{1}{N})^{1+\alpha} \, 
	(y_i - x_i + \frac{1}{N})^{1-\alpha}}
	 \right) \frac{1}{\alpha\, (t-y_M)^\alpha} \\
	 & \quad \ = 
	 \frac{C_{\alpha}}{(Nx_1 + 1)^{1-\alpha}}
	\left( \prod_{i=2}^M 
	\frac{C_{\alpha}}{(N(x_i-y_{i-1})+1)^{1+\alpha} \, 
	(N(y_i - x_i)+1)^{1-\alpha}}
	 \right) \frac{1}{\alpha\,(N(t-y_M)+1)^\alpha} \,.
\end{split}
\end{equation}
Looking back at \eqref{eq:rendisc} and
recalling \eqref{eq:RADONNIKODYM}, 
 it follows that relation \eqref{eq:RNfact} holds with
\begin{align*} 
	r_N(y',x,y) & := (N(x-y')+1)^{1+\alpha} \, K\big(N(x - y')\big) \,
	\frac{(N(y - x)+1)^{1-\alpha}}{C_\alpha} \,
	u\big(N(y - x)\big) \,
	\left\{ \frac{L\big(N(\lceil y\rceil - y)\big)}{L\big(N(\lceil y'\rceil - y')\big)}
	\right\} \,, \\
	\tilde r_N(z,t) & := \alpha \, (N(t-z)+1)^{\alpha} \, 
	\frac{\bar K\big(N(t - z)\big)}{L\big(N(\lceil z\rceil - z)\big)} \,,
\end{align*}
where we have ``artificially'' added the last terms inside the brackets, which
get simplified telescopically when one considers the product in \eqref{eq:RNfact}.
(In order to define $r_N(y',x,y)$ also when $y' = x = 0$, which is necessary
for the first term in the product in \eqref{eq:RNfact},
we agree that $K(0) := 1$.)

Recalling \eqref{assRP0} and \eqref{assRP1}, there is some constant 
$\mathtt{C} \in (1, \infty)$ such that for all $n\in\N_0$
\begin{equation} \label{eq:basicest}
	K(n) \le \mathtt{C} \frac{L(n)}{(n+1)^{1+\alpha}} \,, \qquad 
	\bar K(n) \le \mathtt{C} \frac{L(n)}{\alpha (n+1)^{\alpha}} \,, \qquad
	u(n) \le \mathtt{C} \frac{C_\alpha}{L(n) \, (n+1)^{1+\alpha}} \,.
\end{equation}
Plugging these estimates into the definitions of $r_N$, $\tilde r_N$
yields the first and second relations in \eqref{eq:rNbound}, with $C = \mathtt{C}^2$
and $C = \mathtt{C}$, respectively.
Finally, given $\eta > 0$ there is $M_0 = M_0(\eta) < \infty$ such that
for $n \ge M_0$ one can replace $\mathtt{C}$ by $(1+\eta)$ in \eqref{eq:basicest},
which yields \eqref{eq:rNbound20} and \eqref{eq:rNbound2}.\qed

\begin{remark}\rm
To prove $f^{(1)}\prec f^{(2)}$ we have shown that it is possible to give an upper bound, 
cf. \eqref{eq:RNfact}, for the Radon-Nikodym derivative $\RN_t$ by suitable functions 
$r_N$ and $\tilde r_N$ satisfying Lemma~\ref{lemma:lemmone1}.
Analogously, to prove the complementary step $f^{(3)}\prec f^{(2)}$,
that we do not detail, one would need an analogous upper bound for 
the \emph{inverse} of the Radon-Nikodym derivative, i.e.
\begin{equation} \label{eq:RNfact1}
	\RN_t\left(M, (x_k,y_k)_{k=1}^{M}\right)^{-1} \le
	\left\{\prod_{\ell=1}^M q_N(y_{\ell-1}, x_\ell, y_\ell) \right\}
	{ \tilde q_{N}(y_M, t)} \,,
\end{equation}
for suitable functions $q_N$ and $\tilde q_{N}$ that
satisfy conditions similar to $r_N$ and $\tilde r_{N}$ in 
Lemma~\ref{th:auxi}, thus yielding an analogue of Lemma~\ref{lemma:lemmone1}. 
To this purpose, 
we need to show that the multiple integral $\mathtt{I}_M$ admits an upper bound given by 
a suitable factorization, analogous to \eqref{eq:LBint}. The natural idea is to use uniform bounds that 
are complementary to \eqref{eq:unifbound},  i.e.\
$u_i - v_{i-1} \ge x_i - y_{i-1} - \frac{1}{N}$ etc., which  
work when the distances like $x_i - y_{i-1}$ are at least  $\frac{2}{N}$. When some of such distances 
is $0$ or  $\frac{1}{N}$, the integral must be estimated by hands. This is based on routine 
computations, for which we refer to \cite{cf:T15}.
\end{remark}


\section{Miscellanea}
\label{sec:misc}

\subsection{Proof of Lemma~\ref{th:improve}}
\label{sec:improve}

We start with the second part: assuming \eqref{eq:distconc},
we show that \eqref{assD2} holds. Given $n\in\N$ and a convex
$1$-Lipschitz function $f:\R^n \to \R$,
the set $A := \{\omega\in\R^n: \ f(\omega) \le a\}$ is convex, for all $a\in\R$,
and $\{f(\omega) \ge a+t\} \subseteq \{d(\omega,A) \ge t\}$, because $f$ is $1$-Lipschitz.
Then by \eqref{eq:distconc}
\begin{equation} \label{eq:apply}
	\bbP(f(\omega) \le a) \, \bbP(f(\omega) \ge a+t) \le
	\bbP(\omega \in A) \bbP(d(\omega,A) \ge t)
	\le C_1' \exp \left(-\frac{t^\gamma}{C_2'}\right) \,.
\end{equation}
Let $M_f \in \R$ be a median for $f(\omega)$,
i.e.\ $\bbP(f(\omega) \ge M_f) \ge \frac{1}{2}$ and $\bbP(f(\omega) \le M_f) \ge \frac{1}{2}$.
Applying \eqref{eq:apply} for $a=M_f$ and $a=M_f-t$ yields
\begin{equation*}
	\bbP\Big( \big|f(\omega) - M_f \big| \ge t \Big) \le 4 \, C_1' 
	\exp \left(-\frac{t^\gamma}{C_2'}\right) \,,
\end{equation*}
which is precisely our goal \eqref{assD2}.

\smallskip

Next we assume \eqref{assD2} and we show that \eqref{eq:distconc} holds.
We actually prove a stronger statement: for any $\eta \in (0,\infty)$
\begin{equation}\label{eq:bet}
	\bbP(\omega \in A)^\eta \, \bbP(d(\omega, A) > t) \le
	C_1^{1+\eta} \, \exp \bigg(-\frac{\epsilon_\eta \, t^\gamma}{C_2} \bigg) \,,
	\quad \text{with} \quad
	\epsilon_\eta := 
 \frac{\eta}{(1+\eta^{\frac{1}{\gamma-1}})^{\gamma-1}}.
\end{equation}
In particular, choosing $\eta = 1$, \eqref{eq:distconc} holds
with $C_1' := C_1^2$ and $C_2' = 2^{(\gamma-1)^+} C_2$.\qed

If $A$ is convex, the function $f(x) := d(x,A)$ is convex, $1$-Lipschitz and also $M_f \geq 0$, hence
by \eqref{assD2}
\begin{align} \label{eq:use}
	\bbP(\omega \in A) & =
	\bbP(f(\omega) \le 0)
	\le \bbP(|f(\omega) - M_f| \ge M_f) \le C_1 \exp \bigg(-\frac{M_f^\gamma}{C_2}\bigg) \,, \\
	\label{eq:use2}
	\bbP(d(\omega, A) > t)
	& \le \bbP(|f(\omega) - M_f| > t-M_f)
	\le C_1 \exp \bigg(-\frac{(t-M_f)^\gamma}{C_2}\bigg) \,,
	\quad \forall t \ge M_f \,,
\end{align}
hence for every $\eta \in (0,\infty)$ we obtain
\begin{equation}\label{eq:use3}
	\bbP(\omega \in A)^\eta \, \bbP(d(\omega, A) > t) \le
	C_1^{1+\eta} \, \exp \bigg(-\frac{1}{C_2}
	\big(\eta\, M_f^\gamma +  (t-M_f)^\gamma \big)\bigg) \,,
	\quad \forall t \ge M_f \,.
\end{equation}

The function $m \mapsto \eta \, m^\gamma +  (t-m)^\gamma$ is convex and,
by direct computation, it attains its minimum in the interval $[0, t]$.
at the point $m = \bar m := t / (1+\eta^{1/(\gamma-1)})$. Replacing $M_f$ by
$\bar m$ in \eqref{eq:use3} yields
precisely \eqref{eq:bet} for all $t \ge M_f$. 

It remains to prove \eqref{eq:bet} for $t \in [0,M_f)$. This follows by \eqref{eq:use}:
\begin{equation*}
	\bbP(\omega \in A)^\eta \, \bbP(d(\omega, A) > t) \le
	\bbP(\omega \in A)^\eta \le C_1^\eta \exp \bigg(-\frac{\eta \, M_f^\gamma}{C_2}\bigg)
	\le C_1^{1+\eta} \, \exp \bigg(-\frac{\epsilon_\eta \, t^\gamma}{C_2} \bigg)
	\quad \ \text{for } \ t \le M_f \,,	
\end{equation*}
where the last inequality holds because $\eta \ge \epsilon_\eta$ (by \eqref{eq:bet}) and $C_1 \ge 1$
(by \eqref{assD2}, for $t=0$).\qed

\subsection{Proof of Proposition~\ref{th:conve}}
\label{sec:conve}
By convexity, $f(\omega) - f(\omega') \le \langle \nabla f(\omega), \omega-\omega'\rangle
\le |\nabla f(\omega)| \, |\omega-\omega'|$  for all $\omega, \omega' \in \R^n$, 
where $\langle \cdot, \cdot \rangle$
is the usual scalar product in $\R^n$. Defining the convex set $A := \{\omega \in \R^n:
\ f(\omega) \le a-t\}$, we get
\begin{equation*}
	f(\omega) \le a-t + |\nabla f(\omega)| \, |\omega-\omega'| \,, \qquad
	\forall \omega \in \R^n, \ \forall \omega' \in A \,,
\end{equation*}
hence $f(\omega) \le a-t + |\nabla f(\omega)| \, d(\omega, A)$ for all 
$\omega \in \R^n$. Consequently, by inclusion of events and \eqref{eq:distconc},
\begin{equation*}
	\bbP(f(\omega) \ge a, \, |\nabla f(\omega)| \le c) \le
	\bbP(d(\omega, A) \ge t/c) \le
	\frac{C_1'}{\bbP(\omega \in A)} \,	\exp \left(-\frac{(t/c)^\gamma}{C_2'}\right) \,.
\end{equation*}
Since $\bbP(\omega \in A) = \bbP(f(\omega) \le a-t)$
by definition of $A$, we have proved \eqref{eq:keycon}.\qed

\section*{Acknowledgements}

We thank Rongfeng Sun and Nikos Zygouras for fruitful discussions.

\bibliographystyle{amsplain}
\bibliography{biblioT.bib}

\end{document}